\documentclass{amsart}
\usepackage{latexsym}
\usepackage{amsmath}
\usepackage{amssymb}
\usepackage[all]{xy}
\usepackage{comment}

\usepackage{autobreak}
\usepackage{enumitem}
\usepackage[hypertexnames=false]{hyperref} % the option "hypertexnames=false" is added for autonum
\usepackage[capitalize]{cleveref}
\crefname{equation}{}{}
\crefname{enumi}{}{}

\usepackage[\ifdefined\enabletodonotes\else disable\fi]{todonotes}

\usepackage{autonum} % must be after hyperref and showkeys

\usepackage{tikz}
\usetikzlibrary{positioning}
\usetikzlibrary{arrows.meta}

\numberwithin{equation}{section}
\newtheorem{thm}{Theorem}[section]
\newtheorem{prop}[thm]{Proposition}
\newtheorem{cor}[thm]{Corollary}
\newtheorem{lem}[thm]{Lemma}

\theoremstyle{definition}
\newtheorem{defn}[thm]{Definition}

\newtheorem{assertion}[thm]{Assertion}
\theoremstyle{remark}
\newtheorem{rem}[thm]{Remark}
\newtheorem{ex}[thm]{Example}

\newcommand{\Q}{{\mathbb Q}}
\newcommand{\Z}{{\mathbb Z}}
\newcommand{\C}{{\mathcal C}}
\newcommand{\e}{\varepsilon}
\newcommand{\R}{{\mathbb R}}
\newcommand{\mapright}[1]{%
 \smash{\mathop{%
  \hbox to 1cm{\rightarrowfill}}\limits_{#1}}}
\newcommand{\maprightd}[2]{%
 \smash{\mathop{%
  \hbox to 0.5cm{\rightarrowfill}}\limits^{#1}\limits_{#2}}}
\newcommand{\mapleft}[1]{%
 \smash{\mathop{%
  \hbox to 1cm{\leftarrowfill}}\limits_{#1}}}
\newcommand{\mapleftu}[1]{%
 \smash{\mathop{%
  \hbox to 0.8cm{\leftarrowfill}}\limits^{#1}}}
\newcommand{\maprightu}[1]{%
 \smash{\mathop{%
  \hbox to 1cm{\rightarrowfill}}\limits^{#1}}}
\newcommand{\maprightud}[2]{%
 \smash{\mathop{%
  \hbox to 1cm{\rightarrowfill}}\limits^{#1}_{#2}}}
\newcommand{\mapleftud}[2]{%
 \smash{\mathop{%
  \hbox to 1cm{\leftarrowfill}}\limits^{#1}_{#2}}}

\newcommand{\minmodel}{\mathfrak{M}}
\newcommand{\modelcat}{\mathcal{M}}
\newcommand{\et}[1]{\mathsf{et}({#1}^{\R_+})}

\author[K. Kuribayashi]{Katsuhiko Kuribayashi}
\address{%
  Department of Mathematics,
  Faculty of Science,
  Shinshu University,
  Matsumoto, Nagano 390-8621, Japan
}
\email{kuri@math.shinshu-u.ac.jp}
\author[T. Naito]{Takahito Naito}
\address{%
  Nippon Institute of Technology,
  Gakuendai, Miyashiro-machi, Minamisaitama-gun, Saitama 345-8501, Japan
}
\email{naito.takahito@nit.ac.jp}

\author[K. Sekizuka]{Kengo Sekizuka}
\address{%
Department of Science and Technology,
Graduate School of Medicine, Science and Technology,
Shinshu University,
Matsumoto, Nagano 390-8621, Japan
}
\email{23hs601d@shinshu-u.ac.jp}

\author[S. Wakatsuki]{Shun Wakatsuki}
\address{%
  Graduate School of Mathematics, Nagoya University,
  Furo-cho, Chikusa-ku, Nagoya, Aichi 464-8601, Japan
}
\email{shun.wakatsuki@math.nagoya-u.ac.jp}

\author[T. Yamaguchi]{Toshihiro Yamaguchi}
\address{%
  Faculty of Education,
  Kochi University, Akebono-cho, Kochi 780-8520, Japan
}
\email{tyamag@kochi-u.ac.jp}

\begin{document}
\title[A distance between maps via interleavings of Sullivan algebras]{A distance between maps via interleavings of relative Sullivan algebras
%--Draft--
}

\footnote[0]{{\it 2020 Mathematics Subject Classification}: Primary 55N31, 55P62; Secondary 55U35.

{\it Key words and phrases.} Interleaving distance, extended tame object, relative Sullivan algebra, Moore--Postnikov tower.

%This research was partially supported by a Grant-in-Aid for Scientific
%Research (B) 25287008
%from Japan Society for the Promotion of Science.

%Department of Mathematical Sciences,
%Faculty of Science,
%Shinshu University,
%Matsumoto, Nagano 390-8621, Japan
%e-mail:{\tt kuri@math.shinshu-u.ac.jp}
}

\begin{abstract}
In this article, we consider extended tame persistence commutative differential graded algebras (CDGAs) associated with relative Sullivan algebras.
In particular, if the relative Sullivan algebra is a model for a map between spaces, then the persistence CDGA is isomorphic to the persistence object obtained by a Postnikov tower for the map with the polynomial de Rham functor in the homotopy category of extended tame persistence CDGAs.

Moreover,
the interleaving distance in the homotopy category (IHC) in the sense of Lanari and Scoccola enables us
to introduce a pseudodistance on the homotopy set of maps via the persistence CDGA models for maps. In contrast to persistence cochain complexes, the IHC of persistence CDGAs does not coincide with  the cohomology interleaving distance in general. Due to the reason, we also discuss formalities of a persistence CDGA with interleavings.
Computational examples of the pseudodistances between maps are showcased.
\end{abstract}

\maketitle

\tableofcontents

\section{Introduction}
Recently, fascinating persistence invariants appear in algebra, geometry and topology which are obtained by swing-backs from topological data analysis; see, for example, \cite{Bl-L, BOZ, CGL, HLM, M-Z, Zhou}. In particular, Zhou \cite{Zhou} studies persistent Sullivan minimal models for persistent spaces; see Section \ref{sect:multi-p}. 

In \cite{Bl-L}, the {\it interleavings} for persistence objects with values in a model category are considered.
Being based on the consideration due to Blumberg and Lesnick \cite{Bl-L}, Lanari and Scoccola \cite{LS} introduce
the {\it interleaving distance in the homotopy category} $d_{\text{IHC}}$ for persistence objects. In the considerations of $d_{\text{IHC}}$ and other distances, which are essentially given in \cite{Bl-L}, Lanari and Scoccola use the projective model structure of the category of persistence objects with values in a cofibrantly generated model category.
In \cite{CGL}, Chach\'olski, Giunti and Landi have given the category of {\it tame} persistence objects a model structure.

In this article, by applying the distance $d_{\text{IHC}}$ to tame persistence objects, we construct a two parameter homotopy invariant for continuous maps; see Sections \ref{sect:ID} and \ref{sect:pCDGAs}. More precisely, the interleaving distance between maps is introduced and investigated by making use of the model structure on the category $\et{\mathsf{CDGA}}$ of {\it extended} tame persistence commutative differential graded algebras (CDGAs) over the rational field $\Q$.
The key to the discussion is the functor $\Theta$ constructed in Theorem \ref{thm:main}, which assigns an object in the homotopy category of
$\et{\mathsf{CDGA}}$ to a map between path-connected spaces via a relative Sullivan model for the map. 

The functor $\Theta$ is indeed defined algebraically with the dimension of the extended part of a relative Sullivan algebra for a map. However,
we remark that the extended tame persistence CDGA $\Theta(f)$ associated with a map $f$ between simply-connected spaces is isomorphic to the image of the Moore--Postnikov tower of $f$ by the polynomial de Rham functor in the homotopy category of $\et{\mathsf{CDGA}}$; see Theorem \ref{thm:main_II}.
Thus, by formulating the result with a partial Quillen equivalence in the sense of Moreno-Fern\'andez and B. Stonek \cite{M-S},
we see that the distance $d_{\text{IHC}}$ between $\Theta(f)$ and $\Theta(g)$ is equal to
the distance $d_{\text{IHC}}$ between the pointwise rationalizations of the Moore--Postnikov towers of $f$ and $g$ in the category of extended tame {\it copersistence} simplicial sets; see Corollary $\ref{cor:Alg_vs_Space}$.
Furthermore, we have a pseudodistance
%$
\[
d_{\text{H}} : [X, Y] \times [X, Y] \to \R_{\geq 0}\cup \{\infty\}
\]
%$
on the homotopy set by making use of the functor $\Theta$; see Theorem \ref{thm:Ho_sets}.
%Section \ref{sect:GeometricRealizations}.

The result \cite[Theorem 3.3]{KNWY} yields 
that the distance $d_{\text{IHC}}$ coincides with
the {\it cohomology interleaving distance} $d_{\text{CohI}}$ in the category of persistence cochain complexes over a field even though the homology functor induces the inequality $d_{\text{CohI}} \leq d_{\text{IHC}}$ in general.
On the other hand, we have examples which give the strict inequality $d_{\text{CohI}} < d_{\text{IHC}}$
if the two distances are dealt with in the category $\et{\mathsf{CDGA}}$.
Indeed, the Hopf map $S^3 \to S^2$ and the trivial map with the same domain and codomain is such an example.
Another example is given by considering the distance $d_{\text{H}}$ between appropriate two elements in the homotopy group
$\pi_3 \left( {\mathbb C} P^2 \# \overline{{\mathbb C} P^2} \right)\otimes \Q$; see Section \ref{section:sensitiveOne}. We stress that the strict inequality depends on the choice of the underlying field. In fact, in the second example above, the two distances between the elements coincide if we choose the complex field \(\mathbb{C}\) in the consideration instead of $\Q$; see Remark \ref{rem:connectedsumCP2_overC}.

A key to proving the equality $d_{\text{CohI}} = d_{\text{IHC}}$ in \cite[Theorem 3.3]{KNWY} is that every persistence cochain complex is $d_{\text{HI}}$-{\it formal}; see \cite[Proposition 3.5]{KNWY}. In general, the formality is defined by using the {\it homotopy interleaving distance} on a model category; see \cite{Bl-L, LS}. Thus, we are interested in considering formalities for persistence CDGAs. The topic is discussed in Section \ref{section:formalities} relating to the notion of formalizability of a map in the sense of Thomas \cite{Thomas};
see Propositions \ref{prop:implications_formalities} and \ref{prop:d_{IHC}_formalizable}.

It is worthwhile mentioning that a transferred model structure on the category of persistence CDGAs, which are not necessarily tame, is introduced in \cite{HLM} by using the {\it interval sphere} model structure on the category of persistence cochain complexes over $\Q$.  In this article, our interest is restricted to extended tame persistence CDGAs and their invariants and then we invoke the model structure on $\et{\mathsf{CDGA}}$ introduced in \cite{CGL}.  The novelty is that we assign a persistence CDGA to a single continuous map.

%\subsection{The organization of this manuscript}
The rest of this manuscript is arranged as follows. Section \ref{sect:ID} recalls the interleavings up to homotopy.
Moreover, interleaving distances introduced in \cite{Bl-L, LS} are reconsidered in the category of extended tame persistence objects.
To this end, we use the model structure of the category of tame persistence objects with values in a model category, which is due to
Chach\'olski, Giunti and Landi \cite{CGL}.
In Section \ref{sect:pCDGAs}, after recalling the notion of a relative Sullivan algebra, we introduce
the functor $\Theta$ mentioned above and investigate its properties.
In Section \ref{sect:MPs-P_CDGAs}, we relate the persistence CDGA associated with a map $f$ via the functor $\Theta$
to the Moore--Postnikov tower for $f$.
%Geometric realization of $\Theta(f)$ for each map $f$ is discussed in \ref{sect:GeometricRealizations}.
Section \ref{sect:Interleavings_maps} considers the interleaving distance in the homotopy category for maps.
Section \ref{section:sensitiveOne} is devoted to producing examples of persistence CDGAs associated with maps for each of which
the distance $d_{\text{IHC}}$ is greater than $d_{\text{CohI}}$. In Section \ref{section:formalities}, we show that formalities of a persistence CDGA defined by interleaving distances are equivalent to one another.  Moreover, the formalities are related to the formalizability of a map.
Section \ref{sect:Examples} gives more computational examples of the distances $d_{\rm IHC}$ of maps.
Section \ref{sect:Perspective} describes perspective of our work.

\section{The interleaving distances between extended tame objects}\label{sect:ID}

Let $\C$ be a category and $\mathcal{C}^{(\R, \leq)}$ the functor category, where $(\R, \leq)$ is the poset defined with the usual order which is regarded as a category.  Originally, the interleavings {\it up to homotopy} are defined in the category $\modelcat^{(\R, \leq)}$ endowed with the projective model structure
for a cofibrantly generated model category $\modelcat$. In order to define interleavings up to homotopy in a full subcategory of $\modelcat^{(\R, \leq)}$, we need to reconsider results in \cite{LS}.

We begin by recalling strict interleavings in the functor category $\mathcal{C}^{(\R, \leq)}$ for a general category $\mathcal{C}$.
For a real number $\e \geq 0$, define a functor $T_\e : (\R, \leq) \to (\R, \leq)$ by $T_\e(a) = a+ \e$.
%Moreover, we define a natural transformation $\eta_\e : id_{(\R, \leq)} \Rightarrow T_\e$ by $\eta_\e(a) : a \leq a + \e$.
The {\it $\e$-shift functor} $( \ )^\e :  \C^{(\R, \leq)} \to \C^{(\R, \leq)}$ is defined by $( \ )^\e (F) = F^\e := FT_\e$.
%Then, we see that$(F, G, \varphi, \psi)$ is an $\e$-interleaving if and only if

\begin{defn}\label{defn:interleaving}(\cite[Definition 4.2]{CCGGO}, \cite[Definition 3.1]{B-S})  Objects $F$ and $G$ in $\C^{(\R, \leq)}$ are {\it $\e$-interleaved} if there exists a commutative diagram
\begin{eqnarray}\label{eq:interleaving_1}
\xymatrix@C35pt@R25pt{
F \ar[r]  \ar[dr]^(0.7){\varphi} &  F^{\e}  \ar[r] \ar[rd]^(0.7){\varphi^\e}|\hole & F ^{2\e}   \\
G  \ar[r]  \ar[ur]^(0.3){\psi}|\hole& G^\e \ar[ur]^(0.3){\psi^\e} \ar[r]  & G^{2\e}
}
\end{eqnarray}
in which horizontal arrows are the natural transformations defined by the structure maps of $F$ and $G$. The pair  $(\varphi, \psi)$ of the natural transformations is called an $\e$-{\it interleaving} between $F$ and $G$.
\end{defn}

\begin{rem}
The commutative diagram in Definition \ref{defn:interleaving} yields the commutativity of the diagrams
\begin{eqnarray}\label{eq:interleaving}
\xymatrix@C10pt@R15pt{
&&F(i) \ar[rr]^-{F(i \leq i+2\e)} \ar[dr]_-{\varphi(i)} &  & F(i+2\e)   &\text{and}& & F(i +\e) \ar[rd]^{\varphi(i+\e)}& \\
&&& G(i+\e) \ar[ur]_{\psi(i+\e)} &                                               &&G(i) \ar[rr]_-{G(i \leq i+2\e)} \ar[ru]^-{\psi(i) }&& G(i+2\e)
}
\end{eqnarray}
for all $i \in \R$. We note that $F$ is isomorphic to $G$ in $\C^{(\R, \leq)}$ if and only if $F$ and $G$ are $0$-interleaved.
\end{rem}

\begin{defn}
For objects $F$ and $G$ in $\C^{(\R, \leq)}$, the interleaving distance $d_{\text{I}}(F, G)$ between $F$ and $G$ is defined by
\[
d_{\text{I}}(F, G) := \text{inf}(\{ \e \geq 0 \mid \text{$F$ and $G$ are $\e$-interleaved} \}\cup \{\infty\}).
\]
%We set that $d_{\text{I}}(F, G) = \infty$ if $F$ and $G$ are not $\e$-interleaved for any $\e\geq 0$.
\end{defn}

\begin{defn}\label{defn:et}  
Let $\C$ be a category and $\R_+$ the full subcategory of $(\R, \leq)$ whose objects are non-negative real numbers.
A sequence $\tau_0 < \tau_1 < \cdots  <\tau_n<\cdots$ in $[0,\infty)$, which is divergent or finite,
{\it discretises} a functor $X:\R_+\to \C$ if $X({s\leq t}):X(s)\to X(t)$ may fail to be an isomorphism only if there is $a\in \mathbb{N}$ such that $s<\tau_a\leq t$. A functor $X : \R_+\to \C$ is called an {\it extended tame functor} if there is a sequence that discretises it. Let  $\et{\C}$ denote the full subcategory of the functor category ${\C}^{\R_+}$ whose objects are extended tame functors. % and whose morphisms are all of the natural transformations.
An object $X$ in $\et{\C}$ is called {\it tame} if the sequence which discretises $X$ is finite; see \cite{CGL}.
\end{defn}

Following \cite[Section 2.2]{CGL}, we introduce a factorization of a morphism $g : X\to Y$ in $\et{\C}$. Let  $\tau_0 < \tau_1 < \cdots  <\tau_n<\cdots $ be a sequence discretising both $X$ and $Y$. By induction on
$\{0, 1, 2, ..., n, ... \}$, we define morphisms
$\bar{g}({\tau_a}) : X({\tau_a}) \to Q({\tau_a})$ and $\hat{g}({\tau_a}) : Q({\tau_a}) \to Y({\tau_a})$ in $\C$ as follows:
%
%\begin{center}
$(\bar{g}(0) : X(0) \to Q(0)):= (1 : X(0) \to X(0))$ \ and \  $(\hat{g}(0) : Q(0) \to Y(0)):=(g(0) : X(0)\to Y(0))$.
%\end{center}
For $a > 0$, the object $Q({\tau_a})$ is defined by $\text{colim}(\xymatrix@C35pt@R20pt{Y({\tau_{a-1}}) & X({\tau_{a-1}}) \ar[l]_-{g({\tau_{a-1}})} \ar[r]^-{X({\tau_{a-1} < \tau_a})} & X({\tau_a}) })$ with $\bar{g}({\tau_a}) : X({\tau_a}) \to Q({\tau_a})$ and $\hat{g}({\tau_a}) : Q({\tau_a}) \to Y({\tau_a})$ which fit in the commutative diagram
\begin{eqnarray}\label{diagram:the_pushout}
\xymatrix@C40pt@R15pt{
 X({\tau_{a-1}}) \ar[r]^{X({\tau_{a-1} < \tau_a})} \ar[d]_{g({\tau_{a-1}})}& X({\tau_a}) \ar[d]^{\bar{g}({\tau_a})}  \ar@/^0.8pc/[rdd]^{g({\tau_a})} &\\
 Y({\tau_{a-1}}) \ar@/_0.8pc/[rrd]_{Y({\tau_{a-1} < \tau_a})}\ar[r] & Q({\tau_a}) \ar[rd]^{\hat{g}({\tau_a})} & \\ % \ar@{.>}@/_2.0pc/
  & & Y({\tau_a})
}
\end{eqnarray}
consisting of the pushout square.
For $a > 0$, define $Q({\tau_{a-1}< \tau_a}) : Q({\tau_{a-1}}) \to Q({\tau_a})$ to be the composite of the morphism $Y({\tau_{a-1}}) \to Q({\tau_a})$ represented by the bottom horizontal arrow in (\ref{diagram:the_pushout}) and $\hat{g}({\tau_{a-1}}) : Q({\tau_{a-1}})  \to Y({\tau_{a-1}})$.
The construction above gives an extended persistence object $Q$ via the left Kan extension along the inclusion $\{\tau_0 < \cdots < \tau_n <\cdots \} \to \R$ of poset\footnote{We observe that the functor $(\lfloor \ \rfloor)^*$ mentioned in Section \ref{sect:relativeSullivan_pCDGA} below is nothing but the left Kan extension along the inclusion
$j : \mathbb{Z}\to \mathbb{R}$.}; see \cite[Section 2.1]{CGL}.
Thus, we have a factorization $g = \hat{g}\bar{g}$ of $g$.

Throughout this manuscript, we use the same terminology as that in \cite{DS}  for model categories.
The following two results are due to Chach\'olski, Giunti and Landi \cite{CGL}.

\begin{thm}\label{thm:ModelCat}\cite[Theorem 2.2]{CGL}
    Let $\modelcat$ be a model category. The following choices of weak equivalences, fibrations and cofibrations form a model structure on $\et{\modelcat}$. A morphism $g:X\to Y$ in $\et{\modelcat}$ is a
    \begin{itemize}
        \item weak equivalence if $g(t):X(t)\to Y(t)$ is a weak equivalence for all t,
        \item fibration if $g(t):X(t)\to Y(t)$ is a fibration for all t,
        \item cofibration if $\hat{g}(t): Q(t)\to Y(t)$ is a cofibration for all t.
    \end{itemize}
\end{thm}

In the definition above, it is not necessarily assumed that the model category admits functorial factorizations. 

%The following proposition characterizes cofibrant objects in $\et{\modelcat}$.
\begin{prop}\label{prop:cofibrantOb}\cite[Proposition 2.3]{CGL}
Let $\modelcat$ be a model category. \\
{\rm (i)} If $g:X\to Y$ is a cofibration in $\et{\modelcat}$, then $g(t):X(t) \to Y(t)$ is a cofibration in $\modelcat$ for any $t$ in $\mathbb{R}_+$. \\
{\rm (ii)} An object $X$ in $\et{\modelcat}$ is cofibrant if and only if $X(0)$ is cofibrant and, for any $s < t$ in $\R_+$, the transition morphism $X({s<t}) : X(s) \to X(t)$ is a cofibration in $\modelcat$.
\end{prop}

Originally, the results \cite[Theorem 2.2 and Proposition 2.3]{CGL} are proved for the full subcategory of $\et\modelcat$ consisting of tame objects.
The induction arguments in the original proofs of \cite[Theorem 2.2 and Proposition 2.3]{CGL} are valid to obtain Theorem \ref{thm:ModelCat} and Proposition \ref{prop:cofibrantOb}. For an extended tame object $F$, a sequence $\{\tau_n\}_{n\geq 0}$ which discretises $F$ is unbounded if the sequence is not finite. Then, in particular, we may construct a cofibrant replacement of $F$ by taking inductively a cofibrant replacement of each structure map $F({\tau_n}) \to F({\tau_{n+1}})$. Here, we use  the condition that the sequence is divergent.

%In what follows, we write $F(i)$ and $F(i < j)$ for the component $F^i$ of a persistence object $F$ at $i$ and for the translation given by the morphism $i < j$, respectively.

\begin{rem} 
The model structure in Proposition \ref{prop:cofibrantOb} is closely related to that of the category of towers in a category described in
\cite[Section VI, Definition 1.1]{GJ}; see Remark \ref{rem:CDGA_sSet} (iii).
\end{rem}

Let $\modelcat$ be a model category.
%The interleaving distance in the homotopy category is defined on the homotopy category
%$\text{Ho}(\modelcat^{(\R, \leq)})$ of the model category $\modelcat^{(\R, \leq)}$ with the projective model structure.
Since the $\delta$-shift functor $(\ )^\delta : \et{\modelcat} \to \et{\modelcat}$ preserves weak equivalences, it follows that the functor
$(\ )^\delta$ induces the self functor  $(\ )^\delta$ on the homotopy category $\text{Ho}(\et{\modelcat})$. Then, we can consider
the commutative diagram (\ref{eq:interleaving_1}) in $\text{Ho}( \et{\modelcat})$. Moreover, we say that objects $F$ and $G$ in $\text{Ho}( \et{\modelcat})$ are $\e$-{\it interleaved in the homotopy category} if they are $\e$-interleaved in the sense in Definition \ref{defn:interleaving}; see \cite[Section 2.2.2]{LS}.
The {\it interleaving distance in the homotopy category} between objects $F$ and $G$ in $\text{Ho}( \et{\modelcat})$ is defined by
\[
d_{\text{IHC}}(F, G) \! :=\! \text{inf}(\{ \e \geq 0 \mid \text{$F$, \!$G$ are $\e$-interleaved in the homotopy category} \}\cup \{\infty\}).
% \ \text{and}
\]

For objects $X$ and $Y$ in $\et{\modelcat}$, we say that $X$ and $Y$ are $\e$-{\it homotopy interleaved} if there exist $X\simeq X'$ and $Y\simeq Y'$ such that $X'$ and $Y'$ are $\e$-interleaved in $\modelcat^{(\R, \leq)}$; see \cite[Section 3.3]{Bl-L}. Here
$W\simeq W'$ means that there is a zigzag of weak equivalences connecting $W$ and $W'$.

Let $j : \et{\modelcat} \to \modelcat^{(\R, \leq)}$ be  the inclusion functor  and $q_* : \modelcat^{(\R, \leq)} \to \text{Ho}(\modelcat)^{(\R, \leq)}$ the  the functor induced by the localization functor $q : \modelcat \to \text{Ho}(\modelcat)$.
We say that $X$ and $Y$ in  $\et{\modelcat}$  are $\e$-{\it homotopy commutative interleaved} if $q_*jX$ and $q_*jY$ are
$\e$-interleaved in $\text{Ho}(\modelcat)^{(\R, \leq)}$.%\todo{[RR2 4)] period}
%Let $h :  \modelcat^{(\R, \leq)}  \to \text{Ho}(\modelcat)^{(\R, \leq)}$ be the functor defined by the composite of $\theta$ and the localization functor $\modelcat^{(\R, \leq)} \to \text{Ho}(\modelcat^{(\R, \leq)})$. Observe that $\theta = q_*$.

%\medskip
Let $X$ and $Y$ be objects in $\et\modelcat$. Following Blumberg and Lesnick \cite{Bl-L}, and Lanari and Scoccola \cite{LS},
we introduce the {\it homotopy interleaving distance} and the
{\it  homotopy commutative interleaving distance} defined by
\[
d_{\text{HI}}(X, Y) \! := \text{inf}(\{ \e \geq 0 \mid \text{$X$, $Y$ are $\e$-homotopy interleaved} \}\cup \{\infty\} )\ \text{and}
\]
\[
d_{\text{HC}}(X, Y) \! := \text{inf}(\{ \e \geq 0 \mid \text{$X$, \!$Y$ are $\e$-homotopy commutative interleaved} \}\cup \{\infty\}),
\]
respectively. We refer the reader to Proposition \ref{prop:inequalities} for inequalities which hold for the interleaving distances mentioned above.
%\todo{Kuri (3/27) : Proposition \ref{prop:inequalities} is added.}
%Moreover, Lanari and Scoccola \cite{LS} introduce
%the {\it interleaving distance in the homotopy category}\label{index:homotopyID} define by
%\[
%d_{\text{IHC}}(X, Y) \! :=\! \text{inf}(\{ \e \geq 0 \mid \text{$X$, \!$Y$ are $\e$-interleaved in the homotopy category} \}\cup \{\infty\}).
% \ \text{and}
%\]

\begin{rem}\label{rem:d_{IHC}_and_d_{HC}}
Interleavings in the homotopy category and homotopy commutative interleavings can be composed, respectively.
Thus, we see that the distances $d_{\text{IHC}}$ and $d_{\text{HC}}$ satisfy the triangle inequality and then
$d_{\text{IHC}}$ and $d_{\text{HC}}$ are pseudodistances on the class of objects in $\et{\modelcat}$.
\end{rem}

Let $\modelcat^{(\R, \leq)}$ be the category endowed with the projective model structure
for a cofibrantly generated model category $\modelcat$.
In \cite[2.2.3]{LS},
it is proved that the homotopy interleavings on the category $\modelcat^{(\R, \leq)}$ can be composed by using the functorial fibrant replacement.
However, it is not immediate that the homotopy interleavings on $\et{\modelcat}$ for a general model category $\modelcat$ are composable.

\begin{prop}\label{prop:composable} %\todo{Kuri(2/1) : By virtue of this proposition, $d_{\text{\em HI}}$ is regarded as a pseudodistance on $\et{\modelcat}$. (3/24) : See the previous paragraph.}
Suppose that each object in $\modelcat$ is fibrant. Then, homotopy interleavings can be composed in $\et{\modelcat}$.
As a consequence, the distance $d_{\text{\em HI}}$ is a pseudodistance on the class of objects in $\et{\modelcat}$.
\end{prop}

\begin{proof} It follows from the assumption that
each object in $\et{\modelcat}$ is also fibrant; see the model category structure of $\et{\modelcat}$ described in Definition \ref{thm:ModelCat}.
Then, the result follows from the proof of \cite[Proposition 2.3]{LS} with \cite[Lemma 2.2]{LS}. We may apply \cite[Lemma 6.1.4]{S}
when making trivial fibrations with a common domain in the argument \cite[Proposition 2.3]{LS}. Observe that a trivial fibration is stable under pullback.
\end{proof}

\begin{rem}\label{rem:d_{HI}} %\todo{Kuri (3/24) : $d_{\rm HI}$ is a pseudodistance. See the comment after Definition \ref{defn:IHD_formality}.}
Let $\mathsf{CDGA}$ denote the category of commutative differential graded algebras (CDGAs) over $\Q$.
In the rest of the manuscript except for Section \ref{section:formalities}, we mainly focus on considering the interleaving distance in the homotopy category of $\et{\mathsf{CDGA}}$ the category of extended tame persistence CDGAs.
We observe that the category $\mathsf{CDGA}$ is endowed with the model structure introduced in \cite[4.2 Definition]{B-G}. In particular, each object is fibrant. Then, by Proposition \ref{prop:composable}, the distance $d_{\rm HI}$ is a pseudodistance on the class of objects in
$\et{\mathsf{CDGA}}$.
\end{rem}

\section{Extended tame commutative differential graded algebras}\label{sect:pCDGAs}

%Let $\modelcat$ be a model category and $\mathsf{et}\modelcat:=\mathsf{et}([0, \infty), \modelcat)$ the category of {\it extended} tame persistence objects in $\modelcat$.  By definition, an extended tame object $X$ is discretised by a sequence  $\tau_0 < \cdots < \tau_l < \cdots$ which is divergent or finite.

We begin by recalling relative Sullivan algebras and explain how to relate the algebra to a persistence object; see
\cite{Halperin} for relative Sullivan algebras (KS-extensions) and their homotopy theory.

\subsection{Relative Sullivan models for maps and tame persistence CDGAs.}\label{sect:relativeSullivan_pCDGA}
Let $\iota :\wedge V_A \to \wedge V_A \otimes \wedge W$ be a minimal relative Sullivan algebra. Then, by definition, there exists a filtration $\{W^p(r)\}_{r\geq 0}$ for each $W^p$ such that
\begin{eqnarray}\label{eq:minimalAL}
d : 1\otimes W^p(r)\to \wedge V_A\otimes \wedge (W^{<p}\oplus W^p(r-1)).
\end{eqnarray}
Observe that the minimal relative Sullivan model for a map is unique up to isomorphism; see \cite[Theorem 14.12]{FHT}.
In what follows, we assume further that
$\wedge V_A$ is minimal. Then, the model gives rise to a sequence
$\theta(\iota)$ of CDGAs defined by
$\theta(\iota)(n):=  \wedge V_A \otimes \wedge (W^{\leq n})$ together with the inclusions
\[
\theta(\iota)(n < n+1): \wedge V_A \otimes \wedge (W^{\leq n}) \to  \wedge V_A \otimes \wedge (W^{\leq n+1})=\wedge V_A\otimes  \wedge (W^{\leq n}) \otimes
\wedge (W^{n+1})
\]
as maps connecting the CDGAs. By the definition of the minimality of a relative Sullivan algebra, we have

\begin{prop}\label{prop:relativeSullivanAlgs}
Each inclusion $\theta(\iota)(n < n+1)$ is a minimal relative Sullivan algebra with the filtration $\{W^{n+1}(r)\}_{r\geq 0}$.
\end{prop}

Let $f : X \to A$ be a continuous map between path-connected spaces and $\iota_f :\wedge V_A \to \wedge V_A \otimes \wedge W$ a minimal relative Sullivan model for the map $f$. Thus, we have a commutative diagram
\begin{eqnarray}\label{eq:A_{PL}}
\xymatrix@C20pt@R15pt{
\wedge V_A \ar[r]^-{\iota_f} \ar[d]_{\sim} & \wedge V_A \otimes \wedge W \ar[d]^{\sim}\\
A_{\rm PL}(A)  \ar[r]_{A_{\rm PL}(f)} & A_{\rm PL}(X)
}
\end{eqnarray}
whose vertical arrows are quasi-isomorphisms, where $A_{\rm PL}$ denotes the polynomial de Rham functor; see \cite[Theorem 3.1]{FHT_II} for the existence of a minimal relative Sullivan model for a map.

We say that a map $f : X\to A$ is {\it relatively finite}
if the vector space $W$ is of finite dimension for the minimal relative Sullivan model $\wedge V_A \otimes \wedge W$ for the map $f$.
We observe that $\Theta(f)$ is tame in the sense in Definition \ref{defn:et} if the map $f$ is relatively finite.

In order to relate the sequence $\theta(\iota_f)$ to a persistence object, we recall the floor function
$\lfloor \ \rfloor : \mathbb{R} \to \mathbb{Z}$. Then,
we have an extended tame persistence CDGA $\Theta'(f)$ defined by
$\Theta'(f):= (\lfloor \ \rfloor)^*\theta(\iota_f)$. It is immediate that $\Theta'(f)$ is tame if $f$ is relatively finite.
%\todo{Following the model structure, it is required to prove that $F \to F\otimes\wedge(t, dt)$ is a path object, namely a cofibration.}

\begin{rem}\label{rem:iso}
It follows from \cite[4.6 Theorem]{Halperin} that a minimal relative Sullivan algebra $\iota_f$ associated with a map $f : X \to A$ is uniquely determined up to isomorphism and then so is $\Theta'(f)$ in $\et{\mathsf{CDGA}}$.
\end{rem}

Let  $\text{Func}(I, \mathsf{Top}_0)$ be the functor category from $I$ the category consisting of two objects and nontrivial one arrow to
$\mathsf{Top}_0$ the category of path-connected topological spaces.
The main result in this section is described as follows.%\todo{Kuri (11/24) We consider connected topological spaces.}

\begin{thm}\label{thm:main}
The function $\Theta'$ gives rise to a contravariant functor
\[
\Theta : \text{\em Func}(I, \mathsf{Top}_0) \to
\text{\em Ho}(\mathsf{et}([0, \infty), \mathsf{CDGA}))=
\text{\em Ho}(\et{\mathsf{CDGA}})
\]
where the right-hand side denotes the homotopy category of $\et{\mathsf{CDGA}}$.
%Moreover, the functor $\Theta$ induces the functor $\widetilde{\Theta} : \text{\em Ho}(\text{\em Func}(I, \mathsf{Top}_0)) \to \text{\em Ho}(\et{\mathsf{CDGA}})$ between homotopy categories.
%
\end{thm}

The proof of the theorem is postponed to Section \ref{section:PathCylinderOb}.

We consider the colimit functor $\text{colim} : \et{\mathsf{CDGA}} \to \mathsf{CDGA}$ whose right adjoint is the constant diagram functor $c$. Since $c$ preserves fibrations as well as trivial fibrations, it follows that the pair $(\text{colim}, c)$ is a Quillen adjunction. %\todo{Check them for $\Delta$.}
Thus, we have a left derived functor ${\bf L}(\text{colim}): \text{Ho}(\et{\mathsf{CDGA}}) \to  \text{Ho}(\mathsf{CDGA})$; see, for example,
%\cite[Theorem 2.3.9]{C} and
\cite[Proposition 9.3]{DS}.

\begin{prop}\label{prop:0_to_1} The diagram of functors
\begin{eqnarray}\label{eq:0_to_1}
\xymatrix@C30pt@R15pt{
\text{\em Func}(I, \mathsf{Top}_0)^{\rm op} \ar[r]^-{\Theta}  &  \text{\em Ho}(\et{\mathsf{CDGA}}) \ar[d]^{{\bf L}({\rm colim})}  \\
 \mathsf{Top}_0^{\text{\em op}} \ar[r]_-{q\circ A_{\rm PL}(\ )} \ar[u]^-{i} & \text{\em Ho}(\mathsf{CDGA})
}
\end{eqnarray}
is commutative, where $q : \mathsf{CDGA}\to {\rm Ho}(\mathsf{CDGA})$ denotes the localization functor and $i$ is the functor defined by $i(X) = X \to \ast$ the trivial map.
\end{prop}

\begin{proof}
The commutativity follows from the fact that the image of the functor $\Theta$ consists of cofibrant objects; see Proposition \ref{prop:cofibrantOb}.
\end{proof}

%\begin{rem}\label{rem:n_to_n+1}
%As a consequence of Assertion \ref{assertion:0_to_1}, one has
%a commutative diagram
%\begin{eqnarray}\label{eq:1_to_2}
%\xymatrix@C30pt@R20pt{
%(\text{Func}(I, \mathsf{Top}_0))^{\rm op})^{({\mathbb R}^n, \leq)} \ar[r]^-{\Theta_*}  &  (\text{Ho}(\et{\mathsf{CDGA}}))^{({\mathbb R}^n, \leq)}
%\ar[d]^{(\mathbb{L}(\text{colim}))_*} \\
%( \mathsf{Top}_0^{\text{op}})^{({\mathbb R}^n, \leq)} \ar[r]_-{(q\circ A_{PL}(\ ))_*} \ar[u]^{i_*} & (\text{Ho}(\mathsf{CDGA}))^{({\mathbb R}^n, \leq)}.
%}
%\end{eqnarray}
%Thus, we obtain $(n+1)$-parameter persistence objects from $n$-parameter persistence objects via $\Theta_*$.
%In \cite{Zhou}, Zhou considers the interleaving distance and related distances between persistence spaces with the functor $(q\circ A_{PL}(\ ))_*$ in the case where $n=1$.
%\end{rem}

\subsection{Cylinder and path objects  in $\et{\mathsf{CDGA}}$}\label{section:PathCylinderOb}
In order to develop homotopy theory for extended tame persistence CDGAs,
we clarify a cylinder object and a path object for $\theta(\iota)$ associated with a relative Sullivan algebra $\iota$ in $\et{\mathsf{CDGA}}$.

We recall the cylinder object which is used when defining the left homotopy in $\mathsf{CDGA}$; see \cite[Chapter 5]{Halperin}.
Let $\wedge Z$ be a Sullivan algebra and $\nabla : \wedge Z\otimes \wedge Z \to \wedge Z$ the multiplication.
Define graded vector spaces $\overline{Z}$ and $\widehat{Z}$ by $\overline{Z}^n = Z^{n+1}$  and $\widehat{Z}^n = Z^{n}$, respectively.
We write $\overline{v}$ and $\widehat{v}$ for the elements in $\overline{Z}$ and $\widehat{Z}$ corresponding to $v$.
Then, the map $\nabla$ is decomposed as
\[
\wedge Z \otimes \wedge Z \stackrel{i}{\to} (\wedge Z)^I : = (\wedge Z \otimes \wedge Z \otimes \wedge \overline{Z}, \widetilde{d}) \maprightud{\zeta}{\cong}
(\wedge Z \otimes \wedge \overline{Z} \otimes \wedge \widehat{Z}, D) \maprightud{\rho}{\sim}  \wedge Z,
\]
where the map $i$ is the natural inclusion, $\rho$ is defined by $\rho(\widehat{v}) = 0 = \rho(\overline{v})$, $\rho(v) = v$ and $D(\overline{v}) = \widehat{v}$.
Moreover, the map $\zeta$ is an isomorphism of algebras defined by
$\zeta(\overline{v}) = \overline{v}$, $\zeta(v\otimes 1\otimes 1) = v$ and $\zeta(1\otimes v\otimes 1) = \sum_{n=0}^\infty \frac{\theta^n}{n!}(v)$, where $\theta = sD + Ds$ with the derivation $s$ of degree $-1$ defined by $s(v) = \overline{v}$, $s(\overline{v}) = 0 = s(\widehat{v})$. The differential $\widetilde{d}$ is defined by $\widetilde{d}= \zeta^{-1}D\zeta$.  We observe  that $\zeta(1\otimes v\otimes 1) =  v + \widehat{v} + \sum_{n= 1}^\infty \frac{(sD)^n}{n!}(v)$.

 We see that for objects $F$ and $F'$ in
$\et{\mathsf{CDGA}}$,  the pointwise tensor product  $F\otimes F'$ is the coproduct of  $F$ and $F'$.  The pointwise multiplication gives rise to
a morphism $\nabla : \theta(\iota)\otimes \theta(\iota) \to \theta(\iota)$, where
$\theta(\iota)$ is the persistence CDGA associated with a relative Sullivan algebra $\iota :\wedge V \to \wedge V \otimes \wedge W$.
We define a persistence CDGA $\theta(\iota)^I$ by the sequence
\[
(\wedge V)^I \to (\wedge V\otimes \wedge W^{\leq 1})^I \to  (\wedge V\otimes \wedge W^{\leq 2})^I \to \cdots .
\]
Then, we have a commutative diagram
 \begin{eqnarray}\label{diagram:factrisation_I}
\xymatrix@C40pt@R15pt{
\theta(\iota) \otimes \theta(\iota) \ar@/^1.0pc/[rr]^{\nabla} \ar[r]_-{g} & \theta(\iota)^I \ar[r]_-\pi & \theta(\iota)
}
\end{eqnarray}
in $\et{\mathsf{CDGA}}$, where $g$ and $\pi$ are induced by the natural inclusion and the projection $\pi$ mentioned above, respectively.

\begin{lem}\label{lem:CofFib}
In the commutative diagram (\ref{diagram:factrisation_I}), the map $g$ is a cofibration and $\pi$ is a trivial fibration in $\et{\mathsf{CDGA}}$. It turns out that the persistence CDGA
$\theta(\iota)^I$ is a cylinder object for $\theta(\iota)$.
\end{lem}

\begin{proof}
The definitions of the fibration and the weak equivalence described in Theorem \ref{thm:ModelCat} enable us to deduce that $\pi$ is a trivial fibration.
In order to show that $g$ is a cofibration, we recall the construction with the diagram (\ref{diagram:the_pushout}). Then, we have a commutative diagram
\begin{eqnarray*}%\label{diagram:the_pushout}
\xymatrix@C5pt@R12pt{
(\wedge V\otimes \wedge W^{\leq k})^{\otimes 2} \ar[r]^{\theta (\iota)^{\otimes 2}({k < k+1})} \ar[d]_{g({k})}& (\wedge V\otimes \wedge W^{\leq k+1})^{\otimes 2} \ar[d]^{\bar{g}({k+1})}  \ar@/^8.0pc/[dd]^{g({k+1})} &\\
(\wedge V\otimes \wedge W^{\leq k})^{\otimes 2}\otimes \wedge(\overline{V\oplus W^{\leq k}})  \ar@/_0.8pc/[rd]_-{(\theta (\iota)^I)({k<k+1})\ \ \ } \ar[r] &
(\wedge V\otimes \wedge W^{\leq k+1})^{\otimes 2}\otimes \wedge(\overline{V\oplus W^{\leq k}}) \ar[d]^{\hat{g}(k+1)} & \\ % \ar@{.>}@/_2.0pc/
  &  (\wedge V\otimes \wedge W^{\leq k+1})^{\otimes 2}\otimes \wedge(\overline{V\oplus W^{\leq k+1}}) &
}
\end{eqnarray*}
in which the inside square is a pushout.  It follows that $\hat{g}$ is a relative Sullivan algebra and hence $g$ is a cofibration.
\end{proof}

For a minimal relative Sullivan algebra $\iota$, we consider a commutative diagram
 \begin{eqnarray}\label{diagram:factrisation_II}
\xymatrix@C40pt@R15pt{
 \theta(\iota) \ar@/^1.0pc/[rr]^{\Delta} \ar[r]_-{h} & \theta(\iota)\otimes \wedge(t, dt) \ar[r]_-p &\theta(\iota) \times \theta(\iota),
}
\end{eqnarray}
where $\Delta$ denotes the diagonal map, $h$ is the canonical inclusion and $p$ is the epimorphism defined by $p = ev_0\times ev_1$.
We regard $\wedge(t, dt)$ as a constant persistence CDGA.

\begin{lem}\label{lem:FibCof}
In the commutative diagram (\ref{diagram:factrisation_II}), the map $h$ is a trivial cofibration and $p$ is a fibration in $\et{\mathsf{CDGA}}$. It turns out that the persistence CDGA
$\theta(\iota)\otimes \wedge(t, dt)$ is a path object for $\theta(\iota)$.
\end{lem}

\begin{proof} It is readily seen that $h$ is a weak equivalence. Then,
it suffices to show that $h$ is a cofibration. The diagram (\ref{diagram:the_pushout}) for $h$ is given by
\begin{eqnarray*}%\label{diagram:the_pushout}
\xymatrix@C15pt@R15pt{
\wedge V\otimes \wedge W^{\leq k} \ar[r]^{\theta (\iota)({k < k+1})} \ar[d]_{h({k})}& \wedge V\otimes \wedge W^{\leq k+1} \ar[d]^{\bar{h}({k+1})}  \ar@/^8.0pc/[dd]^{h({k+1})} &\\
(\wedge V\otimes \wedge W^{\leq k}) \otimes \wedge(t, dt)  \ar@/_0.8pc/[rd]_-{(\theta (\iota)\otimes \wedge(t, dt))({k<k+1})\ \ \ \ \ \ } \ar[r] &
(\wedge V\otimes \wedge W^{\leq k+1})\otimes \wedge(t, dt) \ar[d]^{\hat{h}(k+1)} & \\ % \ar@{.>}@/_2.0pc/
  &  (\wedge V\otimes \wedge W^{\leq k+1})\otimes \wedge(t, dt). &
}
\end{eqnarray*} %\todo{$^{k+1}$}
We see that $\hat{h}(k+1)$ is the identity map and then $h$ is a cofibration in $\et{\mathsf{CDGA}}$.
\end{proof}

In order to prove Theorem \ref{thm:main}, we use a relative homotopy in the sense in \cite[9.13 Definition]{Halperin}.

\begin{lem}\label{lem:relativeH}
Let $H$ be a relative homotopy between relative Sullivan algebras $\iota$ and $\iota'$ in $\mathsf{CDGA}$; that is, $H$ fits in the commutative diagram
\begin{eqnarray*}
\xymatrix@C15pt@R15pt{
(\wedge V)^I \ar[r]^-\iota \ar[d]_H & (\wedge V\otimes \wedge W)^I \ar[d]^H \\
\wedge V' \ar[r]_-{\iota'}& \wedge V'\otimes \wedge W'.
}
\end{eqnarray*}
Then, one has  $H(\theta(\iota)^I(k))\subset \theta(\iota')(k)$.
\end{lem}
\begin{proof}
This result follows from the fact that $H$ preserves the degree of CDGAs.
\end{proof}

\begin{proof}[Proof of Theorem \ref{thm:main}] The proof heavily relies on the properties of left homotopies on relative Sullivan algebras investigated in  \cite{Halperin}.
Let $f : X\to A$ and $g : Y\to B$ be continuous maps and $(h, k) : f \to g$ a morphism in $\text{Func}(I, \mathsf{Top}_0)$. Then, the construction in \cite[(10,10)]{Halperin} allows us to obtain a diagram
\begin{eqnarray}\label{eq:two_squares}
\xymatrix@C15pt@R12pt{
\wedge V' \ar[rrr]^{\iota_g} \ar[ddd]_\varphi \ar[rd]_{\gamma'}^\sim & & & \wedge V' \otimes \wedge W'  \ar[ddd]_\psi \ar[ld]^{\eta'}_\sim \\
 & A_{\rm PL}(B) \ar[r]^{g^*} \ar[d]_{k^* }& A_{\rm PL}(Y)  \ar[d]^{h^* } & \\
 &A_{\rm PL}(A) \ar[r]^{f^*}  & A_{\rm PL}(X) & \\
 \wedge V \ar[rrr]_{\iota_f} \ar[ur]^\gamma_\sim & & & \wedge V \otimes \wedge W  \ar[lu]_{\eta}^\sim
}
\end{eqnarray}
in which $\iota_f$ and $\iota_g$ are relative Sullivan models for $f$ and $g$, respectively, for the right and left squares, $(k^*\gamma', h^*\eta') \sim_{\text{rel}} (\gamma\varphi, \eta \psi)$ and all the remaining squares commute; see \cite[9.13 Definition]{Halperin} for the definition of the relative homotopy relation $\sim_{\text{rel}}$. Moreover, if  a pair $(\varphi_1, \psi_1)$ of morphisms from $\iota_g$
to $\iota_f$
satisfies the condition above, then we have
$(\varphi_1, \psi_1) \sim_{\text{rel}} (\varphi, \psi)$. This follows from  \cite[9.19 Theorem]{Halperin}.
In order to consider left homotopy in $\et{\mathsf{CDGA}}$,
 we may use the cylinder object, which is introduced in Lemma \ref{lem:CofFib}, for an extended tame CDGA $\theta(\iota_f)$.
 Thus, it follows from Lemma \ref{lem:relativeH} that a relative homotopy in $\mathsf{CDGA}$ gives rise to a homotopy in $\et{\mathsf{CDGA}}$.
Therefore, the assignment $\Theta$ on the objects is well defined.

The functoriality of $\Theta$ follows from the uniqueness of the pair $(\varphi, \psi)$ up to homotopy and the fact that the composite of morphisms of CDGAs preserves the equivalence relation on relative homotopies; see \cite[9.17 and 9.18 Propositions]{Halperin}.
%  \todo{We have to verify that the well-definedness of $\theta$ and $\Theta$. To this end, it is needed to clarify a path/cylinder object in $\et{\mathsf{CDGA}}$ which defines the notion of the right/left homotopy.}
%We may use the equivalence between the left and right homotopy in the category of CDGAs; see \cite[Section 2.2]{FOT}.
\end{proof}

\begin{rem}\label{rem:The_form}
Let $\iota_f : \wedge V_A \to \wedge V_A \otimes \wedge W$ be a minimal relative Sullivan model for a map $f : X\to A$. Then, by definition, the persistence CDGA $\Theta(f)$ has the form $\wedge V_A \otimes \wedge W^{\leq n}$ on the interval $[n, n+1)$.
\end{rem}

\begin{rem}\label{rem:cofibrantOb}
The property (\ref{eq:minimalAL}) enables us to deduce that the structure map $\theta(\iota)^{s<t}$ in $\theta(\iota)$ for a relative Sullivan algebra $\iota$ is
a minimal relative Sullivan algebra.  By Proposition \ref{prop:cofibrantOb} (ii), we see that for each map $f : X \to Y$, the persistence CDGA $\Theta'(f)$ is cofibrant in
$\et{\mathsf{CDGA}}$. Moreover, every object in $\et{\mathsf{CDGA}}$ is fibrant. Thus, we may consider $\Theta'(f)$ itself without a fibrant-cofibrant replacement in $\text{Ho}(\et{\mathsf{CDGA}})$.
 %\todo{Check the assertions.}
\end{rem}

\begin{rem}\label{rem:SullivanRep} Given a Sullivan representative $\varphi_f : \minmodel(A)=\wedge V_A \to \wedge V_X=\minmodel(X)$ for a map $f : X \to A$, we have a minimal relative Sullivan algebra
$i_f : \wedge V_A  \to \wedge V_A \otimes \wedge W$
for $\varphi_f$; that is, $i_f$ satisfies the condition that $\eta_f\circ i_f = \varphi_f$ for some quasi-isomorphism $\eta_f : \wedge V_A \otimes \wedge W \to \wedge V_X$.
The result \cite[Proposition 2.22]{FOT} enables us to deduce that $i_f$ also fits in the commutative diagram (\ref{eq:A_{PL}}) instead of $\iota_f$.
This implies that the persistence CDGA $\theta(i_f)$ obtained by $i_f$ coincides with $\theta(\iota_f)$ up to isomorphism.
\end{rem}

%\todo{Kuri (3/24): We observe a homotopy in  $\et{\mathsf{CDGA}}$.}
As seen above, a homotopy $h$ between morphisms $\varphi, \psi\colon F \to G$ in $\et{\mathsf{CDGA}}$ is a morphism of the form
$F^I \to G$ or $F \to G\otimes \wedge(t, dt)$,
where $F^I$ is a cylinder object for $F$ and $G\otimes \wedge(t, dt)$ is a path object for $G$.
We observe that the homotopy $h$ gives rise to a pointwise homotopy $h(i)$ between $\varphi(i), \psi(i)\colon F(i) \to G(i)$ for each $i$ but the inverse does not necessarily hold.
The following remark explains how to construct a right homotopy between objects in $\et{\mathsf{CDGA}}$ associated with minimal relative Sullivan algebras.

\begin{rem}\label{rem:Homotopies}
Let
$h: \wedge V \otimes \wedge W \to \wedge V'\otimes \wedge W' \otimes \wedge (t, dt)$ be a morphism of CDGAs between
minimal relative Sullivan algebras $\iota : \wedge V \to \wedge V\otimes \wedge W$ and  $\iota': \wedge V' \to \wedge V'\otimes \wedge W'$.
We do not necessarily assume that $h(V) \subset \wedge V'$.

Let $\varepsilon$ be a positive real number.
Suppose that $h(v)$ is in $\wedge V' \otimes \wedge (W'^{\leq \lfloor \varepsilon \rfloor})\otimes \wedge(t, dt)$ for $v\in V$.
Then, since $h(w)$ is in $\wedge V' \otimes \wedge (W'^{\leq i})\otimes \wedge (t, dt)$ for $w \in W^{\leq i}$, it follows that
the restriction of $h$ gives rise to a map $\widetilde{h} : \theta(\iota) \to \theta(\iota')^\varepsilon \otimes \wedge (t, dt)$ in $\et{\mathsf{CDGA}}$.
\end{rem}

%We may write $\Theta_A$ for the functor in Assertion \ref{assertion:main} when stressing the overcategory for the given base space $A$.
%Such a functor is also defined on other overcategory $\mathsf{Top}/B$.

%%%%%%%%%%%%%%%%%%%%%%%%%%%%%%%%%%%%%%%%%%%%%%%%%%%%%%%%%%%
%%%%%%%%%%%%%%%%% Postnikov towers and persistence CDGAs %%%%%%%%%%%%%%%%%%%%%

 \section{A Postnikov tower meets an extended tame persistence CDGA}\label{sect:MPs-P_CDGAs}
 %\todo{Kuri, Sekizuka (3/12) This section is added.}
%We use the same notations and terminology as those in \cite{KNSWY}.
It is well known that the minimal model for a simply-connected space $X$ is related to a Postnikov tower for $X$ via Hirsch extensions.
As we did not find any literature explaining in detail the same relationships between a relative Sullivan model for a map $f$ and
a Postnikov tower for $f$, we clarify such a relationship from the viewpoint of persistence theory in this section.

\subsection{A partial Quillen equivalence and the functor $\Theta$}\label{sect:pQe}
We begin with the notion of a {\it partial Quillen equivalence} introduced in \cite{M-S}.
%In this section, it is assumed that a model category, which is a codomain of a functor category, admits functorial factorizatons (\cite{Ho, H}).
Let $F:\mathcal{C} \rightleftarrows \mathcal{D}:U$ denote a Quillen adjunction between model categories with the unit $\eta$ and the counit $\epsilon$; see \cite{DS} for derived functors.
Let $q_X : QX \to X$ be a cofibrant replacement of $X\in \mathcal{C}$ % in $\mathcal{C}$
and $r_Y : Y \to RY$ a fibrant replacement of $Y \in \mathcal{D}$.
 %in $\mathcal{D}$.
 Then, we define the {\it derived unit} and the {\it derived counit} by the composites
\[
\xymatrix{
\widetilde{\eta_X} : X\ar[r]^-{\eta_X}&UFX\ar[r]^-{U_{r_{FX}}}&URFX \ \ \ \  \text{and} \ \ \ \
%}
%\]
%and %the derived counit is the composition
%\[
%\xymatrix{
\widetilde{\epsilon_Y} : FQUY\ar[r]^-{F_{q_{UY}}}&FUY\ar[r]^-{\epsilon_Y}&Y ,
}
\]
respectively. %Here, $\eta$ and $\epsilon$ are the unit and the counit of the adjunction $L\dashv R$.

\begin{lem}\label{lem:WE} The following conditions are equivalent.
\begin{itemize}
\item[{\rm (i)}]   For every cofibrant object $X$ of $\mathcal{C}$ and every fibrant object $Y$ of $\mathcal{D}$,
    %each objectwise derived unit of $X$ and each objectwise derived counit of $Y$
    the derived unit $\widetilde{\eta_X}$ and the derived counit $\widetilde{\epsilon_Y}$
    are weak equivalences.
\item[{\rm (ii)}]   For every cofibrant object $X$ of $\mathcal{C}$ and every fibrant object $Y$ of $\mathcal{D}$,
a map $f:X\to U(Y)$ is a weak equivalence in $\mathcal{C}$ if and only if its adjoint $f^{\flat}: F(X)\to Y$ is a weak equivalence in $\mathcal{D}$.
\end{itemize}
    As a consequence, if each $\widetilde{\eta_X}$ and each $\widetilde{\epsilon_Y}$ are weak equivalences
    for every cofibrant object $X$ and every fibrant object $Y$, then the derived functors
    %\[
    ${\bf L}F  :  {\rm Ho}(\mathcal{C}) \rightleftarrows  {\rm Ho}(\mathcal{D}) : {\bf R}U$
    %\]
    are inverse equivalences of categories.
\end{lem}
\begin{proof} See, for example, \cite[Lemma C.3.6]{R-V}.
 The latter half follows from \cite[Theorem 9.7 (ii)]{DS}. In fact, the condition (ii) is the sufficient one in the theorem.
\end{proof}

\begin{defn}\label{defn:pQe}(cf. \cite[Definition 2.4]{M-S})
Let $F:\mathcal{C} \rightleftarrows \mathcal{D}:U$ be a Quillen adjunction. Suppose $\mathcal{C}_0 \subset \mathcal{C}$ and $\mathcal{D}_0 \subset \mathcal{D}$ are full subcategories such that the following hold:

{\rm{(1)}} The two subcategories are closed under weak equivalences; that is, %, in the sense that
if $C$ is weakly equivalent to $C_0$ with $C_0 \in \mathcal{C}_0$, then $C \in \mathcal{C}_0$ too, and analogously for $\mathcal{D}_0$.

{\rm{(2)}} $FX  \in \mathcal{D}_0$ for a cofibrant object $X \in \mathcal{C}_0$ and $UY \in \mathcal{C}_0$  for a fibrant object $Y \in \mathcal{D}_0$%\footnote{The original condition (2) in \cite{M-S} is as follows: $L\mathcal{C}_0 \subset \mathcal{D}_0$, $R\mathcal{D}_0 \subset \mathcal{C}_0$.}.

{\rm{(3)}} Every  objectwise derived unit $X\to URFX$ is a weak equivalence for a cofibrant $X \in \mathcal{C}_0$. Dually, every objectwise derived counit $FQUY\to Y$ is a weak equivalence for
a fibrant $Y \in \mathcal{D}_0$.

Then, we call the adjunction $F:\mathcal{C} \rightleftarrows \mathcal{D}:U$ a {\it partial Quillen equivalence for the subcategories $\mathcal{C}_0$ and $\mathcal{D}_0$}.
\end{defn}

\begin{rem}\label{rem:equivalence}   %\todo{In this remark, we  explain a difference between a partial Quillen equivalence in Definition \ref{defn:pQe} and the original one.}
Originally, the notion of a partial Quillen equivalence in \cite{M-S} is defined for two model categories which admit functorial factorizations.
Moreover, the second condition requires that
$F\mathcal{C}_0 \subset \mathcal{D}_0$ and $U\mathcal{D}_0 \subset \mathcal{C}_0$.
As mentioned in \cite[The first paragraph in Page 5]{M-S},
(**) :
the restriction of the derived adjunction ${\rm Ho}(\mathcal{C}) \rightleftarrows {\rm Ho}(\mathcal{D})$ to full subcategories consisting of objects
$\mathcal{C}_0$ and $\mathcal{D}_0$, respectively, gives an equivalence of categories.

While the condition (2)  in Definition \ref{defn:pQe} is different from the original one and the model category that we deal with does not require
the functorial factorizations,
the assertion (**) remains valid. % even if the last condition is replaced with that in Definition \ref{defn:pQe}.
This follows from the fact that the unit and counit of the derived functors in Lemma \ref{lem:WE} give rise to isomorphisms on the objects in the full subcategories of the homotopy categories; see the proof of \cite[Theorem 9.7 (ii)]{DS}.
Observe that the conditions (1) and  (2) yield the well-definedness of the restrictions of the derived functors.
\end{rem}

To describe the Sullivan--de Rham equivalence theorem
\cite[10.1 Theorem]{B-G} in terms of the partial Quillen equivalence,
we recall finiteness properties of a CDGA and a simplicial set.

\begin{defn}\label{defn:finite_Q}
(i) A CDGA $A$ is of finite $\Q$-type if each $M^i$ is finite dimensional for the minimal model $M$ for $A$. \\
(ii) A simplicial set $X$ is of finite $\Q$-type if the $i$th rational cohomology group of $X$ for $i\geq 0$ is finite dimensional.
\end{defn}

\begin{thm}\label{thm:RHT} {\rm (cf. \cite[10.1 Theorem]{B-G} and \cite[Theorem 3.4]{M-S}.)}
%\todo{This theorem explains the Sullivan--de Rham equivalence theorem in terms of the partial Quillne equivalence.}
Let $\mathcal{C}_0$ be the full subcategory of $\mathsf{CDGA}$ consisting of connected CDGAs of finite $\Q$-type.
Let  $\mathcal{D}_0$ denote the full subcategory of $\mathsf{sSet}$ consisting of nilpotent path-connected rational simplicial sets of finite $\Q$-type.
%We recall the result \cite[10.1 Theorem]{B-G}. In particular, the theorem asserts that
Then, the adjoint pair
%\[
$\langle \ \rangle : \mathsf{CDGA} \rightleftarrows \mathsf{sSet}^{\rm op} : A_{\rm PL}(\ )$
%\]
of the realization functor $\langle \ \rangle$ and the polynomial de Rham functor $A_{\rm PL}(\ )$ gives rise to a partial Quillen equivalence for the subcategories
$\mathcal{C}_0$ and $\mathcal{D}_0^{\rm op}$.
\end{thm}

\begin{proof}
It is readily seen that the condition (1) holds. The well-definedness of the functors in \cite[10.1 Theorem]{B-G} allows us to deduce that the condition (2) holds. We consider the condition (3).
Since the adjunction map $\varphi_Y$ dealt  with in the proof of \cite[10.1 Theorem (i)]{B-G}
is nothing but the objectwise derived counit $\widetilde{\epsilon_Y}$, it follows from the proof that  $\widetilde{\epsilon_Y}$ is
a weak equivalence.
Moreover, the proof of \cite[10.1 Theorem (ii)]{B-G} yields that the unit $\eta_X$ is a weak equivalence for a cofibrant $X \in \mathcal{C}_0$ and then so is the derived unit $\widetilde{\eta_X}$.
We observe that the homotopy category of fibrant and cofibrant objects of a model category $\mathcal{M}$, which is equivalent to ${\rm Ho}(\mathcal{M})$,   is considered in \cite[Sections 8, 9 and 10]{B-G}.
%We observe that the left derived functor is defined with cofibrant replacements; see
%\cite[Section 9]{DS} for more details.
\end{proof}

%Suppose that two model categories admit a partial Quillen equivalence for appropriate subcategories. Then, the result \cite[Lemma 5.2]{M-S} asserts that the induced Quillne adjunction between functor categories is also partial Quillen equivalence provided the projective model structures exist on the functor categories.
The following lemma is a variant of \cite[Lemma 5.2]{M-S}.

\begin{prop}\label{prop:pQe}
Let $\mathcal{C}$ and $\mathcal{D}$ be model categories. Let $F:\mathcal{C} \rightleftarrows \mathcal{D}:U$ be a partial Quillen equivalence for the subcategories $\mathcal{C}_0$ and $\mathcal{D}_0$. Then, the induced Quillen adjunction $F_*:\mathsf{et}(\mathcal{C^{\mathbb{R_+}}}) \rightleftarrows \mathsf{et} (\mathcal{D^{\mathbb{R_+}}}):U_*$ %between them
is a partial Quillen equivalence for the subcategories $\mathsf{et}(\mathcal{C}_0^{\mathbb{R_+}})$ and $\mathsf{et}(\mathcal{D}_0^{\mathbb{R_+}})$.
\end{prop}

\begin{proof}
    The proof is verbatim the same as that of the original result. Instead of the derived counit and the derived unit obtained by the  cofibrant and fibrant replacement functors in the proof of  \cite[Lemma 5.2]{M-S}, we consider the derived counit and the derived unit defined at the beginning of this section. % with the same argument as in  \cite[Lemma 5.2]{M-S}.
Moreover, in our proof, it is necessary to pay attention to the difference in the definitions of the partial Quillen equivalence
as mentioned in Remark \ref{rem:equivalence}. %\todo{Kuri (3/28): Why do we repeat the proof? This is an excuse. }

    It is readily seen that the condition (1) holds. By Proposition \ref{prop:cofibrantOb}, we see that a cofibrant object of $\mathsf{et}(\mathcal{C^{\mathbb{R_+}}})$ is pointwise cofibrant. By definition,
a fibrant object of $\mathsf{et}(\mathcal{C^{\mathbb{R_+}}})$ is pointwise fibrant.
Then, the condition (2) holds.
   % \todo{Sekizuka will prove that the condition on the objectwise derived unit in (3) holds.}
    We verify that the condition (3) holds. Let $X\in \mathsf{et}(\mathcal{C}_0^{\mathbb{R_+}})$ be a cofibrant object. Proposition \ref{prop:cofibrantOb} (ii) yields that $X_0$ is a cofibrant object and the transition morphism $X({s<t}):X(s) \to X(t)$ for any $s<t$ in
    $\mathbb{R}_+$  is a cofibration in $\mathcal{C}_0$. Thus, we see that $X(t)$ is cofibrant in $\mathcal{C}_0$ for $t\in \mathbb{R}_+$.
     Let $\underline{r}_Y:Y\to \underline{R}Y$ be a fibrant replacement of $Y\in \mathsf{et}(\mathcal{D}_0^{\mathbb{R_+}})$.
     Observe that $(\underline{R}Y)(t)$ is in $\mathcal{D}_0$ and then $ \underline{R}Y$ is an object in  $\mathsf{et}(\mathcal{D}_0^{\mathbb{R_+}})$. 
    We prove that the composite
    \[
    \xymatrix{
    \widetilde{\eta_X} : X\ar[r]^-{{{\eta}_*}_X} &U_*F_*X\ar[r]^-{{U_*}_{{\underline{r}}_{F_*X}}} & U_*\underline{R}F_*X
    }
    \]
    is a weak equivalence in $\mathsf{et}(\mathcal{C}_0^{\mathbb{R_+}})$, where $\eta_*$ denotes the unit of $F_*\dashv U_*$, which is indeed componentwise the unit $\eta$ of $F\dashv U$. Therefore, we need to check that the composite
    \[
    \xymatrix@C30pt@R15pt{
    (\widetilde{\eta_X})(t) : X(t)\ar[r]^-{{{\eta}}({X(t)})} &UFX(t)\ar[r]^-{{U}_{{\underline{r}}_{F_*X}(t)}} & U(\underline{R}F_*X)(t)
    }
    \]
    %\todo{Naito(3.30):$UFX^t$ instead of $FUX^t$?}
    is a weak equivalence in $\mathcal{C}_0$ for all $t \in \mathbb{R}_+$.
    The model category structure on $\mathsf{et} (\mathcal{D^{\mathbb{R_+}}})$ described in Theorem \ref{thm:ModelCat} implies that
     $(\underline{R}F_*X)(t)$ is fibrant. Moreover,  by Proposition \ref{prop:cofibrantOb} (i), we see that
     $(\underline{r}_{F_*X})(t)$ is an acyclic cofibration.   Let $r_Y:Y\to RY$ be a fibrant replacement of $Y \in \mathcal{D}_0$.
     Then, there exists a lift $h$ which makes the diagram
    \[
    \xymatrix@C25pt@R15pt@M=7pt{
    FX(t) \ar@{>->}[d]^{\sim}_{(\underline{r}_{F_*X})(t)} \ar@{>->}[r]^{r_{FX(t)}}_{\sim}&R(FX(t)) \\
    (\underline{R}F_*X)(t)\ar@{.>}[ur]_h
    }
    \]
    %By the 2-out-of-3 property of weak equivalences,
    %Then,
    %the lift $h$ is a weak equivalence.
    %We will now use the fact that a fibration in $\mathsf{et}(\mathcal{D}^{\mathbb{R_+}})$ is a levelwise fibration [Theorem \ref{thm:ModelCat}]. This implies that a fibrant object in $\mathsf{et}(\mathcal{D}^{\mathbb{R_+}})$ is levelwise fibrant. Thus $\underline{P}(L_*X)^t$ is fibrant in $\mathcal{D}$ for all $t$.
    commutative.
    Since $h$ is a weak equivalence between fibrant objects, it follows from the dual to \cite[9.9 Lemma (K. Brown)]{DS} that
    the right Quillen functor $U$ takes $h$ to a weak equivalence. Then, we have the commutative diagram
    
    %\smallskip
    \[
    \xymatrix@C40pt@R20pt{
    X(t) \ar[r]^-{\eta_{X(t)}} \ar@/^20pt/@{.>}[rr] & UFX(t) \ar[r]^-{U_{(r_{F_*X})(t)}} \ar@/_20pt/[rr]_{U_{r_{FX(t)}}} & U\underline{R}(F_*X)(t) \ar[r]^{U_h}_{\sim} & UR(FX(t)).
    %&&
    }
    \]
   As mentioned above, the object $X(t)$ is cofibrant in $\mathcal{C}_0$.
    Moreover, by assumption, the adjoint  $F\dashv U$ is a partial Quillen equivalence.
    Therefore,
    the total composite, which is an objectwise derived unit of $F\dashv U$ at $X(t)$, is a weak equivalence.
    Since $U_h$ is a weak equivalence, it follows from the 2-out-of-3 property that
     the dotted composite, which is nothing but the morphism $(\widetilde{\eta_X})(t)$,  is also a weak equivalence. %, which is what we wanted to prove.
    %Moreover, the original proof uses the fact that a cofibrant object in the functor category is a pointwise cofibrant; see \cite[11.6.3]{H}.
    %In the part, we may apply Propositio \ref{prop:cofibrantOb}.

    The dual argument works to show that the objectwise derived counit is a weak equivalence.
\end{proof}

\begin{rem}
Let $\mathcal{C}$ be a model category and $\mathsf{tame}(\mathcal{C^{\mathbb{R_+}}})$ denote the full subcategory of
$\mathsf{et}(\mathcal{C^{\mathbb{R_+}}})$ comprised of tame objects; see Definition \ref{defn:et}. The category
$\mathsf{tame}(\mathcal{C^{\mathbb{R_+}}})$ is endowed with the model category structure described in Theorem \ref{thm:ModelCat}, which is defined originally in \cite{CGL}. Proposition \ref{prop:pQe} remains valid even if `$\mathsf{et}$' is changed with `$\mathsf{tame}$'. This follows from the proof of the proposition.
\end{rem}

%We recall a Postnikov tower for a morphism in $\mathsf{sSet}$ the category of simplicial sets introduced by Goerss and Jardine.

In order to describe the main result in this section, we recall the definition of a Postnikov tower for a map.

 \begin{defn}{\cite[Chapter VI Definition 3.9]{GJ}}\label{defn_Postnikov}
Let $f:X\to Y$ be a morphism of simplicial sets. %map between connected topological spaces.
Then a Postnikov tower for $f$ is a tower of spaces
 $\{ X_n, q_{n+1} : X_{n+1} \to X_n \}_{n\geq 0}$ equipped with maps of towers
\[
\{ i_n \} : \{ X \} \to \{ X_n \},
\quad
\{ p_n \} : \{ X_n \} \to \{ Y \}
\]
such that
\begin{itemize}
\item[(1)] the composite $p_n \circ i_n$ is $f$ for any $n$;
\item[(2)] for any choice of vertex of $X$, the map $(i_n)_* : \pi_k (X)\to \pi_k (X_n)$ is an isomorphism for $k \leq n$;
\item[(3)]  for any choice of vertex of $X$, the map $(p_n)_* : \pi_k (X_n) \to \pi_k (Y)$ is an isomorphism for $k>n+1$;
%and a monomorphism for $k=n+1$.
\item[(4)]  for any choice of vertex $v$ of $X$, there is an exact sequence
\[
\xymatrix@C15pt@R15pt{
0 \ar[r] & \pi_{n+1}(X_n) \ar[r] & \pi_{n+1}(Y) \ar[r]^-{\partial} & \pi_n(F_v) ,
}
\]
where $F_v$ is the homotopy fiber of $f$ at $v$ and $\partial$ is the connecting homomorphism in the long exact sequence of the homotopy fibration $F_v \to X \to Y$.
\end{itemize}
Here, $\{ X \}$ and $\{ Y \}$ denote the constant towers of $X$ and $Y$, respectively.
\end{defn}
 %\todo{Kuri (3/1): We here explain the Moore--Postnikov tower for a Kan fibration $f : X\to Y$ in the category of simplicial sets.}

We have a special Postnikov tower $\{ X(n) \}$ for a Kan fibration $f : X\to B$ with fibrant $B$ which is called the {\it Moore--Postnikov tower} for $f$; see \cite[VI, 3]{GJ}.

For $q$ simplices $\sigma, \tau : \Delta^q \to X$, we define a relation $\sigma \sim_n \tau$ if
\begin{itemize}
\item[i)]
$f\circ \sigma = f\circ \tau$ and
\item[ii)]
$\sigma_| = \tau_| : {\rm sk}_n\Delta^q \to  X$.
\end{itemize}
The relation gives an equivalence relation on $X$ and then the projections $i_n : X \to X(n):=X/\!\sim_n$ and well-defined maps $p_n : X(n) \to B$ induced by $f$ are defined.
The consideration in \cite[Chapter VI]{GJ} yields that  $\{ \{ X(n) \}, \{i_n\}, \{p_n\} \}$ is a Postnikov tower for $f : X\to B$.
Observe that there exist canonical projections $q_n :  X(n) \to X(n-1)$ and
the map $p_n : X(n) \to B$ is a fibration with $p_n = p_{n-1}\circ q_n$. Moreover, it follows that the construction is functorial. Thus,
we have a functor ${\rm MP}$ from the subcategory consisting of Kan fibrations with fibrant bases of $\rm{Func}(I, \mathsf{sSet})$ to
$\mathsf{et(sSet^{(\mathbb{Z}_+ , \leq )^{\rm op}})}$ defined by sending a Kan fibration $f$ to the Moore--Postnikov tower for $f$.
%\todo{Naito(3.30) : Change $\mathsf{et((sSet))^{\mathbb{Z}_+^{\rm op}})}$ to $\mathsf{et(sSet^{(\mathbb{Z}_+ , \leq )^{\rm op}})}$}

%We call a map $f : X\to Y$ is of {\it finite type} if the homotopy fiber $F_f$ is connected and the rational homotopy $\pi_i(F_f)\otimes \Q$ of the homotopy fiber of $f$ is trivial provided $i\geq N$  for some integer $N >1$.
 % \todo{Recall the definition of a partial Quillen equivalence; see \cite[Definition 2.4]{M-S}.}

Let $S_*(\ )$ be the singular simplex functor and $\mathcal{H}$ the homotopy fibration functor.  Let $\lfloor \ \rfloor : (\R_+, \leq) \to (\Z_+, \leq)$ denote the floor function. Observe that $f$ is a Serre fibration if and only if $S_*(f)$ is a Kan fibration; see \cite[Page 11]{GJ}.
We denote by $\mathcal{F}_{\text{sc}}$ the full subcategory of $\rm{Func}(I, \mathsf{Top}_0)^{\rm op}$
consisting of maps between simply-connected spaces of finite $\Q$-type. Observe that for a simply-connected space $X$, the space $X$ is of finite $\Q$-type if and only if the vector spaces $H^*(X; \Q)$ for $i \geq 1$ are finite dimensional; see \cite[Proposition 12.2]{FHT}.

\begin{thm}\label{thm:main_II} With the same notation as in Theorem \ref{thm:RHT},
one has a diagram of categories and functors %\todo{We use the projective model category structure on $\mathsf{(sSet)^{\R_+^{\rm op}}}$. Then, $\mathsf{et((sSet)^{\R_+^{\rm op}})}$ is regarded as its full subcategory.  }
\begin{eqnarray}\label{eq:equivalence}
\xymatrix@C28pt@R15pt{
\rm{Func}(I, \mathsf{Top})^{\rm op} \ar@/^1.8pc/[rr]^{mp:=(\lfloor \ \rfloor)^*\circ  {\rm MP} \circ S_*( \ ) \circ \mathcal{H}} & \mathsf{et(CDGA^{\R_+})} \ar@<1ex>[r]^-{\langle \ \rangle_*} \ar[d] &\mathsf{et((sSet^{\rm op})^{\R_+})} \ar@<1.2ex>[l]^-{A_{\rm PL}( \ )_*}_-{\bot}  \ar[d] \\
\rm{Func}(I, \mathsf{Top}_0)^{\rm op} \ar[r]^-{\Theta}  \ar[u]&  {\rm Ho}(\mathsf{et(CDGA^{\R_+})}) \ar@<1ex>[r]^-{{\bf L}\langle \ \rangle_*} & {\rm Ho}(\mathsf{et((sSet^{\rm op})^{\R_+})})  \ar@<1.2ex>[l]^-{{\bf R}A_{\rm PL}( \ )_*}_-{\bot}  \\
\mathcal{F}_{\rm sc} \ar@{.>}[r]_{\Theta_{\rm res}} \ar[u] \ar[ru]_{\Theta_{\rm res}} \ar@/_1.95pc/[rr]_-{\text{(Rationalization)} \circ (mp)=:(mp)_\Q} & {\rm Full}(\mathsf{et(\mathcal{C}_0^{\R_+}})) \ar@<1ex>[r]^-{{\bf L}\langle \ \rangle_*} \ar[u] &
{\rm Full}(\mathsf{et((\mathcal{D}_0^{\rm op})^{\R_+}})) \ar@<1.2ex>[l]^-{{\bf R}A_{\rm PL}( \ )_*}_-{\simeq} \ar[u]
}
\end{eqnarray}
for which the following assertions hold.

{\rm (i)} The adjunction $(\langle \ \rangle_*,  A_{\rm PL}( \ )_*)$ induced by $(\langle \ \rangle,  A_{\rm PL}( \ ))$ pointwise gives rise to a partial Quillen equivalence for the subcategories $\mathsf{et}(\mathcal{C}_0^{\R_+})$ and $\mathsf{et}((\mathcal{D}_0^{\rm op})^{\R_+})$.
% in the sense in \cite{M-S}.
As a consequence, the adjunction in the bottom of the lower-right square gives an equivalence of the categories.

{\rm (ii)} The diagram starting at $\mathcal{F}_{\rm sc}$ and ending at ${\rm Ho}(\mathsf{et(CDGA^{\R_+})})$ via
$\mathsf{et((sSet^{\rm op})^{\R_+})}$ is commutative.

{\rm (iii)} The lower diagram comprised of $\Theta_{\rm res}$, the adjoint and the rationalization $(mp)_\Q$ is commutative.
%where we replace the categories with {\it pointed ones}.

Here, $\Theta$ is the functor described in Theorem \ref{thm:main}, $\Theta_{\rm res}$ is its restriction and ${\rm Full}(\mathcal E)$ denotes the full subcategory of the homotopy category whose objects are in $\mathcal{E}$.
Moreover,
%in the right-lower square, the upper pair denotes the adjunction consisting of the {\it derived functor} and
%the lower is the pair of the restrictions.
the (Rationalization) is the composite ${\bf L}\langle \ \rangle_*\circ {\bf R}A_{\rm PL}( \ )_*$.
\end{thm}

\begin{rem}\label{rem:CDGA_sSet}
(i) We have an isomorphism ${\rm Func}(J, \C^{\rm op}) \cong  {\rm Func}(J^{\rm op} , \C)^{\rm op}$ between functor categories.
While the codomain of the functor $mp$ is the full subcategory $(\mathsf{et((sSet)^{\R_+^{\rm op}})})^{\rm op}$ of the  functor category
${\rm Func}(\R_+^{\rm op} , \mathsf{sSet})^{\rm op}$ consisting of extended tame copersistence simplicial sets, we use the category
 $\mathsf{et((sSet^{\rm op})^{\R_+})}$ as the codomain up to the isomorphism of the categories.

 (ii) The (Rationalization) gives indeed rise to the {\it levelwise $\Q$-localization}. In fact, the counit $\e_X : {\bf L}\langle \ \rangle_*\circ {\bf R}A_{\rm PL}(X)_* \to X$ is indeed the levelwise $\Q$-localization in $\mathsf{sSet^{\rm op}}$; see \cite[11.2]{B-G} and \cite[Theorem 17.12]{FHT}.
Since each object in $\mathsf{et((sSet^{\text{op}})^{\R_+})}$ is fibrant, it follows that  ${\bf R}A_{\rm PL}( \ )_* = A_{\rm PL}( \ )_*$.

 (iii) The domain of the composite ${\rm MP}\circ S_*( \ ) \circ {\mathcal H} :
 \rm{Func}(I, \mathsf{Top}) \to \mathsf{sSet}^{(\mathbb{Z}_+, \leq)^{\rm op}}$ is a simplicial model category of towers of simplicial sets; see \cite[Chapter VI, 1]{GJ}. Moreover, we see that the adjoint
 \[
\lfloor \ \rfloor^* :  (\mathsf{sSet}^{(\mathbb{Z}_+, \leq)^{\rm op}})^{\rm op}=(\mathsf{sSet}^{\rm op})^{({\mathbb Z}_+, \leq)} \rightleftarrows  \mathsf{et((sSet^{\rm op})^{\R_+})} : {\rm res}
 \]
 %\todo{Naito(3.30) : $(\mathsf{sSet}^{\rm op})^{({\mathbb Z}_+, \leq)}$ instead of $(\mathsf{sSet}^{\rm op})^{({\mathbb Z}, \leq)}$? $\mathsf{et((sSet^{\rm op})^{\R_+})}$ instead of $\mathsf{et(sSet^{\rm op})}^{\R_+}$?}
 of the functor induced by the floor function and the restriction is a Quillen adjunction.
  %\todo{Kuri (3/25): We know the model structure of the category of towers.}
\end{rem}

\begin{cor}\label{cor:Alg_vs_Space} For objects $f$ and $g$ in $\mathcal{F}_{\rm sc}$, it holds that
\[
d_{\rm IHC}(\Theta (f), \Theta(g)) = d_{\rm IHC}(mp(f)_\Q, mp(g)_\Q),
\]
where the right-hand side denotes the interleaving distance between copersistence simplicial spaces in the homotopy category.
\end{cor}

\begin{proof}
This follows from the commutativity of the bottom diagram in (\ref{eq:equivalence}) and the equivalence of categories.
\end{proof}

\begin{proof}[Proof of Theorem \ref{thm:main_II}]
%Regarding the proof of (i) in Assertion \ref{assertion:main},
(i) The assertion follows from Proposition \ref{prop:pQe}. Remark \ref{rem:equivalence} and Theorem \ref{thm:RHT} yield the latter half of the assertion.
%We observe that  $\mathsf{sSet}$ and $\mathsf{CDGA}$ are cofibrantly generated model strucrures and hence the categories admit functorial factorizations.
%[Outline of the proof.] We may reconsider the proof of \cite[Lemma 5.2]{M-S} with equivalent conditions for the Quillen equivalence between model cateogries in the book due to Hovey. Moreover, we use the result asserting that  a cofibration in the functor category endowed with the projective model category structure is a levelwise cofibration; see [The book due to Hirschhorn, 11.6.3]. This assertion holds in the category $\mathsf{et}(\mathcal{M})$ for a model category $\mathcal{M}$.

(ii) %The result is proved by using properties of Moore--Postnikov (MP) tower for a map which appear in the proof of \cite[Chapter VI, 2, Theorem 3.11]{GJ}. %and a cofibrant replacement of the object $(A_{\rm PL}( \ )_* \circ mp)(f)$ in $\mathsf{et(CDGA^{\R_+})}$ for a map $f : X \to B$.
We give an appropriate cofibrant replacement of the object  $(A_{\rm PL}( \ )_* \circ mp)(f)$ in $\mathsf{et(CDGA^{\R_+})}$ for a map $f : X \to B$.
Consider the commutative diagram
%\begin{align}\label{eq:s}
\begin{eqnarray}\label{eq:squares}
\xymatrix@C12pt@R15pt{
% & & & & \\
& \wedge V_B \ar[d]_{=} \ar[r]_-{j_1}& \wedge V_B\otimes \wedge W_1 \ar@{.>}[r]_-{j_2} \ar[d]_\sim  & \cdots \ar@{.>}[r] & \wedge V_B\otimes \wedge W_{n-1} \ar@{.>}[r]_-{j_n} \ar[d]_\sim &
 \wedge V_B\otimes \wedge W_{n} \ar@{.>}[r] \ar[d]_\sim  & \cdots \\
& \wedge V_B \ar[d]_\sim \ar[r]& \wedge V_B\otimes \wedge Z_1 \ar[r] \ar[d]_\sim & \cdots \ar[r] & \wedge V_B\otimes \wedge Z_{n-1} \ar[r] \ar[d]_\sim &
 \wedge V_B\otimes \wedge Z_{n} \ar[r] \ar[d]_\sim & \cdots \\
& A_{\rm PL}(B)\ar[r]^-{A_{\rm PL}(q_1)} &  A_{\rm PL}(X(1)) \ar[r]^-{A_{\rm PL}(q_2)} & \cdots \ar[r] &  A_{\rm PL}(X(n-1)) \ar[r]^-{A_{\rm PL}(q_n)}&  A_{\rm PL}(X(n))  \ar[r] &\cdots
}
\end{eqnarray}
%\end{align}
with solid arrows in which the second sequence consists of cofibrations and the third one is induced by the MP tower of $f$. Moreover, the retraction to minimal relative Sullivan algebras gives rise to the first line; that is, each vertical map is the inclusion and a quasi-isomorphism.
This follows from \cite[Theorem 14.9]{FHT}.
We show that there exist dots arrows such that each square is commutative and the first sequence is isomorphic to $\Theta(f)$ in $\text{Ho}(\mathsf{et(CDGA^{\R_+})})$.

The construction in \cite[pages 204 and 205]{FHT} enables us to obtain a morphism $j_n$ which fits the commutative diagram
\[
\xymatrix@C20pt@R5pt{
 & \wedge V_B\otimes \wedge W_{n-1} \ar[dd]^{j_n} \ar[r] & \wedge W_{n-1} \ar[dd]^{\widetilde{j_n}}\\
\wedge V_B \ar[ru] \ar[rd]&  \\
&  \wedge V_B\otimes \wedge W_{n} \ar[r] & \wedge W_{n}
}
\]
and the square in (\ref{eq:squares}) containing $j_n$ and $A_{\rm PL}(q_n)$ as the top and bottom arrows, respectively, is homotopy commutative. Moreover, the map $\widetilde{j_n}$ induced by $j_n$ is a Sullivan representative for the natural map $(q_n)_| : F(n-1) \to F(n)$, where $F(n)$ is the fiber of the fibration $p_n : X(n) \to B$. We observe that
the space $F(n)$ is the $n$th stage of the Moore--Postnikov tower of the fiber $F$ of $f$. Moreover, since $B$ is simply connected by assumption, it follows that $\wedge W_n$ is a minimal model for $F(n)$; see the proof of \cite[Theorem 3.11]{GJ}. Then, we see that $W_n = W_n^{\leq n}$ and
$j_n : W_{n-1} \to W_n^{\leq n-1}$ is an isomorphism. Therefore, we may assume that $j_n$ is an inclusion.
Thus, the colimit ${\rm colim}_n(\wedge V_B\otimes \wedge W_n)$ is regarded as a relative Sullivan algebra of the form $\iota' : \wedge V_B \to \wedge V_B\otimes \wedge \widetilde{W}$.

By virtue of \cite[Proposition 2.22]{FOT}, we may replace inductively vertical arrows connecting the first and third sequences
in (\ref{eq:squares}) with quasi-isomorphisms so that the each square is strictly commutative.  It turns out that  $\theta(\iota')$ is a cofibrant replacement for $A_{\rm PL}(mp(f))$ in
$ \mathsf{et(CDGA^{\R_+})}$.

In what follows, we show that $\theta(\iota')$ is isomorphic to  $\theta(\iota_f)$.
The universality of the colimit gives rise to a morphism $u$ which makes the square consisting of solid arrows and $u$ in the diagram
\begin{align}\label{eq:squaresII}
\xymatrix@C20pt@R15pt{
\wedge V_B \ar[d]_{=} \ar[r]& \wedge V_B\otimes \wedge \widetilde{W} \ar@{.>}@/^3pc/[dd]^(0.7){u} \ar@{.>}[d]^\Phi \ar[r]& \wedge \widetilde{W}
\ar@{.>}@/^3pc/[dd]^(0.7){\overline{u}} \ar@{.>}[d]^{\overline\Phi}\\
\wedge V_B \ar[d]_\sim \ar[r]& \wedge V_B\otimes \wedge W \ar[d]_\sim^v \ar[r] & \wedge W \ar[d]^{\overline{v}}_\sim\\
A_{\rm PL}(B)\ar[r]_-{A_{\rm PL}(f)} &  A_{\rm PL}(X) \ar[r]& A_{\rm PL}(F)
}
\end{align}
commutative. The lifting property (\cite[Proposition 14.6]{FHT}) implies that there exists a morphism $\Phi$ such that the upper-left square is commutative and
$v \circ \Phi \simeq u$ rel $\wedge V_B$. Moreover, the relative homotopy gives a homotopy between $\overline{v} \circ \overline{\Phi}$ and  $\overline{u}$, where $\overline{v}$, $\overline{\Phi}$ and  $\overline{u}$ are morphisms induced by $v$, $\Phi$ and $u$, respectively.

By \cite[Proposition 15.5]{FHT}, we see that $\overline{v}$ is a quasi-isomorphism.
Therefore, the restriction of
$\overline{\Phi}$ to $\wedge W_n$  gives rise to a Sullivan representative for $i_n$ in a Moore-Postnikov tower
$\{i_n\} : \{ F\} \to \{F(n)\}$ of the fiber $F$. This implies that the linear part $Q(\overline{\Phi}) : W_n \to W^{\leq n}$ is an isomorphism.
Observe that $\wedge \widetilde{W}$ and $\wedge W$ are minimal.
Then, the morphism $\overline{\Phi}$ is an isomorphism; see Proposition \ref{prop:linear-part} below. It follows from \cite[Proposition 3.10]{FHT_II} that $\Phi$ is an isomorphism.

(iii) The result (i) and the commutativity in (ii)  yield the assertion (iii).
\end{proof}

\begin{rem} With the same notation as in the proof of Theorem \ref{thm:main_II}, in the category of simplicial sets,
we have a commutative diagram %\todo{Kuri (1/16): The evident map is an isomorphism.}
\begin{eqnarray*}
\xymatrix@C20pt@R1pt{
 & & \text{lim} \langle\wedge V_B\otimes \wedge (W^{\leq n}) \rangle \ar[dd]^{\text{pr}_n}\\
S_*(X)_{\Q}=\langle \wedge V_B\otimes \wedge W \rangle \ar@/_0.8pc/[rrd]_{\langle (\text{inc.})_n \rangle} \ar[r]^-{\langle\eta\rangle}_-{\cong} & \langle\text{colim} \wedge V_B\otimes \wedge (W^{\leq n})\rangle
\ar[rd]_{\langle i_n \rangle} \ar[ru]_{u}^{\cong}& \\
& &  \langle\wedge V_B\otimes \wedge (W^{\leq n}) \rangle,
}
\end{eqnarray*}
where $\eta :\text{colimit} \wedge V_B\otimes \wedge W^{\leq n} \to \wedge V_B\otimes \wedge W$ is a canonical isomorphism. The realization functor $\langle \ \rangle$ is the left adjoint to the functor $A_{\rm PL}$ and then $\langle \ \rangle$ preserves the colimit. Thus, the evident map $u$ is an isomorphism. As a consequence, the universality of the limit enables us to conclude that the composite $u\circ \langle\eta\rangle$ is the evident map $S_*(X)_{\Q}=\langle\wedge V_B\otimes \wedge W \rangle \to \text{lim} \langle\wedge V_B\otimes \wedge (W^{\leq n}) \rangle$ which is an isomorphism.
\end{rem}

\subsection{The geometric realization of an extended tame CDGA}\label{sect:GeometricRealizations}
%\todo{Naito added a subsection on geometric desctription (1/20)}
In this subsection, we work in the category of topological spaces and
we discuss the geometric realization of the extended tame CDGA $\theta (\iota)$ obtained from a minimal relative Sullivan algebra $\iota : \wedge V \to \wedge V \otimes \wedge W$. Here, we assume that $W= W^{\geq 1}$. However, the condition that $V^1 = 0$ is not necessarily assumed.
%We recall a Postnikov tower for a map in $\mathsf{Top}$  introduced by Goerss and Jardine.

While the original definition (Definition \ref{defn_Postnikov}) of a Postnikov tower of a map is given for simplicial sets, we here consider its analogue for path-connected topological spaces.
In the end of the section, we give  geometric realizations of specific extended tame CDGAs described in Sections \ref{sect:Interleavings_maps} and \ref{sect:Examples}.

Let $| \, \cdot \, |$ be the spatial realization functor; see \cite[\S 17]{FHT}. %\todo{Kuri (2/20) : Some sentences are revised to make the format consistent.}
By applying the %spatial realization
functor $| \, \cdot \, |$ to the sequence $\theta (\iota)$, we obtain a tower of topological spaces
\[
\xymatrix@C15pt{
\left\{ |\theta (\iota) (n)| \right\} : |\theta (\iota) (0)| & |\theta (\iota) (1)| \ar[l] & |\theta (\iota)(2)| \ar[l] & \cdots \ar[l] & |\theta (\iota)(n)| \ar[l] & \cdots \ar[l]
}
\]
Each structure map
$\theta(\iota)(n) \to \theta(\iota)(n+1)$ of $\theta(\iota)$ is a relative Sullivan algebra.
Then, by \cite[Proposition 17.9]{FHT}, we see that the sequence above is a tower of fibrations.

Let $i_n$ and $p_n$ be the spatial realizations of the inclusions $\theta(\iota)(n) \to \wedge V\otimes \wedge W$ and $\wedge V \to \theta(\iota)(n)$,  respectively.

 %\todo{Kuri (3/1): Do we assume that each $\theta(\iota)(n)$ is relatively finite? We are glad when the answer is No because of Remark \ref{rem:finite_type} below. }
\begin{prop}\label{prop:P_tower}
The tower $\left\{ |\theta (\iota) (n)| \right\}_{n\geq 0}$ equipped  with $\{i_n\}$ and $\{p_n\}$ is a Postnikov tower for the map $|\iota| : |\wedge V \otimes \wedge W| \to |\wedge V|$  in the sense in Definition \ref{defn_Postnikov}. %Goerss and Jardine.
\end{prop}
\begin{proof}
%Let $i_n$ and $p_n$ be the spatial realization of the inclusions $\theta(\iota)(n) \to \wedge V\otimes \wedge W$ and $\wedge V \to \theta(\iota)(n)$,  respectively.
To prove that $\{ |\theta(n)| \}$ is a Postnikov tower, we verify the conditions (1)--(4) in Definition\ref{defn_Postnikov}.
%\todo{Kuri (2/18) : We have to show that the condition (4) holds.}

(1) Clearly, $p_n \circ i_n = |\iota|$ for each $n$ by the naturality of the spatial realization.

(2) In the case $n=0$, $|\wedge V\otimes \wedge W|$ and $|\theta (\iota)(0)|=|\wedge V|$ are connected.
Thus, it is immediate that $(i_0)_* : \pi_0 (|\wedge V\otimes \wedge W|) \to \pi_0 (|\theta (\iota)(0)|)$ is an isomorphism.
Suppose that $n\geq 1$. By \cite[Proposition 17.9]{FHT}, we see that $i_n$ is a fibration with fiber $\left| \wedge (W^{>n}) , \bar{d} \right|$. Here $(\wedge (W^{>n}) , \bar{d})$ is the quotient of $\wedge V\otimes \wedge W$ by the ideal generated by $\left( \wedge V \otimes \wedge (W^{\leq n}) \right)^+$.
Since $(\wedge W^{>n}, \bar{d})$ is a Sullivan algebra with generators only in degrees $k > n$, it follows from \cite[\S 1.7 Theorem 1.1 (iii)]{FHT_II} and the formality of $S^n$ that
\[
\pi_k \left( \left| \wedge (W^{> n}) , \bar{d} \right| \right)
\cong
\left[
\left( \wedge (W^{> n}), \bar{d} \right),
A_{\rm PL}(S^k)
\right]
\cong\left[
\left( \wedge (W^{> n}), \bar{d} \right),
H^*(S^k ; \Q)
\right].
\]
Since the right-hand side homotopy set consists only of the trivial morphism for $k \leq n$, we have $\pi_k \left( \left| \wedge (W^{> n}) , \bar{d} \right| \right) =0$  for $k \leq n$.
Thus, the homotopy long exact sequence of homotopy groups associated with $i_n$ implies the condition (2).

(3) The proof of the condition (3) is essentially the same as that of (2).
In particular, the result  \cite[Proposition 17.9]{FHT} allows us to deduce that $p_n : |\theta (\iota)(n)| \to |\wedge V|$ is a fibration with fiber $\left|\wedge (W^{\leq n}), \bar{d}\right|$.
We also see that $\pi_k \left( \left| \wedge (W^{\leq n}) , \bar{d} \right| \right) =0$  for $k \geq n+1$ from \cite[\S 1.7 Theorem 1.1 (iii)]{FHT_II}.
% It follows from \cite[Theorem 1.1 (iii)]{FHT_II} and the formality of $S^n$ that
% \[
% \pi_k \left( \left| \wedge (W^{\leq n}) , \bar{d} \right| \right)
% \cong
% \left[
% \left( \wedge (W^{\leq n}), \bar{d} \right),
% A_{PL}(S^k)
% \right]
% \cong\left[
% \left( \wedge (W^{\leq n}), \bar{d} \right),
% H^*(S^k ; \Q)
% \right].
% \]
% Since the right-hand side homotopy set consists only of the trivial morphism for $k \geq n+1$, we have $\pi_k \left( \left| \wedge (W^{\leq n}) , \bar{d} \right| \right) =0$  for $k \geq n+1$.
By the long exact sequence of homotopy groups associated with $p_n$, we conclude that condition (3) is satisfied.

(4) Consider the commutative diagram
\[
\xymatrix@C20pt@R15pt{
0 \ar[r] & \pi_{n+1}(|\theta (\iota)(n)|) \ar[r]^-{(p_{n})_*} & \pi_{n+1}(|\wedge V|) \ar[r]^-{\partial} & \pi_{n+1}(F)
\\
& \pi_{n+1}(|\wedge V \otimes \wedge W|), \ar[u]^-{(i_{n})_*} \ar[ur]_-{|\iota|_*}  &&
}
\]
where $F$ is the fiber of $|\iota|$.
The commutativity of this diagram follows from (1).
We shall show that the top sequence in the diagram is exact.
By the same argument as in (3), we see that $(p_{n})_*$ is a monomorphism.
Similarly, we can verify that $(i_n)_*$ is an epimorphism by the same argument as in (2).
It follows from the commutativity of the diagram that
%\[
${\rm Ker} \, \partial = {\rm Im} \, |\iota|_* =   {\rm Im} ( (p_{n})_* \circ  (i_{n})_*) = {\rm Im} \, (p_{n})_*$ .
%\]
This completes the proof.
\end{proof}

% \begin{rem}
% Let $B$ and $C$ be simply connected CDGAs, that is, $B^0 = C^0 = \K$ and $B^1 = C^1 = 0$.
% Note that for any CDGA map $\varphi : B \to C$, there exists a minimal relative Sullivan model $m : B \otimes \wedge W \stackrel{\simeq}{\longrightarrow} C$ for $\varphi$ that satisfies the assumption on the differential in Proposition; see \cite[Proposition 14.3 and 14.9]{FHT} for a construction of minimal relative Sullivan models.
% \end{rem}

\begin{rem}\label{rem:finite_type}
In Proposition \ref{prop:P_tower}, we do not assume that the each component $\theta(i)$ of the persistence CDGA is relatively finite; that is, $\dim \theta(i)^n <\infty$ for $n\geq 0$. Then, it does not seem that Theorem \ref{thm:main_II} (iii) yields Proposition \ref{prop:P_tower}.
\end{rem}

%We provide examples of Postnikov towers and their Sullivan models.
We provide examples each of which is a decomposition of a map $f$ whose pointwise rationalization is a Postnikov tower for the map $f_\Q$.
These serve as specific geometric realizations of extended tame CDGAs described in Examples \ref{ex:CP_S_distance} and \ref{ex:p_q}.

\begin{ex}\label{ex:Postnikov:S_CP}
Let $i : S^2  \hookrightarrow {\mathbb C} P^2$ be the inclusion.
Recall that $\mathbb{C}P^2$ is obtained by attaching a 4-dimensional disk $D^4$ to $S^2$ via the Hopf map $h: S^3 \to S^2$.
We show that a Postnikov tower for the rationalization $i_{\Q}$ is given by
\begin{eqnarray}\label{Postnikov:S_CP}
\xymatrix@C20pt@R15pt{
{\mathbb C} P^2_{\Q} & {\mathbb C} P^2_{\Q} \ar[l]_-{=} & {\mathbb C} P^2_{\Q}  \ar[l]_-{=} & (X_3)_{\Q} \ar[l]_-{(\text{pr}_1)_{\Q}} & S^2_{\Q} \ar[l]_-{\bar{f}_{\Q}} & S^2_{\Q} \ar[l]_-{=} & \cdots , \ar[l]_-{=}
}
\end{eqnarray}
where $X_3$ is the pullback $\mathbb{C} P^2 \times_{K(\Z , 4)} PK(\Z , 4)$ of the path fibration $p : PK(\Z , 4) \to K(\Z , 4)$, which is defined by $p(\gamma)=\gamma(1)$, along a map $f:\mathbb{C} P^2 \to K(\Z , 4)$.
	The map $f: \mathbb{C}P^2 \to K(\mathbb{Z}, 4)$ is defined as the composite $g \circ q$, where $q: \mathbb{C}P^2 \to S^4$ is the quotient map that collapses $S^2$ to the base point.
	The map $g: S^4 \to K(\mathbb{Z}, 4)$ represents the fundamental class of $S^4$ under the identification $[S^4, K(\mathbb{Z}, 4)] \cong H^4(S^4; \mathbb{Z})$.
Since $f \circ i$ is the trivial map by the definition of $f$, the universal property of the pullback induces a unique map $\bar{f} : S^2 \to \mathbb{C}P^2 \times_{K(\mathbb{Z}, 4)} PK(\mathbb{Z}, 4)$ that makes the following diagram commutative:
\begin{eqnarray}\label{diag:S^2_CP}
\xymatrix@C20pt@R15pt{
S^2 \ar[rd]^-{\bar{f}} \ar@/^13pt/[rrd]^-{\text{trivial map}} \ar@/_15pt/[ddr]_-{i}& &
\\
& \mathbb{C} P^2 \times_{K(\Z , 4)} PK(\Z , 4) \ar[d]_-{\text{pr}_1} \ar[r]^-{\text{pr}_2} & PK(\Z , 4) \ar[d]^-{p}
\\
& \mathbb{C} P^2 \ar[r]^-{f=g\circ q} & K(\Z , 4).
}
\end{eqnarray}

The associated maps $\{i_n\}$ and $\{p_n\}$ in the Postnikov tower are defined by
\[
 i_n=
  \begin{cases}
    i_{\Q} & (n=0,1,2) \\
    \bar{f}_{\Q} & (n=3) \\
    id  & (n\geq 4)
  \end{cases}
\quad
\text{and}
\quad
 p_n=
  \begin{cases}
    id & (n=0,1,2) \\
    ({\rm pr}_1)_{\Q} & (n=3) \\
    i_{\Q}  & (n\geq 4).
  \end{cases}
\]
We shall verify that this tower satisfies conditions (1)--(4) of the definition of a Postnikov tower.
Condition (1) is clear from the definitions of $i_n$ and $p_n$.
Condition (2) is straightforward for $n \neq 3$.
To verify condition (2) for $n=3$ and $k \leq 3$ in the non-rational case,  %\todo{Kuri (3/12) : We pay attention to the computation of $\pi_k(X_3)$ in the non-rational case. We show that $\bar{f}_* : \pi_k(S^2) \to \pi_k(X_3)$ is an isomorphism for $k\leq 3$.}
we consider the fiber sequence
\begin{eqnarray}\label{fibration:X_3}
\xymatrix{
K({\mathbb Z} , 3) \cong \Omega K({\mathbb Z} , 4) \ar[r]
&
X_3 \ar[r]^-{{\rm pr}_1}
&
{\mathbb C} P^2.
}
\end{eqnarray}
Since $({\rm pr}_1)_* : \pi_k(X_3) \to \pi_k(\mathbb{C}P^2)$ and $i_* : \pi_k(S^2) \to \pi_k(\mathbb{C}P^2)$ are isomorphisms for $k<3$, it follows that $\bar{f}_* : \pi_k(S^2) \to \pi_k(X_3)$ is an isomorphism.

We consider the case $k=3$.
Observe that $i \circ h : S^3 \to \mathbb{C}P^2$ is null-homotopic.
We can choose a null-homotopy $H: S^3 \times I \to \mathbb{C}P^2$ defined by the composite $H = \Phi \circ \alpha$, where $\Phi: D^4 \to \mathbb{C}P^2$ is the characteristic map of the 4-cell and $\alpha: S^3 \times I \to D^4$ is given by $\alpha(u, t) = (1-t)u + tb$ for a basepoint $b \in S^3 = \partial D^4$.
By the construction of $H$, the composition $q \circ H$ induces a relative homeomorphism
\[
q\circ H : (S^3 , \{ b \})\times (I, \partial I) \stackrel{\cong }{\longrightarrow} (S^4, *).
\]
By the homotopy lifting property of the fibration ${\rm pr}_1$, there exists a homotopy $G:S^3 \times I \to X_3$ such that the following diagram commutes:
\[
\xymatrix@C25pt@R20pt{
S^3 \times \{ 0 \} \ar@{^{(}->}[d] \ar[r]^-{\bar{f}\circ h} & X_3 \ar[d]^-{{\rm pr}_1}
\\
S^3 \times I \ar[r]^-{H} \ar[ru]^-{G} & \mathbb{C}P^2.
}
\]
Explicitly, the map $G$ is given by $G(u,t)=(H(u,t), \gamma_{u,t})$, where $\gamma_{u,t} \in PK(\mathbb{Z},4)$ is the path defined by $\gamma_{u,t}(s)=g\circ q \circ H (u, ts)$ for $s\in [0,1]$.
Since $H(u,1)$ is the base point, the path $\gamma_{u,1}=g\circ q\circ H(u, \, \cdot \, )$ is contained in $\Omega K(\mathbb{Z},4)$.
This defines a map $\phi : S^3 \to \Omega K(\mathbb{Z}, 4)$ by $\phi(u) = \gamma_{u, 1}$ which makes the diagram
\[
\xymatrix@C40pt@R15pt{
S^3 \ar@/^18pt/[rr]^-{\phi} \ar[r]_-{{\rm ad}(q\circ H)}  \ar[d]_-{h} & \Omega S^4 \ar[r]_-{\Omega g}& \Omega K(\mathbb{Z},4) \ar[d]^-{\iota}
\\
S^2 \ar[rr]^-{\bar{f}} && X_3.
}
\]
commutative up to homotopy with the homotopy induced by $G$.
Here, ${\rm ad}(q\circ H)$ is the adjoint of $q\circ H$, and $\iota$ denotes the inclusion of the fiber.
Since $\pi_3(\mathbb{C}P^2) = 0$, which follows from the long exact sequence associated with the $S^1$-bundle $S^1 \to S^5 \to \mathbb{C}P^2$, it is readily seen that $\pi_3(\iota)$ is an isomorphism.
The composite $q \circ H$ gives rise to a homeomorphism $\Sigma S^3 \to S^4$. Then, we see that ${\rm ad}(q \circ H)$ is a generator of $\pi_3 (\Omega S^4) \cong \Z$.
Furthermore, since $g$ represents the fundamental class, $\Omega g$ induces an isomorphism on $\pi_3 (\Omega S^4)$. Thus, $\phi$ represents a generator of $\pi_3(\Omega K(\mathbb{Z}, 4))$, and the commutativity of the diagram implies that $\bar{f}_* : \pi_3(S^2) \to \pi_3(X_3)$ is an isomorphism.
Therefore, condition (2) is satisfied for $n=3$.

In the rational setting, it is well known that $\pi_{k}(S^2_{\mathbb Q})=0$ for $k>3$ and $\pi_{k}({\mathbb C}P^2_{\mathbb Q})=0$ for $k>5$.
The long exact sequence of the fibration \eqref{fibration:X_3} implies that $\pi_k ((X_3)_{\mathbb Q}) \cong \pi_k ({\mathbb C}P^2_{\mathbb Q})$ for $k>4$.
Furthermore, the rational homotopy group of the homotopy fiber $F$ of $i : S^2 \to \mathbb{C}P^2$ is given by
\[
\pi_k(F)\otimes \mathbb{Q} \cong
\pi_k(F_{\mathbb{Q}}) \cong
  \begin{cases}
    \mathbb{Q} & (k=3,4) \\
    0 & (\text{otherwise}) .
  \end{cases}
\]
Combining these observations, conditions (3) and (4) are readily seen to be satisfied.

We next describe the Sullivan model that corresponds to the Postnikov tower of $i_{\mathbb{Q}}$.  %\todo{Kuri (2/20): Naito will complete the assertion.}
Let ${\minmodel}(S^2)=(\wedge(x,y),d)$ and ${\minmodel}({\mathbb C}P^2) = (\wedge(x,z) , d)$ be  minimal Sullivan models of $S^2$ and ${\mathbb C}P^2$, respectively.
Their differentials satisfy $d(x)=0$, $d(y)=x^2$, and $d(z)=x^3$, where $|x|=2$, $|y|=3$, and $|z|=5$.
A Sullivan representative for $i$ is given by $i^* : \wedge (x,z) \to \wedge (x,y)$, $i^*(x)=x$ and $i^*(z)=xy$.
Then, we have a minimal relative Sullivan model $ \iota : \wedge(x, z)  \to \wedge (x,z, v,w)$  of $i^*$ via the quasi-isomorphism defined by
\[
\eta : \wedge (x, z, v,w) \stackrel{\sim}{\longrightarrow} \wedge (x,y),
\]
where  $ \eta (v)=y$,  $\eta (w) = 0$, $\eta(z) = xy$, $dv=x^2$ and $dw = xy - z$; see Remark \ref{rem:SullivanRep}.

The relative Sullivan model for $p : PK(\Z , 4 ) \to K(\Z , 4)$ is of the form
\[
(\wedge (w) , 0) \hookrightarrow (\wedge (w , v) , d),
\]
where $d(v)=w$ and $d(w)=0$; see \cite[\S 15(b)Example 2]{FHT} for the model of $K(\Z , 4)$.
By the choice of $f$, the morphism $f^* : \wedge (w) \to \wedge (x,z)$ given by $f^*(w)=x^2$ provides a Sullivan representative for $f$.
It follows from \cite[\S 15 Proposition 15.8]{FHT} that
\[
(\wedge (x,z) ,d) \otimes_{(\wedge (w) , 0)} (\wedge (w,v),d) \cong (\wedge (x,z,v),d)
\]
is a Sullivan model for $X_3={\mathbb C} P^2 \times_{K(\Z , 4)} PK(\Z , 4)$ and the inclusion $\iota_2 : (\wedge (x,z) ,d) \to  (\wedge (x,z,v),d)$ is a Sullivan representative for $\text{pr}_1$.
Moreover, we see that a morphism $\bar{f}^* : (\wedge (x,z,v),d) \to (\wedge (x,y),d)$ defined by $\bar{f}^* (x)=x$, $\bar{f}^*(z)=xy$, $\bar{f}^* (v)=y$ is a Sullivan representative for $\bar{f}$.
Indeed, the commutativity of  the lower-left triangle in the diagram \eqref{diag:S^2_CP} implies that any Sullivan representative $\varphi : (\wedge (x,z,v),d) \to (\wedge (x,y),d)$ for $\bar{f}$ must satisfy $\varphi(x)=x$.
Since $\varphi$ commutes with the differentials, it follows that $\varphi$ is uniquely determined as $\bar{f}^*$.
We also see that the inclusion $\iota_3 : (\wedge (x,z,v),d) \to (\wedge (x,z,v,w),d)$ provides a relative Sullivan model of $\bar{f}^*$ via the quasi-isomorphism $\eta$.
Therefore, the sequence
\[
\xymatrix@R0pt{
&
2
&
3
&
4
&
\\
\cdots \ar[r]^-{=}
&
(\wedge (x, z) ,d)
\ar[r]^-{\iota_2}
&
(\wedge (x,z,v),d)
\ar[r]^-{\iota_3}
&
(\wedge (x,z,v,w),d)
\ar[r]^-{=}
&
\cdots
}
\]
is a cofibrant replacement of the image of
the sequence \eqref{Postnikov:S_CP}
by $(\lfloor \ \rfloor)^* \circ A_{\rm PL}( \ )_*$,
%is a Sullivan model for the Postnikov tower \eqref{Postnikov:S_CP},
and this coincides with $\theta (\iota)$.
\end{ex}

\begin{ex}\label{ex:paths}
Let $p:PS^2 \to S^2 $ be the path fibration of $S^2$.
A Postnikov tower for the rationalization $p_{\Q}$ is given by
\begin{eqnarray}\label{Postnikov:PS^2}
\xymatrix@C40pt{
S^2_{\Q} & S^3_{\Q} \ar[l]_-{h_{\Q}} & PS^2_{\Q} \ar[l]_-{\tilde{{\rm ev}}( \, \cdot \, , 1)_{\Q}} & PS^2_{\Q} \ar[l]_-{=} & \cdots \ar[l]_-{=}
}
\end{eqnarray}
Here, $h$ is the Hopf map. In fact, since $h$ is a fiber bundle, it is a Hurewicz fibration from \cite[Chapter 7, \S 4]{May}.
We define the map $\tilde{{\rm ev}} : PS^2 \times I \to S^3$ as a lift in the following diagram
\[
\xymatrix@C50pt@R20pt{
PS^2 \times \{ 0 \} \ar[r]^-{*} \ar[d]_-{\text{incl.}} & S^3 \ar[d]^-{h}
\\
PS^2 \times I \ar[r]^-{{\rm ev}} \ar[ru]^-{\tilde{{\rm ev}}} & S^2,
}
\]
where ${\rm ev}$ in the bottom row is the evaluation map.
The associated maps $\{i_n\}$ and $\{p_n\}$ in the Postnikov tower are defined by
\[
 i_n=
  \begin{cases}
    p_{\Q} & (n=0) \\
    \tilde{{\rm ev}}( \, \cdot \, , 1)_{\Q} & (n=1) \\
    id  & (n\geq 2)
  \end{cases}
\quad
\text{and}
\quad
 p_n=
  \begin{cases}
    id & (n=0) \\
    h_{\Q} & (n=1) \\
    p_{\Q}  & (n\geq 2).
  \end{cases}
\]
It is straightforward to verify that this construction satisfies conditions (1)--(4) in Definition \ref{defn_Postnikov}.
Conditions (1) and (2) follow directly from the definition of the maps. In particular, condition (2) for $n=1$ is satisfied since both $S^3_{\mathbb{Q}}$ and $PS^2_{\mathbb{Q}}$ are simply connected.
We also see that, in the rational setting, conditions (3) and (4) are immediately satisfied since $PS^2_{\mathbb{Q}}$ is contractible and the homotopy groups $\pi_k(S^2_{\mathbb{Q}})$ and $\pi_k(S^3_{\mathbb{Q}})$ vanish for $k > 3$.

%The algebraic description of this example will be provided in \S \ref{ex:p_q}.

It follows that a Sullivan representative for the Hopf map $h:S^3 \to S^2$ is given by $h^* : (\wedge (x,y),d) \to (\wedge (e), 0)$, $h^*(x)=0$, $h^*(y)=e$,  where $|e|=3$.
For the relative Sullivan algebra $\iota_1 : (\wedge (x,y), d) \to (\wedge (x,y,z) , d) $, a quasi-isomorphism
\[
\eta : (\wedge (x,y,z) , d) \stackrel{\simeq}{\longrightarrow} (\wedge (e), 0 )
\]
defined by $\eta (x)=0$, $\eta (y)=e$, $\eta (z)=0$ provides a relative Sullivan model for $h^*$.
Finally, it suffices to consider a model for $\tilde{{\rm ev}}( \, \cdot \, , 1)$.
It is readily seen that the inclusion $\iota_2 : (\wedge (x,y,z), d) \to (\wedge (x,y,z,w), d)$ is a relative Sullivan model for $\tilde{{\rm ev}}( \, \cdot \, , 1)$.
Indeed, since $(\wedge (x,y,z,w), d)$ is acyclic, any CDGA morphism to $(\wedge (x,y,z,w), d)$ is null-homotopic.
Therefore, the sequence
\[
\xymatrix@R0pt{
0
&
1
&
2
&
\\
(\wedge (x, y) ,d)
\ar[r]^-{\iota_1}
&
(\wedge (x,y,z),d)
\ar[r]^-{\iota_2}
&
(\wedge (x,y,z,w),d)
\ar[r]^-{=}
&
\cdots
}
\]
is a Sullivan model for the Postnikov tower \eqref{Postnikov:PS^2}, and this coincides with $\theta (\iota)$.
\end{ex}

\begin{rem}
Let $P$ be the sequence whose pointwise rationalization is the tower (\ref{Postnikov:S_CP}). Then, it follows from the latter half of
Example \ref{ex:Postnikov:S_CP} that $A_{\rm PL}( \ )_* (P) \cong \Theta (i : S^2  \hookrightarrow {\mathbb C} P^2)$ in
${\rm Full}(\mathsf{et(\mathcal{D}_0^{\R_+}}))$. Then, we see that the sequence ${\bf L}\langle \ \rangle_*(\Theta (i))$
is isomorphic to the tower $(\ref{Postnikov:S_CP})$ and then $(\ref{Postnikov:S_CP})$ is a Postnikov tower for $i_\Q$.
In Example \ref{ex:Postnikov:S_CP}, we prove the fact with explicit constructions of structure maps of the tower.
In particular, we pay attention to show that $\bar{f}_* : \pi_k(S^2) \to \pi_k(X_3)$ is an isomorphism for $k\leq 3$ in the non-rational case.

In the same way, it is proved that the sequence (\ref{Postnikov:PS^2}) is a Postnikov tower for $p_\Q$ in Example \ref{ex:paths}.
\end{rem}

%%%%%%%%%%%%%%%%%%%%%%%%%%%%%%%%%%%%%%%%%%%%%%%%%%%%%%%%%

\section{The interleaving distance between maps in the homotopy category}\label{sect:Interleavings_maps}

We introduce a bivariate homotopy invariant for maps with the functor
$\Theta$ in Theorem \ref{thm:main} and
the interleaving distance in the homotopy category $\text{Ho}(\et{\mathsf{CDGA}})$.
%\todo{We first recall the definitions of an interleaving and the interleaving distance.}

\subsection{A pseudodistance on the homotopy set of maps}
For maps $f: X\to A$ and $g: Y \to B$ between path-connected spaces, we consider the interleaving distance
$
d_{\text{IHC}}(\Theta(f), \Theta(g))
$
in the homotopy category $\text{Ho}(\et{\mathsf{CDGA}})$.
We remark that $\Theta(f)$ is a cofibrant object in the category of extended tame persistence CDGAs; see Proposition \ref{prop:cofibrantOb}.
Thanks to this fact, we can compute the interleaving distance between $\Theta(f)$ and $\Theta(g)$ with homotopy.
%\todo{For $f,g:X\to Y$ and $h:Y\to Z$, is it true that $d_{IHC}(f,g)\geq d_{IHC}(h\circ f,h\circ g)$ in general ? Then $h_*:[X,Y]\to [X,Z]$ is continuous by persistent topology. If not, find a counter example !(Yam)}

%Let $p\mathsf{Ch}^*_\Q$ be the category of persistence cochain complexes over $\Q$ endowed with the projective model structure.
%Then, we see that \todo{Kuri : We have to check the assertion. (2/9): Do we need this comment?}
%the forgetful functor $U : \mathsf{CDGA} \to \mathsf{Ch}^*_\Q$ induces the functor $U_* : \text{Ho}(\et{\mathsf{CDGA}}) \to \text{Ho}(p\mathsf{Ch}^*_\Q)$.
%Therefore, the interleaving distance on $\text{Ho}(p\mathsf{Ch}^*_\Q)$ gives a lower bound for the distance $d_{\text{IHC}}$ on
%$\text{Ho}(\et{\mathsf{CDGA}})$.

While the proofs of the following results are standard, the assertions are fascinating.

\begin{prop}\label{prop:ID_equivalences} Let $f : X\to A$ and $g:Y \to B$ be maps between simply-connected spaces.

{\rm (i)} %Let $f : X\to A$ and $Y \to B$ be in $\text{\em Func}(I, \mathsf{Top})$. Then,
$d_{\text{\em IHC}}(\Theta(f), \Theta(g)) = 0$ if and only if  the rationalizations of maps $f$ and $g$ are isomorphic in
$\text{\em Func}(I, \text{\em Ho}(\mathsf{Top}))$; that is, one has a homotopy commutative diagram
\[
\xymatrix@C30pt@R12pt{
X_\Q \ar[r]^{f_\Q} \ar[d]_h^\simeq & A_\Q \ar[d]^k_\simeq\\
Y_\Q \ar[r]_{g_\Q} & B_\Q
}
\]
in $\mathsf{Top}$ in which vertical arrows are homotopy equivalences.

{\rm (ii)} %\todo{(12/16) `finite type' is removed. }
%For maps $f : X\to A$ and $Y \to B$,
 Suppose that $d_{\text{\em IHC}}(\Theta(f), \Theta(g)) < \infty$. Then,
the rational homotopy type of $X$ coincides with that of $Y$.
\end{prop}

\begin{proof} (i)  We use the same notation as in Remark \ref{rem:SullivanRep}. Let $\varphi_f$ and $\varphi_g$ be Sullivan representatives for $f$ and $g$, respectively. We choose minimal relative Sullivan algebras $i_f$ and $i_g$ in Remark \ref{rem:SullivanRep}.

The ``only if'' part is proved.
 For any $\delta >0$, the persistence CDGAs $(\lfloor \ \rfloor)^*\theta(i_f)$ and $(\lfloor \ \rfloor)^*\theta(i_g)$ are $\delta$-interleaved in the homotopy category and then $\theta(i_f)$ and $\theta(i_g)$ are $0$-interleaved in the homotopy category. %\todo{Consider the case where $\delta <\frac{1}{2}$. }
Therefore, we have a morphism  $\varphi : \theta(i_f) \to \theta(i_g)$ of persistence CDGAs, which is a weak equivalence.  The morphism gives rise to a commutative diagram
 % \todo{Kuri (3/26): We use $\theta(i_f)(i)$ but not $\theta(i_f)_i$ }
\begin{eqnarray}\label{eq:rational_type}
\xymatrix@C20pt@R20pt{
\theta(i_f)(0) \ar[r] \ar[d]_{\varphi(0)}  \ar@/^1.5pc/[rr]_{\varphi_f}& \text{colim}_n\theta(i_f)(n) \ar[d]^{\Phi:=\text{colim}\varphi(n)} \ar[r]^-{\eta_f}_-\simeq & \wedge V_X \ar@{.>}[d]^\psi & \hspace{-0.8cm}= \minmodel(X) \\
\theta(i_g)(0) \ar[r]  \ar@/_1.5pc/[rr]^{\varphi_g}& \text{colim}_n\theta(i_g)(n)  \ar[r]_-{\eta_g}^-\simeq & \wedge V_Y &  \hspace{-0.8cm}= \minmodel(Y)
}
\end{eqnarray}
of solid arrows in $\mathsf{CDGA}$ in which $\varphi(0)$ and $\Phi$ are isomorphisms on the rational cohomology.
%\todo{Naito(3.30) : $\varphi(0)$ instead of $\varphi_0$?}
By using homotopy inverses $\xi_f$ and $\xi_g$ of $\eta_f$ and $\eta_g$, respectively, we have a dots arrow $\psi$ with $\Phi\circ \xi_f \simeq \xi_g \circ \psi$. Then, we see that the right square is homotopy commutative.
%We may replace the right-hand side CDGAs with their minimal models.
By realizing the maps, % and the homotopies in $\mathsf{CDGA}$ which give the interleavings,
we have the result.

The ``if'' part follows from the fact that $\theta(i_f)$ is isomorphic to $\theta(i_g)$. In fact, in the diagram (\ref{eq:two_squares}), the small square with $\gamma'$ is homotopy commutative. Then, we see that the big square with $\eta$ is homotopy commutative. It follows from
\cite[Proposition 12.9]{FHT} that the big diagram itself is homotopy commutative. Observe that $\eta$ is a quasi-isomorphism.
The result \cite[Proposition 2.22]{FOT} makes the big diagram commutative replacing $\psi$ with an extension $\widehat{i_f\varphi}$ of
$i_f\circ \varphi$ which satisfies the condition that $\widehat{i_f\varphi}\simeq \psi$. The map $\psi$ is a quasi-isomorphism and hence so is
$\widehat{i_f\varphi}$.
Moreover, the map $\varphi$ is an isomorphism and two relative Sullivan algebras
$\wedge V\otimes \wedge W$ and  $\wedge V'\otimes \wedge W'$ are minimal. Therefore, it follows from \cite[Theorem 14.11]{FHT} that $\widehat{i_f\varphi}$ is an isomorphism.
We have $\Theta'(f) \cong \Theta'(g)$ as a persistence CDGA.

We observe that the equivalence between left and right homotopies in the proof above is used.

(ii)
%The finiteness of the relative Sullivan models for the maps enables us to deduce that the one of the interleavings which give the finite interleaving distance is a quasi-isomorphism between Sullivan models for $X$ and $Y$. This completes the proof.
%A proof without assuming the finiteness:
%\todo{Kuri (12/15): This is an answer to Yamaguchi's question on the finiteness condition.}
Let $\iota_f : \wedge V \to \wedge V\otimes \wedge W$ and $\iota_g : \wedge V \to \wedge V\otimes \wedge Z$ be relative Sullivan models for $f$ and $g$, respectively.
Let $(\varphi, \psi)$ be an $\e$-interleaving between $\theta(\iota_f)$ and
$\theta(\iota_g)$. Then, the morphisms $\varphi(i) : \wedge V\otimes \wedge W^{\leq i} \to  \wedge V\otimes \wedge Z^{\leq i+\e}$ of CDGAs induce
a morphism $\Phi :  \wedge V\otimes \wedge W \to  \wedge V\otimes \wedge Z$ of CDGAs. Observe that $\Phi\circ i_k = j_{k+\e} \circ \varphi(k) : \wedge V\otimes \wedge W^{\leq k} \to \wedge V\otimes \wedge Z$, where $i_k$ and $j_{k+\e}$ are the canonical inclusions. %We show that $\Phi$ is a quasi-isomorphism.

%\todo{Naito(3.30) : Mixing $H^*(\Phi)$ and $\Phi_:$ here.}
We show the surjectivity of $\Phi_*$. For an element $x$ in $H^n(\wedge V\otimes \wedge Z)$, there exists an element $u$ in
$H^n(\wedge V\otimes \wedge Z^{\leq n+1})$ such that $(j_{n+1})_*(u)= x$.
Let  $\alpha$ be the element $(i_{n+1+\e})_*\circ \psi(n+1)_*(u)$ in $H^n(\wedge V\otimes \wedge W)$.
Then, we see that
\begin{align}
\Phi_*(\alpha) &= (\Phi_* \circ  (i_{n+1+\e})_*\circ \psi(n+1)_*)(u) \\
&  =  ((j_{n+1+2\e})_* \circ \varphi(n+1+\e)_* \circ \psi(n+1)_*)(u) \\
& = ( ( j_{n+1+2\e})_* \circ \iota_g(n+1 < n+1+2\e)_* )(u) = ( j_{n+1})_*(u)=x.
\end{align}

The injectivity of $\Phi_*$ is proved as follows. Suppose that $\Phi_*(y) = 0$ for an element $y \in H^n(\wedge V\otimes \wedge W)$. Then, there exists an element $v$ in
$H^n(\wedge V\wedge W^{\leq n+1})$ such that $(i_{n+1})_*(v)= y$. The definition of $\Phi$ enables us to deduce that
$\Phi_*\circ (i_{n+1})_* = (j_{n+1+\e})_*\circ \varphi(n+1)_*$.  The map $(j_{n+1+\e})_*$ is an isomorphism on the cohomology $H^n(\wedge V\otimes \wedge Z^{\leq n+1+\e})$. Thus, we have $\varphi(n+1)_*(v)= 0$. This yields that
\begin{align}
0 &=((i_{n+1+2\e})_*\circ \psi(n+1+\e)_*\circ \varphi(n+1)_*)(v) \\
 &= ((i_{n+1+2\e})_*\circ \iota_f(n+1 < n+1+2\e)_* )(v) = (i_{n+1})_*(v)= y.
 \end{align}
As a consequence, we see that the morphism $\Phi$ is a quasi-isomorphism. The same argument as in (i) with the right-hand square in (\ref{eq:rational_type})
allows us to obtain the result.
\end{proof}

%\begin{rem}
%The circumstance in Proposition \ref{prop:ID_equivalences} (ii) is similar as in the category of $\R^n$-valued functions; see \cite[Section 5.1 and Remark 5.1]{L}.
%\end{rem}

When considering the distance in $\text{Ho}(\et{\mathsf{CDGA}})$ via the functor $\Theta$, our interest goes to maps whose domains coincide.

\begin{thm}\label{thm:Ho_sets}
The function $d : [X, Y] \times [X, Y] \to \R_{\geq 0}\cup \{\infty\}$ defined by $d([f], [g]) = d_{\text{\em IHC}}(\Theta(f), \Theta(g))$
is a well-defined pseudodistance on $[X, Y]$. Assume further that $X$ and $Y$ are simply connected. Then, $d([f], [g])=0$ if and only if there exist  self homotopy equivalences $\alpha$ and $\beta$ on $X_\Q$ and $Y_\Q$, respectively,  such that $\beta \circ f_\Q  \simeq  g_\Q \circ \alpha$.
\end{thm}

The following lemma is used in proving the theorem above.

%\begin{cor}
%For elements $f$ and $g$ in $\pi_n(X)$ with $d_{\text{\em IHC}}(\Theta(f), \Theta(g))=0$, one has $f_\Q = r \cdot g_\Q$ for some nonzero $r \in \Q$
%under the identification $\pi_n(X_\Q) \cong \text{\em Hom}(V_X, \Q)$.
%\end{cor}

%\begin{proof}
%A rational model $\psi$  for the self homotopy equivalence $\alpha$ in Theorem \ref{thm:Ho_sets} is an isomorphism. Then $\psi$  assigns $rv$ to a generator $v \in (V_{S^n})^n$ with some nonzero $r \in \Q$.  We have the result.
%\end{proof}

\begin{lem}\label{lem:well-definedness} Suppose that $f  : X \to Y$ is homotopic to $g : X \to Y$. Then, the persistence object $\theta(\iota_f)$ is isomorphic to $\theta(\iota_g)$, where $\iota_f$ and $\iota_g$ are minimal relative Sullivan models for $f$ and $g$, respectively.
\end{lem}

\begin{proof} We consider a diagram of solid arrows
\begin{eqnarray*}
\xymatrix@C60pt@R15pt{
 %& & \wedge V_Y \otimes \wedge W' \ar@/^1.5pc/[ddl]^{w'}_{\sim} \\
\wedge V_Y \ar[r]^-{\iota_f} \ar[d]_u^{\sim} \ar@/^1.5pc/[rr]^{\iota_g}& \wedge V_Y \otimes \wedge W \ar[d]^w_{\sim} \ar@{.>}[r]^\varphi&  \wedge V_Y \otimes \wedge W' \ar@/^1.0pc/[dl]^{w'}_{\sim} \\
A_{\rm PL}(Y)  \ar[r]_{A_{\rm PL}(f), \ A_{\rm PL}(g)} & A_{\rm PL}(X) &
}
\end{eqnarray*}
in which $A_{\rm PL}(f) \circ u \simeq A_{\rm PL}(g)\circ u$; see \cite[Proposition 12.6]{FHT}. The Lifting lemma \cite[Lemma 12.4]{FHT} allows us to obtain the map
$\varphi$ with $w' \circ \varphi \simeq w$. Observe that $\varphi$ is a quasi-isomorphism.
We see that $w'\circ (\varphi \circ \iota_f) \simeq w'\circ \iota_g$ and then
$\varphi \circ \iota_f \simeq \iota_g$.  The result \cite[Proposition 2.22]{FOT} implies that $\widehat{\iota_g} \circ \iota_f = \iota_g$ for some $\widehat{\iota_g}$ with $\widehat{\iota_g} \simeq \varphi$. Since $\widehat{\iota_g}$ is a quasi-isomorphism, it follows from \cite[Theorem 14.11]{FHT} and the minimality of  $\iota_f$ and $\iota_g$ that $\widehat{\iota_g}$ is an isomorphism and hence $\theta(\iota_f)$ is isomorphic to $\theta(\iota_g)$.
\end{proof}

\begin{proof}[Proof of Theorem \ref{thm:Ho_sets}] The well-definedness follows from  Lemma \ref{lem:well-definedness}.
Proposition \ref{prop:ID_equivalences} (i) implies the latter half of the assertion.
\end{proof}

%\begin{prop}
%For maps $f:X\to Y$, $g:X\to Y$ and  a map $h:Y\to Z$
%in simply connected spaces,
 %  $d_{\text{IHC}}(\Theta(h\circ f),\Theta(h\circ g))\leq d_{\text{IHC}}(\Theta(f),\Theta(g))$.
%\todo{Yam (1/11): It is proved inductively
 %    by Prop 10.4 of Griffiths-Morgan[13].
  %   See Good Note for the proof.}
%\end{prop}

%\begin{proof}
%If $\Theta(f)$ and $\Theta(g)$ are  $\epsilon$-interleaved by natural transformations $\psi$ and $\varphi$, then  $\Theta(h\circ f)$
%and $\Theta(h\circ g)$ are so by the same $\psi$ and $\varphi$
%since $M(Z)$ is a sub-DGA of $M(Y)$.
%\end{proof}

%\begin{cor}\todo{Yam (1/11)}
%When $X$, $Y$ and $Z$ is simply connected, the induced map by $h:Y\to Z$,   $h_*:[X,Y]\to [X,Z]$ is continuous.
%\end{cor}

\subsection{Interleaving distance to \(\Theta(id)\)}
\label{sect:IDtoTheta(id)}
% Let \(f\colon X\to Y\) be a map between simply connected spaces.
In this subsection, we investigate the distance
\(d_{\rm IHC}(\Theta(f\colon X\to Y), \Theta(id\colon X\to X))\) provided $f: X\to Y$ is relatively finite in the sense in Section \ref{sect:relativeSullivan_pCDGA}.
As a consequence, the distance is determined under some assumption. %\todo{Kuri(2/21) : Here is  `finite type'. }

Before investigating this distance, we recall the linear part of a Sullivan algebra and its basic properties, which will be used below.
For a Sullivan algebra $(\wedge V,d)$, the space of indecomposables of $\wedge V$ is defined by 
\[ 
Q(\wedge V) := \wedge^+V/(\wedge^+V)^2. 
\]
We identify $Q(\wedge V)$ with $V$ as graded vector spaces.
The differential $d$ induces a differential $ d_1:Q(\wedge V) \to Q(\wedge V)$, and hence $(Q(\wedge V),d_1)$ is a cochain complex. We call $d_1$ the linear part of $d$.
For a CDGA morphism $\varphi:(\wedge V,d)\to (\wedge W,d)$, the induced morphism $Q\varphi:Q(\wedge V) \to Q(\wedge W) $ is called the linear part of $\varphi$. 
We also recall the following basic property of the linear part, which will be used later.

\begin{prop}
[{\rm cf. \cite[Proposition 6.13(i)]{B-G} \cite[Proposition 14.13]{FHT}}]\label{prop:linear-part}
%[{\rm cf. \cite[Proposition 12.8(2), Theorem 14.9, Proposition 14.13]{FHT}}]\label{prop:linear-part}
\leavevmode
%\par\vspace{-1em}
\begin{enumerate}[label=(\arabic*)]
\item If $\varphi_0 \simeq \varphi_1 : (\wedge V,d)\to (\wedge W,d)$ are homotopic morphisms between Sullivan algebras and $(\wedge V,d)$ is simply connected, then $H(Q\varphi_0) = H(Q\varphi_1)$.
\item If $\varphi : (\wedge V,d)\to (\wedge W,d)$ is a morphism of Sullivan algebras, then $\varphi$ is a quasi-isomorphism if and only if $H(Q\varphi)$ is an isomorphism.
\end{enumerate}
\end{prop}

\begin{proof}
Since $(\wedge V,d)$ is simply connected, any homotopy between $\varphi_0$ and $\varphi_1$ is augmentation-preserving.
The assertion (1) follows from
\cite[Proposition~6.13(i)]{B-G}.
% By \cite[Theorem~14.9]{FHT}, we may write
% \[
%  (\wedge V,d)\cong (\wedge V',d)\otimes(E,d)
%  \quad\text{and}\quad
%  (\wedge W,d)\cong (\wedge W',d)\otimes(F,d),
% \]
% where $(\wedge V',d)$ and $(\wedge W',d)$ are minimal Sullivan
% algebras, and $(E,d)$ and $(F,d)$ are contractible Sullivan
% algebras. Thus, we have identifications
% \[
%  H(Q(\wedge V),d_1)\cong V'
%  \quad\text{and}\quad
%  H(Q(\wedge W),d_1)\cong W'
% \]
% since the linear part of a minimal Sullivan algebra has zero
% differential and the linear part of a contractible Sullivan algebra
% is acyclic.
% The morphisms $\varphi_0$ and $\varphi_1$ induce homotopic
% morphisms from $(\wedge V',d)$ to $(\wedge W',d)$.
% Since $H^1(\wedge V',d)=0$, \cite[Proposition~12.8(2)]{FHT} implies that
% their linear parts coincide. Under the identifications above,
% this gives
% $
%  H(Q\varphi_0)=H(Q\varphi_1).
% $
The assertion (2) is \cite[Proposition~14.13]{FHT}.
\end{proof}

Let \(f\colon X\to Y\) be a map between simply connected spaces.
Let \(\wedge V_X\) and \(\wedge V_Y\) be minimal Sullivan models for
\(X\) and \(Y\), respectively.
Denote by \(f^*\colon \wedge V_Y \to \wedge V_X\)
the model for the map \(f\colon X\to Y\).
Take the minimal relative Sullivan model
\(\eta\colon \wedge V_Y \otimes \wedge W \xrightarrow{\simeq} \wedge V_X\) for \(f^*\).
%Assume that \(W\) is finite dimensional,
%and
We observe that
\(\Theta(f)(i) = \wedge V_Y\otimes \wedge W^{\le i}\) and
\(\Theta(id)(i) = \wedge V_X\)
for each \(i\in\mathbb{Z}_+\).

We assume that the map $f$ is relatively finite. 
Let \(N\) be the maximum integer
such that \(W^N \neq 0\).

\begin{prop}
  \label{prop:f_id_N_interleaved} The extended tame persistence CDGAs
  \(\Theta(f)\) and \(\Theta(id)\) are \(N\)-interleaved in the homotopy category.
  In other words,
  \(d_{\rm IHC}(\Theta(f), \Theta(id)) \le N\).
\end{prop}
\begin{proof}
  Note that the quasi-isomorphism \(\eta\) is surjective
  since the linear part of \(\eta\) is surjective
  by \cite[Proposition 14.13]{FHT} and the minimality of \(\wedge V_X\).
  Hence \cite[Lemma 12.4]{FHT} implies that
  there is a morphism \(\rho\colon \wedge V_X \to \wedge V_Y\otimes\wedge W\) of CDGAs
  which makes the diagram
  \begin{equation}
    \xymatrix@C30pt@R20pt{
      & \wedge V_Y\otimes \wedge W \ar@{->>}[d]^{\simeq}_\eta \\
      \wedge V_X \ar[r]_{id} \ar@{.>}[ru]^\rho & \wedge V_X
    }
  \end{equation}
  commute strictly, i.e.\ \(\eta\rho = id\).
  Here \(\eta\) is a homotopy equivalence of CDGAs
  and \(\rho\) is the homotopy inverse of \(\eta\).
  Thus \(\rho\eta\simeq id\).
  %\todo{Kuri (3/26) : It seems that the homotopy equivalence follows from Whitehead theorem.}

  Now we define
  \(\varphi\colon \Theta(f) \Rightarrow \Theta(id)\circ T_N\) and
  \(\psi\colon \Theta(id) \Rightarrow \Theta(f)\circ T_N\)
  by the restrictions of \(\eta\) and \(\rho\), respectively.
  Note that \(\psi\) is well defined
  since \(\wedge V_Y\otimes\wedge W^{\le N} = \wedge V_Y\otimes\wedge W\).
  It is easy to show that
  the pair \((\varphi, \psi)\) is an \(N\)-interleaving in the homotopy category; see Remark \ref{rem:Homotopies}.
\end{proof}

Moreover, we can determine \(d_{\rm IHC}(\Theta(f), \Theta(id))\) explicitly under an assumption on the model for $f^*$.

\begin{prop}
  \label{prop:f_id_distance}
  Assume that any morphism
  \(\alpha\colon \wedge V_X \to \wedge V_Y\otimes \wedge W^{<N}\)
  has no homotopy retract.
  Then, it holds that
  \(d_{\rm IHC}(\Theta(f), \Theta(id)) = N\).
\end{prop}
\begin{proof}
  By \cref{prop:f_id_N_interleaved},
  it is enough to show that
  \(\Theta(f)\) and \(\Theta(id)\) are {\it not} \(\varepsilon\)-interleaved in the homotopy category
  for any \(\varepsilon<N\).
  Assume that there is such an \(\varepsilon\)-interleaving \((\varphi, \psi)\),
  where \(\varphi\colon \Theta(f)\Rightarrow\Theta(id)\circ T_{\varepsilon}\)
  and \(\psi\colon \Theta(id)\Rightarrow\Theta(f)\circ T_{\varepsilon}\).
  We may assume \(N-1\le\varepsilon<N\).
  Then, we have
  \begin{math}
    \psi(0)\colon
    \wedge V_X
    \to \wedge V_Y\otimes \wedge W^{\le \varepsilon}
    = \wedge V_Y \otimes \wedge W^{<N}
  \end{math}
  and
  \(\varphi(\varepsilon)\psi \simeq id\),
  which is a contradiction.
\end{proof}

%\todo{Naito added the computational example of $d_{\rm IHC}$}
\begin{ex}
%\todo{apply \cref{prop:f_id_distance}}
%Based on Proposition \cref{prop:f_id_distance},
By using Proposition \ref{prop:f_id_distance}, we determine the interleaving distance between $\Theta (*:S^{m}\to S^{m})$ and  $\Theta (id:S^{m} \to S^{m})$ in the homotopy category.

 (i) We prove that  the distance $d_{\rm IHC}(\Theta (*:S^{2n}\to S^{2n}), \Theta (id:S^{2n} \to S^{2n}))= 4n-1$.
Let $\minmodel (S^{2n})=(\wedge (x,y),d)$ be the minimal Sullivan model for $S^{2n}$, where $d(x)=0$, $d(y)=x^2$, $|x|=2n$.
The relative Sullivan model for $*$ is a relative Sullivan algebra of the form
\[
\iota_* : \wedge (x,y) \to \wedge (x,y) \otimes \wedge (z,w, x', y') \cong \wedge (x,y,z,w) \otimes \wedge (x', y').
\]
Here, $d(z)=x$, $d(w)=xz-y$ and $\wedge (x', y')$ is a copy of $\wedge (x,y)$;
 see also Subsection \ref{ex:p_q} for the Sullivan models.

Let $W= {\rm span}_{{\mathbb Q}}\{ z,w,x',y' \}$ and $N = 4n-1$.
Then, any morphism $\alpha : \wedge (x,y) \to \wedge (x,y)\otimes \wedge W^{<N} \cong \wedge (x,y)\otimes \wedge (z,w,x')$ has no homotopy retract.
Indeed, if there exists a homotopy retract $\beta$ of $\alpha$, then the induced composite on homology
\[
\xymatrix{
H(\wedge (x,y)) \ar[r]^-{H(\alpha)}
&
H(\wedge (x,y)\otimes \wedge (z,w,x'))
\ar[r]^-{H(\beta)}
&
H(\wedge (x,y))
}
\]
is equal to the identity map of $H(\wedge (x,y))$.
However, $H(\alpha)$ is trivial since $H(\wedge (x,y)) \cong {\mathbb Q}[x]/(x^2)$ and $H(\wedge (x,y)\otimes \wedge (z,w,x')) \cong {\mathbb Q}[x']$ as algebras, which contradicts $H(\beta)\circ H(\alpha) = id$.
Therefore, $d_{\rm IHC}(\Theta(*), \Theta(id) )= 4n-1$ from Proposition \ref{prop:f_id_distance}.

(ii) We show that $d_{\rm IHC}(\Theta (*:S^{2n+1}\to S^{2n+1}), \Theta (id:S^{2n+1} \to S^{2n+1}))=2n+1$.
The persistence CDGAs $\Theta(*)$ and $\Theta(id)$ have the forms
\begin{align}
\xymatrix@C20pt@R0pt{
&\hspace{-2.0cm} 0 & \hspace{-4cm} 2n-1 & \hspace{-5.5cm} 2n &\hspace{-5.5cm}  2n+1 \\
\ \Theta(*):&\wedge(x) =\cdots = \wedge(x) \  \ar@{>->}[r]& \wedge(x, y) \ \ar@{>->}[r]& \wedge(x, y, z) = \cdots,\\
\Theta(id):&\wedge(x)= \cdots =\wedge(x)\ar[r]&\wedge(x)\ar[r]&\wedge(x)= \cdots,
}
\end{align}
where $d(y)=x$, $d(z)=0$, $\deg(x)=\deg(z)=2n+1$ and $\deg(y)=2n$.
%Proposition  \ref{prop:f_id_N_interleaved} yields that  $d_{\rm IHC}(\Theta (*), \Theta (id)) \leq 2n+1$.
Since $H(\wedge (x, y), d)=\Q$, it follows that
every morphism $\alpha \colon \wedge(x) \to \wedge (x, y)$ has no homotopy retract.  By virtue of Proposition \ref{prop:f_id_distance}, we have the result.
\end{ex}

\begin{ex}\label{ex:CP_S_distance}
%\todo{Added the example (Naito)}
Let $i : S^2 \hookrightarrow {\mathbb C}P^2$ be the inclusion.
In this example, we determine the distance $d_{\rm IHC}(\Theta (i:S^{2}\to {\mathbb C}P^2), \Theta (id:S^{2} \to S^{2}))$.
Recall from Example \ref{ex:Postnikov:S_CP} the Sullivan representative $i^*$ for $i$ and its minimal relative Sullivan model $\wedge(x, z) \hookrightarrow \wedge (x,z, v,w)$ defined via the quasi-isomorphism $\eta : \wedge (x,z, v,w) \stackrel{\sim}{\longrightarrow} \wedge (x,y)$.

% A Sullivan representative for the inclusion is given by $i^* : \wedge (x,z) \to \wedge (x,y)$, $i^*(x)=x$ and $i^*(z)=xy$.
% Here, the CDGAs $\wedge(x,z)$ and $\wedge(x,y)$ are minimal Sullivan models of ${\mathbb C}P^2$ and $S^2$, respectively. Their differentials satisfy $d(x)=0$, $d(y)=x^2$, and $d(z)=x^3$, where $|x|=2$, $|y|=3$, and $|z|=5$.
% We have a minimal relative Sullivan model $j : \wedge(x, z) \to \wedge (x,z)\otimes \wedge (v,w)$  of $i^*$ via the quasi-isomorphism defined by
% \[
% \eta : \wedge (x,z)\otimes \wedge (v,w) \stackrel{\sim}{\longrightarrow} \wedge (x,y),
% \]
% where  $ \eta (v)=y$,  $\eta (w) = 0$, $\eta(z) = xy$, $dv=x^2$ and $dw = xy - z$; see Remark \ref{rem:SullivanRep}.
% \todo{Kuri (1/7) : Change $v$ to $y$ in the equation $dw= ..$.  }

We see that $\Theta (i)$ and $\Theta (id)$ are $3$-interleaved in the homotopy category. In fact, the interleaving $(\varphi, \psi)$ between the two persistence CDGAs is given as follows.
\begin{itemize}
\item $\varphi : \Theta (i) \Rightarrow \Theta (id) \circ T_{3}$ is a natural transformation defined by a restriction of $\eta$.
\item $\psi :  \Theta (id) \Rightarrow \Theta (i) \circ T_{3}$ is a natural transformation induced by
\[
\rho : \wedge (x,y)  \to \wedge (x,z, v, w), \ \rho (x)=x, \ \rho(y)=v.
\]
\end{itemize}
%\todo{Kuri: $\rho(y)=z$ is changed to $\rho(y)=v$. Is it correct? }
It is readily seen that $\varphi^3 \psi = \Theta(id)(* < *+6)$ by definition.
%\todo{Kuri (1/9): $\eta \rho = id$ should be $\varphi^3 \psi = \Theta(id)(* < *+6)$.}
We define a homotopy $\Phi : \wedge (x,z,v,w) \to \wedge (x,z,v,w)\otimes \wedge (t,dt)$ by
\[
\Phi (x) = x, \hspace{1em} \Phi (z) = z\otimes t + xv\otimes (1-t) - w \otimes dt, \hspace{1em} \Phi (v)=v, \hspace{1em} \Phi (w) = w\otimes t.
\]
%\todo{Kuri (1/7) :  $\Phi$ is not a homotopy in $\et{\mathsf{CDGA}}$. $\Phi(w)$ is {\it not} in $\wedge(x, z)\otimes \wedge (v, w)^{\leq 3}\otimes \wedge (t, dt)$; see Remark \ref{rem:Homotopies}. }
Then, the restriction of $\Phi$ gives rise to a homotopy $\widetilde{\Phi}: \Theta(i) \to \Theta(i)^6\otimes \wedge(t, dt)$
which yields $\Theta(i)(* < *+6) \simeq \psi^3\varphi$
%\todo{Kuri(1/9) : $\rho \eta$ should be $\psi^3\varphi$.}
in $\et{\mathsf{CDGA}}$; see Remark \ref{rem:Homotopies}. It follows that $(\varphi, \psi)$ is a $3$-interleaving between $\Theta (i) $ and $\Theta (id)$.

Next, we show that $\Theta (i)$ and $\Theta (id)$ are not $\varepsilon$-interleaved in the homotopy category for any $2 \leq \varepsilon < 3$.
Assume that $\Theta (i)$ and $\Theta (id)$ are $\varepsilon$-interleaved in the homotopy category.
Let $\varphi' : \Theta (i) \Rightarrow \Theta (id) \circ T_{\varepsilon}$ and $\psi' : \Theta (id) \Rightarrow \Theta (i) \circ T_{\varepsilon}$ be the natural transformation which give the $\e$-interleaving.
Then, we have the following commutative diagram in $\text{Ho}(\mathsf{CDGA})$:
\[
\xymatrix@C25pt@R12pt{
\Theta (id) (0) \ar[rr]^-{\Theta_1 ( 0  \leq  2\varepsilon)} \ar[rd]_-{\psi' (0)}&& \Theta (id) (2\varepsilon)
\\
& \Theta (i) (\varepsilon). \ar[ru]_-{\varphi' ( \varepsilon)} &
}
\]
By applying the homology functor to the diagram, we have the commutative diagram
\[
\xymatrix@C25pt@R12pt{
\Q[x]/(x^2) \ar[rr]^-{id}  \ar[rd]_-{H(\psi' (0))}&& \Q[x]/(x^2)
\\
& \Q[x]/(x^3) \ar[ru]_-{H(\varphi' (\varepsilon))} &
}
\]
in the category of commutative graded algebras.
However, any algebra map from $\Q[x]/(x^2)$ to $\Q[x]/(x^3)$ is trivial, which contradicts the commutativity of the diagram.
Therefore, we conclude that  $d_{\rm IHC}(\Theta (i), \Theta (id) )= 3$.

%At the end of Example \ref{ex:CP_S_distance}, we would like to point out that the interleaving distance $d_{\rm IHC}(\Theta (i), \Theta (id))$ cannot be determined by Proposition \ref{prop:f_id_distance}, since the assumptions of Proposition \ref{prop:f_id_distance} are not satisfied in this example.
It is worthwhile mentioning that  Proposition \ref{prop:f_id_distance} is not applicable when determining the interleaving distance $d_{\rm IHC}(\Theta (i), \Theta (id))$.
Indeed, the CDGA map $\alpha : \wedge (x,y)\to \wedge (x,z)\otimes \wedge (v)$, $\alpha (x)=x$, $\alpha (y)=v$ has a homotopy retract $\beta : \wedge (x,z)\otimes \wedge (v) \to \wedge (x,y)$ defined by $\beta (x)=x$, $\beta (z)=0$ and $\beta (v)=y$.
\end{ex}

When \(X=*\), the assumption in \cref{prop:f_id_distance} cannot be satisfied.
In this case, we have the following proposition,
which implies that the distance is strictly smaller than the integer \(N\) described before Proposition \ref{prop:f_id_N_interleaved}. 

\begin{prop}
  Let \(Y\) be a space with the base point \(*\in Y\),
  \(f\colon *\to Y\) the inclusion of the base point,
  and \(\wedge V_Y\) the minimal Sullivan model for \(Y\).
  Assume that \(f\colon *\to Y\) is relatively finite
  in the sense in Section \ref{sect:relativeSullivan_pCDGA},
  i.e.\ \(V_Y\) is finite dimensional.
  Then, it holds that
  \(d_{\rm IHC}(\Theta(f\colon *\to Y), \Theta(id\colon *\to *)) = N/2\).
\end{prop} %\todo{Kuri(2/21) : Here is  `finite type'. }
\begin{proof}
  In this case,
  the minimal relative Sullivan model
  \(\eta\colon \wedge V_Y\otimes \wedge W \xrightarrow{\simeq} \Q\) for \(f\colon *\to Y\)
  satisfies
  \(W^n = (V_Y)^{n+1}\) for \(n\in \mathbb N\).
  Recall that \(N = \max\{n \mid W^n\neq 0\}\).

  First, we show that
  \(\Theta(f)\) and \(\Theta(id)\) are \(N/2\)-interleaved in the homotopy category.
  Note that
  \(\Theta(f)(i) = \wedge V_Y\otimes \wedge W \simeq \Q\)
  and
  \(\Theta(id)(i) = \Q\)
  for any \(i\ge N\).
  Hence, the trivial maps
  \(\varphi\colon \Theta(f) \Rightarrow \Theta(id)\circ T_{N/2}\) and
  \(\psi\colon \Theta(id) \Rightarrow \Theta(f)\circ T_{N/2}\)
  give an \(\frac{N}{2}\)-interleaving \((\varphi, \psi)\) in the homotopy category.

  Next, we show that
  \(\Theta(f)\) and \(\Theta(id)\) are not \(\varepsilon\)-interleaved in the homotopy category
  for any \(\varepsilon < \frac{N}{2}\).
  Assume that there is such an \(\varepsilon\)-interleaving \((\varphi, \psi)\),
  where \(\varphi\colon \Theta(f)\Rightarrow\Theta(id)\circ T_{\varepsilon}\)
  and \(\psi\colon \Theta(id)\Rightarrow\Theta(f)\circ T_{\varepsilon}\).
  Since \(\Theta(id)(\varepsilon) = \Q\),
  the inclusion map \(\Theta(f)(0 < 2\varepsilon)\colon \wedge V_Y \to \wedge V_Y\otimes\wedge W^{<2\varepsilon}\)
  is null-homotopic
  and so is \(\wedge V_Y\to \wedge V_Y\otimes\wedge W^{<N}\).
  Hence, its linear part
  \((V_Y, 0) \to (V_Y\oplus W^{<N}, d)\)
  induces the trivial map in the homology. 
  This follows from Proposition \ref{prop:linear-part} (1). 
  %is also null-homotopic. 
  \todo{We use Proposition \ref{prop:linear-part} (1).}
  But this is a contradiction
  since \(\dim V_Y > \dim W^{<N}\).
  This completes the proof.
\end{proof}

\section{Examples where $d_{\rm IHC}$ is greater than $d_{\rm Coh}$
}
\label{section:sensitiveOne}
%\todo{Kuri(8/31): How about this title.}

The homology functor $H$ gives rise to a functor $H : \text{Ho}(\et{\mathsf{CDGA}}) \to \et{\mathsf{GA}}$ to the category of extended tame persistence graded algebras. We here recall the {\it cohomology interleaving distance} $d_{\text{CohI}}(F, G)$ of extended tame persistence CDGAs $F$ and $G$ defined by $d_{\rm CohI}(F, G):= d_{\rm I}(H(F), H(G))$, where $d_{\rm I}$ denotes the interleaving distance defined in $\et{\mathsf{GA}}$.

%Remark that it may hold that
%$d_{\text{CohI}}(\Theta (f),\Theta(g))\neq 0$
%for $f:X\to *$ and $g:Y\to *$ even if $H^*(X;\Q )\cong H^*(Y;\Q)$ as $\Q$-graded algebras.

%\todo{Kuri (3/26): We first consider the cohomology ID in this section.}
The cohomology interleaving distance also detects persistence of graded algebras in the sense, for example, that
$d_{\text{CohI}}(\Theta (X\to \ast),\Theta(Y \to \ast)) >  0$ even if $H^*(X;\Q )\cong H^*(Y;\Q)$ as $\Q$-graded algebras.

\begin{ex}%\todo{Kuri: I think that some references for the minimal models $M(X)$ and $M(Y)$ are needed.}
%\todo{Naito(3.30) : Notation for the connected sum}
Let $X$ be the formal space $(S^3\times S^8)\# (S^3\times S^8)$
and $Y$  a non-formal 11-dimensional manifold such that $\minmodel(Y)=(\wedge(x,y,z),d)$ with
$|x|=|y|=3, |z|=5$, $dx=dy=0$ and $dz=xy$.
Then $H^*(X;\Q )=\wedge (x,y)\otimes \Q[u,w]/I
\cong H^*(Y;\Q)$
with $|u|=|w|=8$ and $I=(xy,xu,yw,u^2,uw,w^2,xw-yu)$.
%\todo{Kuri: $xw-yu$?}
Note that the minimal Sullivan model for $X$ is given by
\begin{eqnarray}\label{eq:connected_sum}
\minmodel(X)=(\wedge (x,y,z,v_{7},v_{7}',u,w, v_9,v_9', v_9'', v_{10}, v_{10}', v_{10}'', \ldots), d)
\end{eqnarray}
with $dx=dy=du=dw=0$, $ dz=xy$, $dv_7=xz$, $dv_7'=yz$, $dv_9=xv_7$, $dv_9'=yv_7'$, $dv_9''=xv_7'+yv_7$, $dv_{10}=xu$,
$dv_{10}'=yw$ and $dv_{10}''=xw-yu, \ldots$; see \cite[Example 3.6]{FOT}.
Here, $|x|=|y|=3$, $|z|=5$ and $|v_i|=i$.
Then, we see that $d_{\text{CohI}} (\Theta (f),\Theta(g))>0$
for $f:X\to *$ and $g:Y\to *$.
Indeed, if the distance is zero, then the map  $\varphi(8) : H(\Theta (g))(8) \to H(\Theta (f))(8)$ induces a map from  $\Q[u]/(u^2)$ to
 $\Q[u]$ or $\Q[w]$ by degree reasons in the model (\ref{eq:connected_sum}), which is a contradiction.
\end{ex}

For the rest of this section, we give examples of maps $f, g : X\to A$ with
\[
d_{\text{IHC}}(\Theta (f), \Theta (g)) \gneqq d_{\rm CohI}(\Theta (f), \Theta (g)).
\]

\begin{prop}\label{ex:The_Hopf_map}%\todo{Kuri (11/21): The example is revised.}
For the trivial map $f=*:S^3\to S^2$ and the Hopf map $h:S^3\to S^2$,
one has
\begin{enumerate}
    \item[{\rm (1)}] $d_{\rm IHC}(\Theta (f), \Theta (h))=3$
    \item[{\rm (2)}]  $d_{\rm CohI}(\Theta (f), \Theta (h))=2$.
\end{enumerate}
\end{prop}

%\begin{ex}\label{ex:The_Hopf_map}
%For maps $f=*:S^3\to S^2$ and the Hopf map $g:S^3\to S^2$, we show
%(1) $d_{\text{IHC}}(\Theta (f), \Theta (g))=3$ and (2) $d_{\text{I}}(H(\Theta (f)), H(\Theta (g)))=2$.

\begin{proof} In the proof, we will only present non-trivial assignments for generators when defining a morphism of CDGAs. 

(1)
Let $\minmodel(S^2)=(\wedge (x,y),d)\to (\wedge (x,y,\bar{x},\bar{y},\tilde{y}),d)$ be the minimal relative Sullivan algebra for $f$ and $\minmodel(S^2)=(\wedge (x,y),d)\to (\wedge (x,y,\bar{x}), d)$ the minimal relative Sullivan algebra for $g$, where the differentials are defined by
%Here
$d(\bar{x})=x$,  $d(\bar{y})=y-x\bar{x}$ and $d(\tilde{y})=0$.
Then, the persistence CDGA  $\Theta(f)$ has the form
 $\wedge(x,y)$ on $[0,1)$,
$\wedge(x,y,\bar{x})$ on $[1,2)$,
$\wedge(x,y,\bar{x},\bar{y})$ on $[2,3)$ and
$\wedge(x,y,\bar{x},\bar{y},\tilde{y})$ on $[3,\infty)$;
see Remark \ref{rem:The_form}.
% On the other hand,
The persistence CDGA $\Theta(h)$ is defined by
$\wedge(x,y)$ on $[0,1)$ and
 $\wedge(x,y,\bar{x})$ on $[1,\infty)$.

For $\e \geq 3$, a part of an interleaving
$\varphi(i):\Theta(f)(i)\to \Theta (h)(i+\e)$
is defined by $\varphi(i)(\tilde{y})=y-x\bar{x}$ and sending the other generators to zero.
In addition, we define $\psi (i) :\Theta(h)(i)\to \Theta (f)(i+\e)$ by $\psi(i)(y)=\tilde{y}$ and sending the others to zero.
It follows that $\Theta(f)$ and $\Theta(h)$
are $\e$-interleaved in the homotopy category. In fact, for $i\geq 0$,
we have $\varphi^{\e}\circ \psi \simeq \Theta(h)(*<*+{2\e})$
 in $\et{\mathsf{CDGA}}$
by restricting the homotopy
$H:\wedge (x,y,\bar{x})\to \wedge (x,y,\bar{x})\otimes \wedge (t,dt)$
given by $H(\bar{x})= \bar{x}t$,
$H(x)= xt-\bar{x}dt$ and
$H(y)= y-x\bar{x}(1-t^2)$.
Moreover,  we have
$\psi^{\e}\circ\varphi\simeq \Theta(f)(*<*+{2\e})$ in $\et{\mathsf{CDGA}}$
%\todo{Kuri (1/9) : We have $\psi^\varepsilon \circ \varphi \simeq \Theta(f)(* < * +{2\varepsilon})$. }
 by the restriction of the homotopy
$H:\wedge (x,y,\bar{x},\bar{y},\tilde{y})\to \wedge (x,y,\bar{x},\bar{y},\tilde{y})\otimes \wedge (t,dt)$
given by $H(\bar{x})= \bar{x}t$,
$H(x)= xt-\bar{x}dt$,
$H(\bar{y})= \bar{y}t^2$,
$H(y)= yt^2+2\bar{y}tdt$ and
$H(\tilde{y})=\tilde{y}$; see Remark \ref{rem:Homotopies}.
%\todo{Kuri: We need a homotopy between the inclusion and the composite.}

Suppose that $2\leq \e <3$ and $(\varphi, \psi)$ is an $\e$-interleaving in the homotopy category.
Then, we see that
$\varphi(\e)\circ \psi(0)$ is homotopic to the structure map $\Theta(h)(0 < 2\e)$, which is the inclusion
$\wedge (x, y) \to \wedge (x, y, \bar{x})$. By applying the functor $H(Q( \ ))$ to the right triangle in (\ref{eq:interleaving}), we have
$H(Q(\Theta(h)(0 < 2\e)))(y) =y$; see Subsection \ref{sect:IDtoTheta(id)}. 
On the other hand,
since
$H(Q(\Theta(h)(0 < 2\e)))$ factors through $H(Q(\Theta(f)(\e)))=\Q$, it follows that
$H(Q(\Theta(h)(0 < 2\e)))(y) =0$, which is a contradiction.
%the  constant  map since $\Theta(f)(\epsilon)\simeq(\Q,0)$.
%On the other hand, $\Theta (0<2\epsilon)$ is  the relative model of
%the Hopf map.
%Thus it is not $\epsilon$-interleaved for $2\leq\epsilon <3$.
Hence, we obtain the result (1).
%sNext we show (2) $d_{\text{I}}(H(\Theta (f)), H(\Theta (g)))=2$.

(2)
Under the same notation as in the proof of (1),
the persistence CDGA $H(\Theta(f))$ is given by
$\Q[x]/(x^2)$ on $[0,1)$,
$\wedge(y)$ on $[1,2)$,
 $\Q$ on $[2,3)$ and
 $\wedge(\tilde{y})$ on $[3,\infty)$.
%On the other hand,
Moreover, the persistence CDGA $H(\Theta(h))$ has the form
given by
$\Q[x]/(x^2)$ on $[0,1)$ and
 $\wedge(y)$ on $[1,\infty)$.
For $\e \geq 2$, we define
$\varphi(i):H(\Theta(f))(i)\to H(\Theta (h))(i+\e)$ by $\varphi(i)(\tilde{y})=y$.
Another part of the interleaving $\psi (i) :H(\Theta(h))(i)\to H(\Theta (f))(i+\e)$ is
defined by $\psi(i)(y)=\tilde{y}$.
We see that $H(\Theta(f))$ and $H(\Theta(h))$ are $\e$-interleaved.
If $1\leq \e <2$, then
$\varphi(1+\e)\circ \psi(1)$ is the constant map since $H(\Theta(f))(\e)=\Q$.
On the other hand, $\Theta (1<1+2\e)$ is  the identity map.
This yields that $H(\Theta(f))$ and $H(\Theta(h))$ are not  $\e$-interleaved. % for $\varepsilon <2$.
We have the result (2).
\end{proof}

We consider another example which gives the inequality $d_{\text{CohI}} < d_{\text{IHC}}$.
Let ${\mathbb C} P^2 \# \overline{{\mathbb C} P^2}$ denote the connected sum of the complex projective plane and one with the opposite orientation. 
Observe that the minimal Sullivan model for the connected sum is of the form
\[
\wedge V =(\wedge (x_1,x_2, y_1 , y_2), d) \  \ \text{with} \ \ dx_i =0, dy_1=x_1^2 + x_2^2 \ \ \text{and} \ \  dy_2=x_1 x_2,
\]
where $|x_i|=2$ and $|y_i|=3$; see \cite[Example 3.8]{FOT}.

We consider a map $f_k : S^3 \to {\mathbb C} P^2 \# \overline{{\mathbb C} P^2}$ which corresponds to the dual of $y_k$ via the identification $\pi_3 \left( {\mathbb C} P^2 \# \overline{{\mathbb C} P^2} \right) \otimes {\mathbb Q} \cong {\rm Hom}(V^3 , {\mathbb Q})$ for $k=1,2$; see \cite[Theorem 15.11]{FHT}.
A Sullivan representative for $f_k$ is given by $f_k^* : \wedge V \to (\wedge (e),0)$, $f_k^*(x_i)=0$, $f_k^*(y_i)=e$ $(k=i)$ and $f_k^*(y_i)= 0$ $(k\neq i)$ with $|e|=3$.
Here, $(\wedge(e), 0)$ is a Sullivan model of $S^3$.
Then, we obtain the inclusion $\iota : \wedge V \to (\wedge V \otimes \wedge (\bar{x}_1, \bar{x}_2, \bar{y}_k) , d)$ as a minimal relative Sullivan model for $f^*_k$ via $\eta$ defined by 
\[
\eta : (\wedge V \otimes \wedge (\bar{x}_1, \bar{x}_2, \bar{y}_k) , d) \longrightarrow (\wedge(e), 0) , \ \eta (\bar{x}_i)=0, \ \eta (\bar{y}_k)=0,
\]
where the differential of $\wedge V \otimes \wedge (\bar{x}_1, \bar{x}_2, \bar{y}_j)$ is given by $d(\bar{x}_i)=x_i$, $d(\bar{y}_1)=y_2 - x_1 \bar{x}_2$ and $d(\bar{y}_2)=y_1 - x_1\bar{x}_1 - x_2 \bar{x}_2$; see Remark \ref{rem:SullivanRep}. 

\begin{prop}\label{ex:connectedsumCP2_distance}
One has
\begin{enumerate}
    \item[{\rm (1)}] $d_{\rm CohI}(\Theta(f_1), \Theta(f_2))=0$,
    \item[{\rm (2)}] $d_{\rm IHC}(\Theta(f_1), \Theta(f_2))=1$.
    %\item[{\rm (2)}] $d_{\rm IHC}(\Theta(f_1), \Theta(f_2))>0$.
\end{enumerate}
\end{prop}

To prove the proposition,
we consider automorphisms on \(\wedge V\) mentioned above.
The following lemma can be proved by a straightforward calculation.

\begin{lem}\label{lem:connectedsumCP2_aut}
  Any isomorphism \(\varphi\colon \wedge V \to \wedge V\) of Sullivan algebras
  can be written as
  \begin{align}
    \varphi(x_1) &= \lambda x_1 + \bar\lambda x_2,&
    \varphi(x_2) &= \sigma\bar\lambda x_1 + \sigma\lambda x_2,\\
    \varphi(y_1) &= (\lambda^2+\bar\lambda^2)y_1 + 4\lambda\bar\lambda y_2,&
    \varphi(y_2) &= \sigma\lambda\bar\lambda y_1 + \sigma(\lambda^2+\bar\lambda^2)y_2,
  \end{align}
  where
  \(\sigma\in\{\pm 1\}\) and
  \(\lambda, \bar\lambda \in \Q\) with \(\lambda^2\neq \bar\lambda^2\).
\end{lem}

By using the above lemma, we prove the following result,
which will be utilized in the proof of Proposition \ref{ex:connectedsumCP2_distance}.

\begin{lem}\label{lem:connectedsumCP2}
For any isomorphism $\varphi : \wedge V \to \wedge V$, the coefficient of $y_2$ in $Q\varphi(y_2)$ is nonzero. Here, $Q\varphi : V \to V$ is the linear part of $\varphi$.
\end{lem}
\begin{proof}
  By Lemma \ref{lem:connectedsumCP2_aut},
  the coefficient of \(y_2\) in \(Q\varphi(y_2)\) is \(\sigma(\lambda^2+\bar\lambda^2)\).
  Since \(\lambda,\bar\lambda\in\Q\) and \(\lambda^2\neq\bar\lambda^2\),
  we have \(\lambda^2+\bar\lambda^2>0\).
  Hence \(\sigma(\lambda^2+\bar\lambda^2)\neq 0\).
\end{proof}
% \begin{proof}
% The assertion is proved by contradiction.
% We may write
% \begin{align}
% &\varphi (x_1) = \lambda_1 x_1 + \lambda_2 x_2, \hspace{1em}
% \varphi (x_2) = \mu_1 x_1 + \mu_2 x_2
% \\
% &\varphi (y_1) = p_1 y_1 + p_2 y_2, \hspace{1em}
% \varphi (y_2) = q_1 y_1 + q_2 y_2
% \end{align}
% for some $\lambda_i , \mu_i , p_i , q_i \in {\mathbb Q}$.
% Assume that $q_2 = 0$.
% First, the compatibility of $\varphi$ with the differentials, $d\varphi = \varphi d$, implies that $\lambda_1 \mu_1 - \lambda_2 \mu_2 = 0$ and $\lambda_1 \mu_2 + \lambda_2 \mu_1 = 0$.
% We also see that $\lambda_1 \mu_2 - \lambda_2 \mu_1 \neq 0$ since $\varphi$ is an isomorphism.
% From $\lambda_1 \mu_2 + \lambda_2 \mu_1 = 0$ and $\lambda_1 \mu_2 - \lambda_2 \mu_1 \neq 0$, we have $\lambda_2 \mu_1 \neq 0$.
% Multiplying both sides of the equation $\lambda_1 \mu_1 - \lambda_2 \mu_2 = 0$ by $1 / \lambda_2 \mu_1$ yields $\lambda_1 = \lambda_2 \mu_2 / \mu_1$.
% Substituting this into the equation $\lambda_1 \mu_2 + \lambda_2 \mu_1 = 0$, we obtain $\mu_1^2 + \mu^2_2 = 0$. This leads to $\mu_1 =0$, which contradicts to $\lambda_2 \mu_1 \neq 0$.
% \end{proof}

We are ready to prove Proposition \ref{ex:connectedsumCP2_distance}.

\begin{proof}[Proof of Proposition \ref{ex:connectedsumCP2_distance}]%\todo{Kuri: The proof is revised with new terminology. Check the proof again.}
The assertion (1) is trivial since $H\Theta(f_1) \cong H\Theta(f_2)$ as persistence CGAs.

We prove the assertion (2). 
First, it is readily seen that $\Theta(f_1)$ and $\Theta(f_2)$ are $1$-interleaved in the homotopy category.
Indeed, consider CDGA morphisms
\begin{align}
\tilde\varphi  &: \wedge V \otimes \wedge (\bar{x}_1 , \bar{x}_2 , \bar{y}_1) \to \wedge V \otimes \wedge (\bar{x}_1 , \bar{x}_2 , \bar{y}_2) \quad \text{and}
\\
\tilde\psi &: \wedge V \otimes \wedge (\bar{x}_1 , \bar{x}_2 , \bar{y}_2) \to \wedge V \otimes \wedge (\bar{x}_1 , \bar{x}_2 , \bar{y}_1)
\end{align}
defined by
\begin{align}
&\tilde\varphi (x_i) = \tilde\psi (x_i)=x_i, \quad \tilde\varphi (\bar{x}_i) = \tilde\psi (\bar{x}_i)=\bar{x}_i, \quad \tilde\varphi(\bar{y}_1)=\bar{y}_2, \quad \tilde\psi (\bar{y}_2)=\bar{y}_1,
\\
&\tilde\varphi(y_1) = \tilde\psi (y_1) = y_2 + x_1 \bar{x}_1 - x_1 \bar{x}_2 + x_2 \bar{x}_2 \quad \text{and}
\\
&\tilde\varphi(y_2) = \tilde\psi (y_2) = y_1 - x_1 \bar{x}_1 + x_1 \bar{x}_2 - x_2 \bar{x}_2.
\end{align}
Since $\tilde\psi\circ\tilde\varphi=id$ and $\tilde\varphi\circ\tilde\psi=id$, and their restrictions on $\Theta(f_1)(t)$ and $\Theta(f_2)(t)$ are compatible with the structure maps, these morphisms give a $1$-interleaving between $\Theta(f_1)$ and $\Theta(f_2)$.

Assume that $\Theta(f_1)$ and $\Theta(f_2)$ are $\varepsilon$-interleaved in the
homotopy category for some $0\leq\varepsilon<1$.
Let $(\varphi,\psi)$ be an $\varepsilon$-interleaving between the persistence CDGA, where
\[
 \varphi : \Theta(f_1) \Longrightarrow \Theta(f_2) \circ T_\varepsilon,
 \qquad
 \psi : \Theta(f_2) \Longrightarrow \Theta(f_1) \circ T_\varepsilon.
\]
Since $\Theta (f_1)(0)=\Theta (f_2)(\varepsilon)=\wedge V$, the component $\varphi(0)$
is a morphism from $\wedge V$ to itself.
It is proved that $\varphi (0)$ is an isomorphism under the assumption. To this end, it suffices to show that $Q\varphi(0)$ is an isomorphism. 

A direct calculation of the cohomology of $Q(\Theta(f_1)(t))$ shows that the structure map $\Theta (f_1)(0\leq 2\varepsilon)$ induces an
isomorphism
\[
 H^3(Q(\wedge V),d_1)
 \stackrel{\cong}{\longrightarrow}
 H^3(Q(\Theta(f_1)(2\varepsilon)),d_1),
\]
where
$
 H^3(Q(\Theta(f_1)(t)),d_1)
 \cong 
 V^3
 =
 \Q \langle y_1,y_2\rangle
$
for $0 \leq t < 2$.
Since $\wedge V$ is simply connected, Proposition~\ref{prop:linear-part}(1)
and the interleaving condition
$\psi(\varepsilon)\circ\varphi(0)
 \simeq
 \Theta(f_1)(0\leq2\varepsilon)$
give
$H^3(Q\psi(\varepsilon))\circ H^3(Q\varphi(0))
 =
 H^3(Q(\Theta(f_1)(0\leq2\varepsilon))).
$
Since the right-hand side is an isomorphism,
$H^3(Q\varphi(0))$ is injective.
Moreover, since $H^3(Q\varphi(0))$ is an endomorphism of the finite-dimensional vector space $H^3(Q(\wedge V),d_1)$, its injectivity implies that it is an isomorphism.
Under the identification 
$
 H^3(Q(\wedge V),d_1) \cong V^3,
$
$H^3(Q\varphi(0))$ coincides with $
 Q\varphi(0)|_{V^3}:V^3\to V^3.
$
It follows that $Q\varphi(0)|_{V^3}$ is an isomorphism.

We next show that $Q\varphi(0)|_{V^2}:V^2\to V^2$ is an isomorphism. 
Suppose that it is not an isomorphism. 
Since
$Q\varphi(0)|_{V^2}$ is a non-isomorphic endomorphism of the two-dimensional
vector space $V^2=\mathbb{Q}\langle x_1,x_2\rangle$, 
\[
Q\varphi(0)(x_1)=\varphi(0)(x_1)
\quad
\text{and}
\quad
Q\varphi(0)(x_2)=\varphi(0)(x_2)
\]
 are linearly dependent in $V^2$.
Therefore, the two elements
\[
 \varphi(0)(dy_1)
 =
 \varphi(0)(x_1)^2+\varphi(0)(x_2)^2
\quad
\text{and}
\quad
 \varphi(0)(dy_2)
 =
 \varphi(0)(x_1)\varphi(0)(x_2)
\]
are linearly dependent in $\wedge^2 (V^2)$.

On the other hand, %the restriction of $d$ to $V^3$, $d:V^3 \to \wedge^2 (V^2)$, is injective.
since
$Q\varphi(0)|_{V^3}$ is an isomorphism, the elements 
$\varphi(0)(y_1)$ and $\varphi(0)(y_2)$ are linearly independent. 
The restriction of $d$ to $V^3$, $d:V^3 \to \wedge^2 (V^2)$, is injective. Then, 
it follows that $d\varphi(0)(y_1)$ and $d\varphi(0)(y_2)$ are also linearly independent.
However, since $\varphi(0)$ is a morphism of CDGAs, we have
$
 d\varphi(0)(y_i)=\varphi(0)(dy_i)
$
for $i=1,2$. This is a contradiction. Therefore,
$Q\varphi(0)|_{V^2}$ is an isomorphism.
Thus, $Q\varphi(0)$ is an isomorphism.
It follows that $\varphi(0)$ is an isomorphism; see also \cite[\S 15, Remark 2]{FHT}.

%%%%%%%% previous version
% In order to prove the assertion (2), we show that $\Theta(f_1)$ and $\Theta(f_2)$ are not $\varepsilon$-interleaved in the homotopy category for $0\leq \varepsilon < 1/2$.
% Assume that $\Theta(f_1)$ and $\Theta(f_2)$ are $\varepsilon$-interleaved in the homotopy category. Then, we have an $\e$-interleaving
% $(\varphi, \psi)$ with $\varphi:\Theta(f_1) \Rightarrow \Theta(f_2)\circ T_{\varepsilon}$ and $\psi:\Theta(f_2) \Rightarrow \Theta(f_1)\circ T_{\varepsilon}$.
% Since $\varphi$ is a natural transformation and the structure map $\Theta(f_1)(0\to \e)$ is the identity map, it follows that $\varphi (0)=\varphi(\varepsilon)$.
% Furthermore, the homotopy commutative diagrams
% \[
% \xymatrix@C15pt@R12pt{
% \Theta (f_1)(0) \ar[rr]^-{id} \ar[rd]_-{\varphi(0)} && \Theta (f_1)(2\varepsilon)
% \\
% & \Theta (f_2)(\varepsilon), \ar[ru]_-{\psi (\varepsilon)} &
% }
% \xymatrix@C15pt@R12pt{
% & \Theta (f_1)(\varepsilon) \ar[rd]^-{\varphi (\varepsilon)} &
% \\
% \Theta (f_2)(0) \ar[rr]^-{id} \ar[ru]^-{\psi(0)} && \Theta (f_2)(2\varepsilon)
% }
% \]
% show that $\varphi(0)$ is a quasi-isomorphism on $\wedge V = \Theta(f_1)(0) = \Theta(f_2)(\e)$.
% It follows from \cite[Proposition 12.10]{FHT} that $\varphi(0)$ is an isomorphism.

Now, the natural transformation $\varphi$ gives rise to the commutative diagram
\[
\xymatrix@C60pt@R20pt{
\Theta (f_1)(0) \ar[r]^-{\Theta (f_1)(0 \leq 2)} \ar[d]_-{\varphi(0)} & \Theta (f_1)(2) \ar[d]^-{\varphi(2)}
\\
\Theta (f_2)(\varepsilon) \ar[r]^-{\Theta (f_2)(\varepsilon \leq \varepsilon + 2)}  & \Theta (f_2)(\varepsilon + 2).
}
\]
Applying the homology functor to the linear part of the diagram above, we obtain the commutative diagram
\[
\xymatrix@C60pt@R20pt{
V \ar[r]^-{q_1} \ar[d]_-{Q\varphi(0)} & H\left( Q \left( \wedge V \otimes \wedge (\bar{x}_1 , \bar{x}_2 , \bar{y}_1 ) \right), d_1 \right) \ar[d]^-{H(Q\varphi(2))}
\\
V \ar[r]^-{q_2} & H\left( Q \left( \wedge V \otimes \wedge (\bar{x}_1 , \bar{x}_2 , \bar{y}_2 ) \right), d_1 \right),
}
\]
where $q_1$ and $q_2$ are quotient maps.
It is readily seen that $q_1 (y_2) = 0$ and then $H(Q\varphi(2))\circ q_1 (y_2) = 0$.
On the other hand, $q_2 \circ Q\varphi(0) (y_2) \neq 0$ from Lemma \ref{lem:connectedsumCP2}, which contradicts the commutativity of the diagram above.
\end{proof}

\begin{rem} %\todo{Kuri (3/26) : The case $\mathbb{C}$ is considered at the end of the section.}
  \label{rem:connectedsumCP2_overC}
  The proof of Lemma \ref{lem:connectedsumCP2} depends on the fact that
  the equation \(\lambda^2+\bar\lambda^2=0\) has no non-trivial solution in \(\Q\).

  If we take coefficients in \(\mathbb{C}\),
  the isomorphism \(\varphi\) defined by
  \begin{align}
    \varphi(x_1) &= x_1 + \sqrt{-1} x_2,&
    \varphi(x_2) &= \sqrt{-1} x_1 + x_2,\\
    \varphi(y_1) &= 4\sqrt{-1} y_2,&
    \varphi(y_2) &= \sqrt{-1} y_1
  \end{align}
  is a counterexample for Lemma \ref{lem:connectedsumCP2}.
  Moreover, the extension
  \begin{equation}
    \tilde\varphi\colon
    (\wedge V \otimes \wedge (\bar{x}_1, \bar{x}_2, \bar{y}_1) , d)
    \to
    (\wedge V \otimes \wedge (\bar{x}_1, \bar{x}_2, \bar{y}_2) , d)
  \end{equation}
  of \(\varphi\) given by
  \begin{align}
    \tilde\varphi(\bar{x}_1) &= \bar{x}_1 + \sqrt{-1}\bar{x}_2,&
    \tilde\varphi(\bar{x}_2) &= \sqrt{-1}\bar{x}_1 + \bar{x}_2,\\
    \tilde\varphi(\bar{y}_1) &= \sqrt{-1}\bar{y}_2 - \bar{x}_1\bar{x}_2
  \end{align}
  is an isomorphism of relative Sullivan algebras over \(\wedge V\).
  Hence
  \(d_{\rm IHC}(\Theta(f_1), \Theta(f_2))=0\) and
  Proposition \ref{ex:connectedsumCP2_distance} (2) does not hold
  with coefficients in \(\mathbb{C}\).
\end{rem}

%We consider the cohomology interleaving distance $d_{\text{CohI}} (\Theta (f),\Theta(g))$ for maps.

%%%%%%%%%%%%%%%%%%%%%%%%%%%%%%%%%%%%%%%%

\section{Formalities of a persistence CDGA}\label{section:formalities}

The purposes of this section are to introduce formalities of an object in $\et{\mathsf{CDGA}}$, and to compare them.
 %\todo{Kuri(3/25): The sentence is added. }
%We say that

\begin{defn}
A persistence CDGA $F$ in $\et{\mathsf{CDGA}}$ is $H$-{\it formal} if there exists a zig-zag sequence of weak equivalences between $F$ and $(H(F), 0)$ in
$\et{\mathsf{CDGA}}$.
\end{defn}

In order to define other formalities of extended tame persistence CDGAs, we consider the commutative diagram
\[
\xymatrix@C30pt@R15pt{
\et{\mathsf{CDGA}} \ar[d]_{\pi} \ar[r]^-{j} & \mathsf{CDGA}^{\mathbb{R}_+} \ar[d]^{q_*} \\
\text{Ho}(\et{\mathsf{CDGA}}) \ar@{.>}[r]_-{j_*}& \text{Ho}(\mathsf{CDGA})^{\mathbb{R}_+},
}
\]
where $\pi$ and $q$ are the localization maps, $j$ is the embedding and $j_*$ is unique map induced by $j$. Observe that
$q_*\circ j$ sends a weak equivalence to an isomorphism.

We recall interleaving distances defined in Section \ref{sect:ID}.
By virtue of \cite[Proposition 3.6]{B-S}, we have  %\todo{Kuri (3/26): The inequalities are added.}
\begin{prop}\label{prop:inequalities}
For objects $F$ and $G$ in $\et{\modelcat}$,
\begin{align}\label{eq:IHC}
d_{\rm HC}(j_*(F), j_*(G)) \leq d_{\rm IHC}(F,G) \leq d_{\rm HI}(F, G) \leq d_{\rm I}(F, G).
\end{align}
\end{prop}

By using these distances, we define formalities of a persistence CDGA.

\begin{defn}\label{defn:IHD_formality} % \todo{Kuri (1/24): General formalities are described.}
Let $F$ be an extended tame persistence CDGA. \\
(1) $F$ is $d_{\text{HI}}$-{\it formal} if $d_{\text{HI}}(F, H(F))=0$. \\
(2) $F$ is $d_{\text{IHC}}$-{\it formal} if $d_{\text{IHC}}(F, H(F))=0$.\\
(3) $F$ is $d_{\text{HC}}$-{\it formal} if $d_{\text{HC}}(j_*(F), j_*(H(F)))=0$.

Here $H(F)$ is regarded as a CDGA with the trivial differential.
\end{defn}

%Observe that each object in the model category $\mathsf{CDGA}$ in \cite{B-G} is fibrant. Then, by Proposition \ref{prop:composable},
By Remark \ref{rem:d_{HI}}, we see that
$d_{\text{HI}}$ is a pseudodistance on the class of objects in $\et{\mathsf{CDGA}}$. %; see Remark \ref{rem:d_{IHC}_and_d_{HC}}.
Therefore, the same arguments as in the proofs
of \cite[Proposition 2.10]{KNWY} and Lemma \cite[Lemma 2.13]{KNWY} enable us to deduce the following result.

%We here recall the {\it cohomology interleaving distance} $d_{\text{CohI}}(F, G)$ of extended tame persistence CDGAs $F$ and $G$ defined by
%$d_{\text{CohI}}(F, G):= d_{\text{I}}(H(F), H(G))$, where $d_{\text{I}}$ denotes the interleaving distance defined in the category of commutative graded algebras.

\begin{lem}\label{lem:An_isometry} Suppose that extended tame persistence CDGAs $F$ and $G$ are
$d_{\rm HI}$-formal. Then
$d_{\rm HI}(F, G) = d_{\rm CohI}(F, G)$. %Thus,  $H : (\C, d_{\C, H}) \to (\A, d_{\A})$ is an isometry provided every object in $\C$ is  $d_{\C, H}$-formal.
\end{lem}

As mentioned in the Introduction, every persistence cochain complex is $d_{\text{HI}}$-formal; see \cite[Proposition 3.5]{KNWY}. The result yields the
equality $d_{\text{CohI}} = d_{\text{IHC}}$ on the category of persistence cochain complexes; see \cite[Theorem 3.3]{KNWY}.  On the other hand,
Lemma \ref{lem:An_isometry}, Propositions \ref{ex:The_Hopf_map} and \ref{ex:connectedsumCP2_distance} imply that there exist non $d_{\rm HI}$-formal objects in $\et{\mathsf{CDGA}}$; see Remark \ref{rem:non-formalizable}. Thus, we are interested in relationships among the formalities in Definition \ref{defn:IHD_formality} and the $H$-formality.

\begin{prop}\label{prop:implications_formalities} %\todo{Kuri (1/16): We prove `equivalences' between formalities of persistence objects in a general setting.}
For an extended tame persistence CDGA $F$, consider the following implications
\[
\xymatrix@C45pt@R25pt{
 \text{$F$ is $H$-formal} \ar@{=>}[r]^-{(1)} & \text{$F$ is $d_{\text{\em HI}}$-formal}  \ar@{=>}[ld]_(0.56){(2)}  \\
 \text{$F$ is $d_{\text{\em IHC}}$-formal} \ar@{=>}[u]^-{(5)} \ar@{=>}@<1.2ex>[r]^-{(3)} % \ar@/_2.2pc/@{=>}[rr]^-{(4)}
 &  \text{$F$ is $d_{\text{\em HC}}$-formal.}  \ar@{=>}@<1.7ex>[l]_-{(4)} %\hspace{-0.5cm}\text{$f : X\to Y$ is formalizable}.
% & \ar@{=>}[r] &  \text{$X$ and $Y$ are formal}.
%\end{align}
}
\]

\medskip
%\begin{align}
%\[
%\xymatrix@C35pt@R10pt{
% \text{$F$ is $H$-formal} \ar@{=>}[r]^-{(1)} & \text{$F$ is $d_{\text{\em HI}}$-formal}  \ar@{=>}[ld]_(0.56){(2)} & \\
% \text{$F$ is $d_{\text{\em IHC}}$-formal} \ar@{=>}[u]^-{(6)} \ar@{=>}@<1.2ex>[r]_-{(3)}
% \ar@/_2.2pc/@{=>}[rr]^-{(4)} &  \text{$F$ is $d_{\text{\em HC}}$-formal}  \ar@{=>}@<1.7ex>[l]^-{(5)}&\hspace{-0.5cm}\text{$f : X\to Y$ is formalizable}.
% & \ar@{=>}[r] &  \text{$X$ and $Y$ are formal}.
%\end{align}
%}
%\]
%Here, the implication (5) holds provided $f$ is relatively finite.
\noindent
{\em (i)} The implications (1), (2) and (3) hold in general.  \\
{\em (ii)} Let $\{\tau_n\}_{n\geq 0}$ be a  sequence which discretises $F$. Suppose that
there exists a positive number $\alpha$ such that $\tau_{n+1} -\tau_n \geq \alpha$ for each $n$.
Then, the implications (4) and (5) hold.
\end{prop}

\begin{rem}
As a consequence of the proposition above, we see that
formalities in Definition \ref{defn:IHD_formality} and the $H$-formality are equivalent to one another under an appropriate assumption, for example,
in the case where the given persistence CDGA is tame; see Definition \ref{defn:et}. %\todo{Kuri (3/26) : A sufficient condition for the formalities to be equivalent. }
\end{rem}

The proof of Proposition \ref{prop:implications_formalities} is given after proving Proposition \ref{prop:d_{IHC}_formalizable} below,  which relates  the $d_{\rm IHC}$-formality to
the {\it formalizability} of a map in the sense of Thomas; see \cite[V]{Thomas}.

\begin{defn}\label{defn:formalizability}
A morphism $\alpha : (B, d_B) \to (C, d_C)$ of CDGAs is {\it formalizable} if there exists a homotopy commutative diagram
\begin{eqnarray}\label{eq:formalizability}
\xymatrix@C30pt@R15pt{
(B, d_B) \ar[r]^{\alpha} & (C, d_C) \\
\wedge V_B \ar[r]^{m_\alpha} \ar[u]^{m_B}_\sim  \ar[d]_{\theta_B}^\sim & \wedge V_C  \ar[u]_{m_C}^\sim  \ar[d]^{\theta_C}_\sim \\
H^*(B) \ar[r]^{H(\alpha)} & H^*(C)
}
\end{eqnarray}
in which vertical arrows are quasi-isomorphisms and $\wedge V_B$ and $\wedge V_C$ are minimal models for $B$ and $C$, respectively.

A continuous map $f : X\to Y$ between path-connected spaces is {\it formalizable} if the morphism $A_{\rm PL}(f) : A_{\rm PL}(Y) \to A_{\rm PL}(X)$ of CDGAs is formalizable. A path-connected space $X$ is {formal} if the trivial map $X \to \ast$ is formalizable.
%\todo{Remark: $PX\to X$ is formalizable iff $X$ is formal. Kuri (3/23) : We define a formal space.}
\end{defn}

In the original definition \cite[V.3 (iii)]{Thomas}, it is required that $H(m_B) = H(\theta_B)$.
This requirement is satisfied by connecting the diagram (\ref{eq:formalizability}) with
the diagram obtained by applying the homology functor to (\ref{eq:formalizability}). Moreover, even if $m_\alpha$ in (\ref{eq:formalizability})
is replaced with a relative Sullivan model $\iota \colon \wedge V_B \to \wedge V_B\otimes W$ for $\alpha$, we have a homotopy commutative diagram
(\ref{eq:formalizability}) with a lift $m_\alpha$ which satisfies the condition that $\eta \circ m_\alpha \simeq \iota$ for a minimal model
$\eta : \wedge V_C \stackrel{\sim}{\to}  \wedge V_B\otimes W$.

%The following proposition explains relationships among the formalizablity of a map $f$ and formalities of the persistence CDGA $\Theta(f)$ described in Definition \ref{defn:IHD_formality}.

\begin{rem}\label{rem:PX}
Let $p : PX \to X$  be the path fibration on a formal space $X$. For $\alpha = A_{\rm PL}(p)$,
we choose a Sullivan representative as the map $m_\alpha$ in the diagram (\ref{eq:formalizability}) above. Since $H(A_{\rm PL}(PX)) \cong \Q$, it follows that
the lower square in (\ref{eq:formalizability}) is commutative. Thus, we see that $p : PX \to X$ is formalizable.
\end{rem}

\begin{prop}\label{prop:d_{IHC}_formalizable}
Let $F$ be the extended tame persistence CDGA $\Theta(f)$ for a map $f : X\to Y$. If $\Theta(f)$ is $d_{\rm IHC}$-formal, then the map $f$ is formalizable.
\end{prop}

\begin{proof}
%In order to prove that the implication (4) holds,
We consider a relative Sullivan model $C(i) \to C(i+1)$ for $H(\Theta(f))(i) \to  H(\Theta(f))(i+1)$ for each $i\geq 0$. We regard $C$ as a persistence object in $ \mathsf{CDGA}^{(\mathbb{Z}_+, \leq)}$.
%\todo{Naito(3.30) : Change $ \mathsf{CDGA}^{(\mathbb{N}, \leq)}$ to $ \mathsf{CDGA}^{(\mathbb{Z}_{+}, \leq)}$ ?}
Then,
we have a cofibrant replacement $\gamma : C':=(\lfloor \ \rfloor )_*C \to H(\Theta(f))$ of $H(\Theta(f))$ in $\et{\mathsf{CDGA}}$.
Since $H(\Theta(f))$ is weakly equivalent to $C'$, it follows from the assumption that
$0=d_{\text{IHC}}(\Theta(f), H(\Theta(f))) =d_{\text{IHC}}(\Theta(f), C')$.
This implies that $\Theta(f)$ and $C'$ are $\delta$-interleaved in the homotopy category for $\delta < \frac{1}{2}$.

Let $(\varphi, \psi)$ be the $\delta$-interleaving of $\Theta(f)$ and $C'$. Since $2\delta <1$, it follows that $\psi(n) =\psi^\delta(n)$ and  $\psi^\delta(n) \circ \varphi(n) \simeq id_{\Theta(f)(n)}$ for each non-negative integer $n$.
Moreover, we see that $\varphi(n) =\varphi^\delta(n)$ and  $\varphi^\delta(n) \circ \psi(n) \simeq id_{C'(n)}$.
These facts imply that $\varphi$ gives rise to a morphism $\widetilde{\varphi} : \Theta(f) \to C'$ of persistence CDGAs for which the component
$\widetilde{\varphi}(n)$ is a quasi-isomorphism for $n\geq 0$. Thus, we have a commutative diagram
\begin{eqnarray}\label{eq:formality1}
\xymatrix@C55pt@R15pt{
 \Theta(f)(n) \ar[r]^-{\Theta(f)(n< n+1)} \ar[d]_{\widetilde{\varphi}(n)}^\sim & \Theta(f)(n+1) \ar[d]^{\widetilde{\varphi}(n+1)}_\sim \\
 C'(n) \ar[r]_-{C'(n < n+1)} & C'(n+1)
}
\end{eqnarray}
for each non-negative integer $n$.
The commutativity of the diagrams yields the commutative diagram
\begin{eqnarray}\label{eq:formality2}
\xymatrix@C10pt@R12pt{
\wedge V \ar[r]\ar[d]_{\widetilde{\varphi}}^\sim & \wedge V\otimes \wedge (W^{\leq m}) \ar[r]^-{\iota_m} \ar[d]_{\widetilde{\varphi}(m)}^\sim & \text{colim}_n(\wedge V\otimes \wedge (W^{\leq n}))  \ar[d]^{k_1} \\
C'(0) \ar[r] \ar[d]_{\gamma(0)}^\sim & C'(m)  \ar[r]^{i_m} \ar[d]_{\gamma(m)}^\sim & \text{colim}_nC'(n) \ar[d]^{k_2} \\
H^*(\wedge V) \ar[r] &H^*(\wedge V\otimes \wedge (W^{\leq m}))   \ar[r]^-{H^*(i)}  & H^*(\wedge V \otimes  \wedge W),
}
\end{eqnarray}
%\todo{Naito(3.30) : Need explanation for $k_1$}
where $\iota_m$ and $i_m$ are the canonical maps and
 $k_1$ is the evident map given by maps $\widetilde{\varphi}(m)$. Moreover, the morphism $k_2$ of CDGAs is induced by the composites
\[
C'(n)  \to H^*(\Theta(f)(n)) \to H^*(\text{colim}_n\Theta(f)(n))=H^*(\wedge V\otimes \wedge W)
\]
of natural maps.  Observe that $\Theta(f)(n) = \wedge V\otimes \wedge (W^{\leq n})$ and $\text{colim}_n(\wedge V\otimes \wedge (W^{\leq n}))=\wedge V\otimes \wedge W$.
To complete the proof, it suffices to show that the composite $k_2 \circ k_1$ is a quasi-isomorphism. We consider $H^\ell(k_2\circ k_1)$ on the $\ell$th cohomology. Then, for a sufficiently large integer $m  (> \! >\ell)$, it follows that $H^\ell(\iota_m)$, $H^\ell(i_m)$ and $H^\ell(i)$ are isomorphisms. We have the result.
\end{proof}

\begin{proof}[Proof of Proposition \ref{prop:implications_formalities}]
(i) By definition, the $H$-formality implies the $d_{\text{HI}}$-formality.
Proposition \ref{prop:inequalities} enables us to conclude that the implications (2) and (3) hold.
%The existences of the functor $\pi$ and $j_*$ yield the implications (2) and (3), respectively.

(ii) As for the implication (5), by applying the same argument as in obtaining the commutative diagram (\ref{eq:formality1}) in the proof of
Proposition \ref{prop:d_{IHC}_formalizable}, we have a commutative diagram
\begin{eqnarray}\label{eq:formality3}
\xymatrix@C45pt@R12.5pt{
 F(\tau_n) \ar[r]^-{F(\tau_n< \tau_{n+1})} \ar[d]_{\widetilde{\varphi}(\tau_n)}^\sim & F(\tau_{n+1})
 \ar[d]^{\widetilde{\varphi}(\tau_{n+1})}_\sim \\
 C'(\tau_n) \ar[r]_-{C'(\tau_n< \tau_{n+1})} & C'(\tau_{n+1}).
}
\end{eqnarray}
Here $C'$ is a cofibrant replacement of $H(F)$; see the paragraph after Proposition \ref{prop:cofibrantOb}.
Observe that we use a positive number $\delta$ with $2\delta < \alpha$ for obtaining the diagram (\ref{eq:formality3}).
This yields that $F$ is $H$-formal.

We prove that the implication (4) holds. The same argument as above enables us to obtain the diagram (\ref{eq:formality3})
which is homotopy commutative.  Here, the equalities in the argument are replaced with those up to homotopy. We observe that the $\delta$-interleaving $(\varphi, \psi)$ of $F$ and $C'$ is considered in
the category $\text{Ho}(\mathsf{CDGA})^{([0, \infty), \leq))}$.

By applying \cite[Proposition 2.22 (Extension of homotopies)]{FOT} to the composite $C'(0< \tau_0)\circ \widetilde{\varphi}(0)$, we have a strictly commutative diagram  (\ref{eq:formality1}) for $n = 0$ replacing $\widetilde{\varphi}(\tau_0)$ with $\widehat{\varphi}(\tau_0)$ which is homotopic to $\widetilde{\varphi}(\tau_0)$.
We observe that $\widehat{\varphi}(\tau_0)$ is a quasi-isomorphism. An inductive argument with the replacement enables us to obtain a strictly commutative diagram
\begin{eqnarray*}%\label{eq:formality3}
\xymatrix@C45pt@R12.5pt{
 F(\tau_n) \ar[r]^-{F(\tau_n< \tau_{n+1})} \ar[d]_{\widehat{\varphi}(\tau_n)}^\sim & F(\tau_{n+1})
 \ar[d]^{\widehat{\varphi}(\tau_{n+1})}_\sim \\
 C'(\tau_n) \ar[r]_-{C'(\tau_n< \tau_{n+1})} & C'(\tau_{n+1})
}
\end{eqnarray*}
in which the vertical arrows are quasi-isomorphisms. The family $\{\widetilde{\varphi}(0), \widehat{\varphi}(\tau_n)\}_{n\geq 0}$ gives rise to a weak equivalence between $F$ and $H(F)$ in $\et{\mathsf{CDGA}}$. This yields that $F$ is $d_{\text{IHC}}$-formal.
%Consider the morphism $\widetilde{\varphi}(0)$.
% By assumption,  the map $f$ is relatively finite. Therefore, we may assume that
%$C'(m) = \text{colim}_nC'(n)$ and $\wedge V\otimes \wedge (W^{\leq m}) =\wedge V\otimes \wedge W$
%for some sufficient large integer $m$. Moreover, it follows that the diagram (\ref{eq:formality2}) is homotopy commutative. This completes the proof.
\end{proof}

%%%%%%%%%%%

\begin{lem}\label{lem:H-formality}
The persistence CDGA $\theta(f)$ is $H$-formal if and only if $\theta(f)(n< n+1) : \theta(f)(n) \to \theta(f)(n+1)$ is formalizable for each $n\geq 0$.
\end{lem}

\begin{proof}
%\todo{Kuri (1/2): A complete proof will be written.}
To prove the `only if' part,  we apply \cite[Lemma A.1]{K2024} to a zig-zag of weak equivalences connecting $\theta(f)$ and $H(\theta(f))$.
Then, we have a commutative diagram
\begin{eqnarray*}%\label{eq:U3}
\xymatrix@C65pt@R15pt{
\theta(f)(n)  \ar[d]_\sim \ar[r]^-{\theta(f)(n<n+1)} & \theta(f)(n+1)  \ar[d]^\sim \\
H( \theta(f)(n))   \ar[r]_-{H(\theta(f)(n<n+1))} &  H( \theta(f)(n+1))% \ar@/_10.0pc/[uu]_\lambda^\simeq
} % \ar@/^1.0pc/[rr]^(0.5){f}  \ar@{>->}[llu]|(.4)\hole_(.6){i_0}
\end{eqnarray*}
for each $n\geq 0$.
By using a Sullivan representative of $\theta(f)(n<n+1)$, we have a homotopy commutative diagram as (\ref{eq:formalizability}), which allows us to deduce that
$\theta(f)(n< n+1) : \theta(f)(n) \to \theta(f)(n+1)$ is formalizable.

In order to prove the `if' part, we lift the Sullivan representative $\alpha_n : \minmodel_n \to \minmodel_{n+1}$ of $\theta(f)(n<n+1)$ in the homotopy commutative diagram, which gives the formalizability,
to a relative Sullivan algebra $\iota_n : \minmodel_n \to \minmodel_n \otimes \wedge W_n$ with a quasi-isomorphism
$\eta_{n+1} :  \minmodel_n \otimes \wedge W_n \stackrel{\sim}{\to}  \minmodel_{n+1}$. Thus, we have $\eta_{n+1} \circ \iota_n = \alpha_n$. The same inductive argument as in the proof of Proposition \ref{prop:implications_formalities}(5) with \cite[Proposition 2.22]{FOT} enables us to obtain the result.
\end{proof}

\begin{ex}\label{Hopf-f} %\todo{Kuri(2/9) move to section \ref{section:formalities}.
%\\Kuri (11/21) This is added. The map $g$ is not formalizable. Then, the result follows from Proposition \ref{prop:implications_formalities}}
Since the Hopf map $h : S^3\to S^2$ is not formalizable, it follows from Propositions \ref{prop:implications_formalities} and \ref{prop:d_{IHC}_formalizable}
that $\Theta(h)$ in Proposition  \ref{ex:The_Hopf_map} is not $d_{\text{HI}}$-formal. Here,
it is proved that $d_{\text{IHC}}(\Theta(h), H(\Theta(h)))= 1$.   % we determine the distance $d_{\text{IHC}}(\Theta(g), H(\Theta(g)))$, which is not non-zero.
The same argument as in Proposition  \ref{ex:The_Hopf_map} yields that $d_{\text{IHC}}(\Theta(h), H(\Theta(h)))\leq1$. We observe that
a cofibrant replacement $C$ of $H(\Theta(h))$ is given by $\wedge(x, y)$ on $[0, 1)$ and $\wedge(x,y,\bar{x},\bar{y},\tilde{y})$ on
$[1, \infty)$.

In order to prove $d_{\text{IHC}}(\Theta(h), H(\Theta(h)))\geq 1$, suppose that $\Theta(h)$ and $H(\Theta(h))$ are $\delta$-interleaved in the homotopy category for some $\delta < 1$ with a $\delta$-interleaving $(\varphi, \psi)$ for which $\psi : \Theta(h) \to C^\delta$ is a morphism in
$\text{Ho}(\et{\mathsf{CDGA}})$. Then, since  $\Theta(h)(0) \to  \Theta(h)(2\delta)$ is the identity, by applying the functor $H(Q( \  ))$,
it follows that $\psi(0)(y) = \lambda y$ in $C(\delta)$ for some nonzero rational number $\lambda$. Therefore, the naturality of $\psi$ implies
that $\psi(1+\delta)(y) = \lambda y$
in $C(1+2\delta)$. Since $\varphi(1+2\delta)\circ \psi(1+\delta) \simeq id_{\wedge (x, y, \overline{y})}$, we see that  $(H(Q( \varphi(1+2\delta)))\circ H(Q(\psi(1+\delta))))(y)  = y$.
However, we have  $y = 0$ in $H(Q(C(1+ 2\delta)))$, which is a contradiction.
\end{ex}

%\begin{rem}\label{rem:EM-models}
By combining Lemma \ref{lem:H-formality} with the following result due to Thomas \cite{Thomas}, we have a necessary and sufficient condition for $\theta(f)$ to be $H$-formal with {\it Eilenberg--Moore models} (E.M. models) for $\theta(f)$ and $H(\theta(f))$. An E.M. model is regarded as a
relative Sullivan algebra in which {\it lower degrees} are considered explicitly; see \cite[II. 2]{Thomas} for more details.
 %; see \cite[V.6. Proposition]{Thomas}.
%\todo{Kuri (1/2):  This fact will be proved with Lemma \ref{lem:H-formality}.}
%To see this, we recall a result in \cite{Thomas}

\begin{prop}\label{prop:Thomas'result} {\rm (cf. \cite[V.6. Proposition]{Thomas})} Let $\alpha : (B, d_B) \to (C, d_C)$ be the morphism of CDGAs in the upper homotopy commutative diagram
in (\ref{eq:formalizability}).  Then, the following conditions are equivalent. \\
(i) $\alpha$ is formalizable. \\
(ii) There exists an isomorphism  $(id, \Psi, \overline{\Psi})$ between an E.M. model $(\mathcal{F}, \eta)$ of $m_\alpha$ and an E.M. model
$(\mathcal{F}', \eta')$ of $H(\alpha)\circ \theta_B$ for some quasi-isomorphism $\theta_B : \minmodel_B \to H^*(B)$ with $H(\theta_B) = H(m_B)$ such that $(\Psi -id)$ strictly lowers the filtration degree and $(\eta')^*\circ \Psi^* = H(m_C)\circ \eta^*$.
\end{prop}
%\end{rem}

%Let $F$ and $G$ be formal persistence CDGA's.
%We have a sequence of functors $\et{\mathsf{CDGA}} \stackrel{\pi}{\to} \text{Ho}(\et{\mathsf{CDGA}}) \stackrel{H^*}{\to} \mathsf{etCGA}$. This enables us to obtain inequalities
%\begin{align}
%d_{\text{I}, \mathsf{etCGA}}(HF, HG) \leq d_{\text{IHC}}(F, G) =  d_{\text{IHC}}(HF, HG) &\leq
%d_{\text{I}, \et{\mathsf{CDGA}}}(HF, HG) \\
%&= d_{\text{I}, \mathsf{etCGA}}(HF, HG).
%\end{align}
%As a consequence, we have
%\begin{prop}\label{prop:formal_etCDGAs}
%$d_{\text{\em IHC}}(F, G) = d_{\text{\em I}, \mathsf{etCGA}}(HF, HG)$.
%\end{prop}

\begin{ex}\label{ex:toy}%(A toy example)
The persistence CDGA $\Theta(p : PS^{2n}\to S^{2n})$ in Example \ref{ex:p_q} below is not $H$-formal.
To see this, suppose that $\theta(p)$ is $H$-formal. Then, by Lemma \ref{lem:H-formality},  we see that the morphism $\alpha :=\theta(p)(0 < 2n-1): \wedge(x, y) \to \wedge (x, y, z)$ is formalizable. Observe that $\alpha$ is an E.M. model for $\alpha$ itself. We have a quasi-isomorphism $\theta_B :  \wedge(x, y) \to H(\wedge(x, y), d)=\Q[x]/(x^2)$ with $H(\theta_B) = id$. For the trivial map $t : H(\wedge(x, y), d)=\Q[x]/(x^2) \to H(\wedge (x, y, z))=\wedge (xz-y)$, the composite $t\circ \theta_B$ admits an E.M. model of the form
$ \wedge(x, y) \to \wedge (x, y, z, w)\otimes \wedge (z')$ with the quasi-isomorphism $\eta' : \wedge (x, y, z, w)\otimes \wedge (z')\to \wedge (xz-y)$ defined by
$\eta(u) = 0$ for $u = x, y, z, w$ and $\eta'(z') = xz-y$. It follows from Proposition \ref{prop:Thomas'result} that these E.M. models are  isomorphic to each other, which is a contradiction.

We see that $p : PS^{2n}\to S^{2n}$ is formalizable; see Remark \ref{rem:PX}.
Thus, it follows
%from Assertion \ref{assertion:theConnectedSum}
%\todo{Fix "??" Delete Assertion ??}
that the
converse of Proposition \ref{prop:d_{IHC}_formalizable} does not hold in general.
\end{ex}

In order to describe an example of an extended tame persistence CDGA $\Theta(f)$ which is $H$-formal, we recall terminology from
 \cite{Thomas1}.

A fibration $f:X\to Y$, whose fiber $F$ and base are simply connected, is said to be {\it weakly homotopically trivial} (W.H.T.) if the connecting map of the rational fibration $F_\Q \to X_\Q \to Y_\Q$ is the zero map; see  \cite[page 75]{Thomas1}. For example, a fibration with a section is so. We observe that a
fibration is W.H.T if and only if $\pi_*(X)\otimes \Q\cong\pi_*(F)\otimes \Q\oplus\pi_*(Y)\otimes \Q$.
%for the homotopy fibre $F$ of $f$;

\begin{prop} %\todo{Yam(1/21)``elliptic'' is removed.}\todo{Kuri(2/1): The assertion is revised a little.}
 For a W.H.T. fibration $f:X\to Y$, if $X$ is a formal space with $\dim \pi_*(X)\otimes \Q<\infty$, then
    $\Theta(f)$ is H-formal; see the paragraph before Proposition \ref{prop:WHT} for a W.H.T. fibration.

%(2)  For formal spaces $X$ and $Y$,
%$f : X\to Y$ is formalizable if and only if
%$d_{\text{IHC}}(\Theta (f), \Theta (H^*(f;\Q )))=0$.
\end{prop}

\begin{proof}
   From \cite[Theorem II]{FH}, the minimal model $\minmodel(X)$ for $X$ is a two stage model, i.e.,
   $\minmodel(X)=(\wedge V,d)$ for which $V=V_0\oplus V_1$, $dV_0=0$ and $\{ d(v_i) \}$ gives a regular sequence in $\wedge V_0$ for some generators $v_i$ of $V_1$. %\todo{Kuri (12/31): "a regular sequence" is added. (1/2) E.M. model will be applicable to prove Prop A.7.}
Then, any two sub-DGAs $\minmodel_a$ and $\minmodel_b$ of $\minmodel(X)$ represent formal spaces
and
there is a commutative diagram:
\[
\xymatrix@C30pt@R15pt{
\minmodel_a \ar[r]^{i} \ar[d]_{\eta_1}^\simeq & \minmodel_b \ar[d]^{\eta_2}_\simeq\\
(H^*(\minmodel_a),0) \ar[r]_{i^*} & (H^*(\minmodel_b),0)
}
\]
%(2) The relative model of the DGA-map $H^*(f;\Q):H^*(Y;\Q),0\to H^*(X;\Q),0$ is equivalent to one of $M(Y)\to \minmodel(X)$.
in which $i$ is the inclusion $\minmodel_a=\wedge V_a\subset \minmodel_b=\wedge V_b$.
Here, the quasi-isomorphism $\eta_i$ is given by
$H^*(\minmodel_i)=\wedge V_{i,0}/(dV_{i,1})$ for
$V_{i,0}\subset V_0$ and $V_{i,1}\subset V_1$ for $i=a$ and $b$.
Then, the inclusion $i$ gives rise to $\Theta(f)$ which is $H$-formal.
%We can put $\minmodel_i$ as $\theta (f)(n)$ for $n\geq 0$.
\end{proof}

\begin{rem}\label{rem:non-formalizable} %\todo{Kuri(2/1): This remark is revised. Kuri (2/23) We consider the lower square instead of the upper one.}
In view of Proposition \ref{ex:connectedsumCP2_distance} and Lemma \ref{lem:An_isometry}, one of the maps $f_1$ and $f_2$ in the proposition is not $d_{\rm{HI}}$-formal. Indeed, the both of maps
$f_k : S^3 \to {\mathbb C} P^2 \# \overline{{\mathbb C} P^2}$
are {\it not} formalizable. To see this, suppose that $f_k$ is formalizable for $k =1, 2$. Then, the homotopy commutativity of the lower square in Definition \ref{eq:formalizability} enables us to deduce that a Sullivan representative for $f_k$ is homotopic to the trivial map, which is a contradiction to the choice of the map.
\end{rem}
%%%%%%%%%%%%%%%%%%%%%%%%%%%%%%%%%%%

%\section*{Acknowledgments}

%\appendix

\section{More calculations of the distances $d_{\text{IHC}}$ of maps}\label{sect:Examples}
%\todo{Kuri (1/5): Some examples and results have been moved to this section.}
This section is devoted to gathering computational examples and upper bounds of the distance $d_{\rm IHC}(\Theta(f), \Theta(g))$ for
maps $f$ and $g$. In particular, the example of the distance for path fibrations is the first one that we have investigated
at the beginning of our work due to the result in Proposition \ref{prop:ID_equivalences} (ii).
%\todo{Kuri(3/25): The sentence is added. }

The following example deals with maps which are not relatively finite.

\begin{ex} %\todo{Kuri (1/11): The results are explained in detail.}
%Note an example of non finite type.
    Let $h : X=S^3\vee S^3\to S^3$ be the map that sends the second $S^3$ to the base point and $g:X=S^3\vee S^3\to *$ the constant map.
The minimal model of $X$ is given by
   $\minmodel(S^3\vee S^3)=(\wedge (x,y,z,u,w, ... ),d)$
   with $|x|=|y|=3$, $|z|=5$, $|u|=|w|=7$, $d(x)=d(y)=0$,
   $d(z)=xy$, $d(u)=xz$ and $d(w)=yz$.
   The minimal relative Sullivan model for $h$ is given by
   %Here $M(h)$ is given by
   the inclusion $(\wedge (x),0)\to (\wedge (x,y,z,u,w, ... ),d)$. Moreover, the inclusion
   $\iota _g :(\Q,0)\to (\wedge (x,y,z,u,w, ... ),d)$  is the minimal relative Sullivan model  for $g$.

   (i)(cf. Proposition \ref{prop:f_id_N_interleaved}) Then, it follows that $d_{\rm IHC}(\Theta (id_X), \Theta (g))= \infty$.  In fact, suppose that there exists
   a $\delta$-interleaving $(\varphi, \psi)$ between $\Theta(id_X)$ and $\Theta(g)$.  By applying the functor $H\circ Q$ to the homotopy relation
   $\psi^\delta\circ \varphi \simeq \Theta(id_X)(* < *+2\delta)$, we have $((H\circ Q)(\psi^\delta) \circ (H\circ Q)(\varphi))(i) = id_{(H\circ Q)\minmodel(X)}$ for each $i\geq 0$.
   The composite in the left-hand side factors through a finite dimensional vector space. The right-hand side is the identity map on the infinite dimensional vector space $(H\circ Q)(\minmodel(X))\cong Q\minmodel(X)$, which is a contradiction.

   (ii)  We have $d_{\rm IHC}(\Theta (h), \Theta (g) )= 3$.
   Indeed, for $\e \geq 3$, they are $\epsilon$-interleaved since the natural transformations  $\varphi$ and $\psi$
   are defined by identities.
   Also,  for $\e<3$,  they are not  homotopy $\varepsilon$-interleaved since
   $H(\psi (\e)\circ \varphi(0))(x)=0$.
   %On the other hand, if $f$ is the identity map $id: S^3\vee S^3\to S^3\vee S^3$, we have $d_{\rm IHC}(\Theta (f), \Theta (g) )= \infty$.
\end{ex}

We observe that the W.H.T. condition for a fibration $f : X\to Y$ is equivalent to that the minimal relative Sullivan model of $f$
gives a minimal  model for $X$.
Then, the minimal model  for $Y$ is regarded as a sub CDGA of that for $X$. %\todo{Kuri(2/1): We need a reference for the fact.}

\begin{prop}\label{prop:WHT} Let $f:X\to Y$ and $g:X\to Z$ be W.H.T. fibrations.
%Then,  for the Sullivan minimal model $\minmodel(X)=(\wedge V,d)$ for the space $X$,
     %$\minmodel(Y)=(\wedge V_1,d)$ and $\minmodel(Z)=(\wedge V_2,d)$,
   % are sub CDGAs of $\minmodel(X)$ with $V_i\subset V$.
    Let  $\minmodel(Y)=(\wedge V_1,d)$ and $\minmodel(Z)=(\wedge V_2,d)$ be minimal models for $Y$ and $Z$, respectively.
    Suppose that $\dim V_i<\infty$ for $i=1$ and $2$.
  Then, it holds that
    \begin{align}
    d_{\rm HC}(j_*(\Theta(f)), j_*(\Theta(g))) & \leq d_{\rm IHC}(\Theta (f), \Theta (g) ) \leq d_{\rm HI}(\Theta (f), \Theta (g) ) \\
    & \leq  d_{\rm I}(\Theta (f), \Theta (g) )\leq
    \max_{i \in \{1, 2\}}\{ \max\{n \mid V_i^n\neq 0\} \}.
    \end{align}
\end{prop}

Observe that the upper bound in the proposition above depends only on the minimal models for codomains of $f$ and $g$.
%\todo{Kuri (3/27): The sentence explains the significance of Proposition \ref{prop:WHT}.}

\begin{proof}[Proof of Proposition \ref{prop:WHT}]  We prove the last inequality. Other ones follow from Proposition \ref{prop:inequalities}.

Without loss of generality, we assume that inclusions
$\wedge V_1 \to  \wedge(V_1\oplus W_1)$ and $\wedge V_2 \to  \wedge(V_2\oplus W_2)$
for some $W_1$ and $W_2$
give rise to the persistence CDGAs $\Theta(f)$ and $\Theta(g)$, respectively. Let $\minmodel(X)=(\wedge V,d)$  be the minimal model for the space $X$. Then,
the W.H.T. condition and the uniqueness of the model enable us to obtain isomorphisms $\widetilde{f} : \wedge(V_1\oplus W_1) \stackrel{\cong}{\to} \wedge V$ and $\widetilde{g} : \wedge(V_2\oplus W_2) \stackrel{\cong}{\to} \wedge V$ of CDGAs.

Let $N$ be the right-hand side integer in the inequality that we show. %$\max \{ N_1,N_2\}$.
%For each $i\geq 0$,
We define $N$-interleavings $\varphi(i):\Theta(f)(i)\to \Theta(g)(i+N)$ and $\psi(i):\Theta(g)(i)\to \Theta(f)(i+N)$ by the
restriction of $\widetilde{g}^{-1}\circ \widetilde{f}$ and $\widetilde{f}^{-1}\circ \widetilde{g}$ to the domains, respectively. Since $V_i = V_i^{\leq N}$ for $i = 1$ and $2$, it follows that $(\varphi, \psi)$ is a well-defined $N$-interleaving.
 %\todo{Kuri (3/26): More details for the proof are added clarifying $\Theta(f)$ and $\Theta(g)$.}
%the inclusions of them into $M(X)$ for $i\geq0$.
\end{proof}

Principal $G$-bundles over a space $X$ are classified by the homotopy set $[X, BG]$ up to isomorphism,
where $BG$ denotes the classifying space of $G$. Thus, Theorem \ref{thm:Ho_sets} enables us to consider the pseudodistance between $G$-bundles. In the example below, we determine the distance $d$ between principal $S^1$-bundles over $\mathbb{C}P^n$.

\begin{ex} %\todo{Kuri: Please write the calculation, Naito. (2026.2.8) Naito added}
We identify the isomorphism class of an $S^1$-principal bundle over $\mathbb{C}P^n$ with the homotopy class of its classifying map $\xi_k : \mathbb{C}P^n \to BS^1$, which corresponds to an integer $k \in \Z \cong H^2(\mathbb{C}P^n; \mathbb{Z}) \cong [\mathbb{C}P^n, BS^1]$.
% We identify the isomorphism class of an $S^1$-principal bundle $\xi_k : X_k \to \mathbb{C}P^n$
% with the homotopy class of the classifying map $\xi_k \in [\mathbb{C}P^n, BS^1]\cong H^2(\mathbb{C}P^n; \mathbb{Z})=\mathbb{Z}$ of the bundle.
The pseudodistance $d(\xi_k, \xi_l)$ in Theorem \ref{thm:Ho_sets} is given by
\[
d(\xi_k, \xi_l)
=
\begin{cases}
0 & (k, l \neq 0 \quad \text{or} \quad k=l=0),\\
2 & (\text{otherwise}).
\end{cases}
\]

This result can be verified by a direct calculation using the Sullivan models as follows.
Let ${\minmodel}(BS^1)=(\wedge (u),0)$ and ${\minmodel}({\mathbb C}P^n)=(\wedge (x,y),d)$ be minimal Sullivan models for $BS^1$ and ${\mathbb C}P^n$, respectively.
Here,  $|u|=|x|=2$, $|y|=2n+1$, $dx=0$ and $dy = x^{n+1}$.
%The differential $d$ satisfies $dx=0$ and $dy = x^{n+1}$.
A Sullivan representative for $\xi_k$ is given by the morphism $\xi_k^* : {\minmodel}(BS^1) \to {\minmodel}({\mathbb C}P^n)$ defined by $\xi_k^* (u)=kx$.
Then, the relative Sullivan model for $\xi_k^*$ is constructed as follows.
\begin{itemize}
\item For $k\neq 0$, we have a relative Sullivan model $(\wedge (u),0) \to (\wedge (u,y_k),d)$ of $\xi_k^*$ via the quasi-isomorphism $\eta : (\wedge (u,y_k),d) \to {\minmodel}({\mathbb C}P^n)$ defined by $\eta (u)=kx$ and $\eta (y_k)=y$, where $d(y_k) = u^{n+1}/k^{n+1}$.
\item For $k=0$, a relative Sullivan model $(\wedge (u),0) \to (\wedge (u,z),d)\otimes (\wedge (x,y),d)$ of $\xi_0^*$ via the quasi-isomorphism $\eta : (\wedge (u,z),d)\otimes (\wedge (x,y),d)\to {\minmodel}({\mathbb C}P^n)$ given by the projection to ${\minmodel}({\mathbb C}P^n)$, where $dz=u$.
\end{itemize}
It is readily seen that the relative Sullivan models for $\xi_k^*$ and $\xi_l^*$ for any $k,l\neq 0$ are isomorphic.
This implies that $d(\xi_k, \xi_l) = 0$ immediately for any $k, l \neq 0$. The case $k = l = 0$ is trivial.

Moreover, $\Theta (\xi_k)$ and $\Theta (\xi_0)$ are $2$-interleaved in the homotopy category for any $k \neq 0$.
An interleaving $(\varphi,\psi)$ is given as follows.
\begin{itemize}
    \item $\varphi : \Theta(\xi_k) \Rightarrow \Theta (\xi_0)\circ T_2$ is a natural transformation defined by a restriction of the morphism $(\wedge (u,y_k),d) \to (\wedge (u,z),d)\otimes (\wedge (x,y),d)$, $u\mapsto x$, $y_k \mapsto y/k^{n+1}$.
    \item $\psi : \Theta(\xi_0) \Rightarrow \Theta (\xi_k)\circ T_2$  is a natural transformation defined by a restriction of the morphism $(\wedge (u,z),d)\otimes (\wedge (x,y),d) \to (\wedge (u,y_k),d)$, $u\mapsto 0$, $z\mapsto 0$, $x\mapsto u$, $y\mapsto k^{n+1}y_k$.
\end{itemize}

Similarly to Example \ref{ex:CP_S_distance}, it can be shown that $\Theta (\xi_k)$ and $\Theta(\xi_0)$ are not $\varepsilon$-interleaved in the homotopy category for $0\leq \varepsilon <2$.
This leads to the conclusion that $d(\xi_k, \xi_0) = 2$.
\end{ex}

In the rest of this section,  we consider extended tame persistence CDGAs $\Theta (p)$  for path fibrations $p$.

\begin{ex}\label{ex:p_q}
%We deal with simple cases of persistence CDGAs which come from continuous maps via the functor $\Theta$.
%\noindent
(i) Let $p : PS^{2n}\to S^{2n}$ be the path fibration, where $n\geq 1$.  The relative Sullivan model for the map $p$ is a relative Sullivan algebra of the form
\begin{align}
\xymatrix@C15pt@R20pt{
(\wedge(x, y), d)\  \ar@{>->}[r]& (\wedge(x, y, z, w), d) \simeq \Q
}
\end{align}
in which $d(y) =x^2$, $d(z)=x$ and $d(w) = xz -y$, where $\deg(x) = 2n$. Thus, we see that $\theta(\iota_p)$, but not $\Theta(p)$, has the form
\begin{align}
\xymatrix@C15pt@R0pt{
\hspace{-2.5cm} 0 &\hspace{-2.8cm}  2n-1 &\hspace{-1cm}  4n-2 \\
\wedge(x, y) =\cdots = \wedge(x, y) \  \ar@{>->}[r]& \wedge(x, y, z) = \cdots
=  \wedge(x, y, z) \ \ar@{>->}[r]& \wedge(x, y, z, w) = \cdots.
}
\end{align}
%Here, the upper numbers denote the persistence degrees in $\text{Func}((\mathbb{Z}_{\geq 0}, \geq), \mathsf{CDGA})$.
Moreover, the cohomology $H(\theta(\iota_p))$ of the persistence CDGA $\theta(\iota_p)$
has the form
\begin{align}
\xymatrix@C15pt@R0pt{
 \hspace{-2.5cm} 0 &\hspace{-2.8cm}  2n-1 &\hspace{-1cm}  4n-2 \\
\Q[x]/(x^2) =\cdots = \Q[x]/(x^2)  \ar[r]^-{t}& \wedge(xz-y) = \cdots
=   \wedge(xz-y) \ar[r]^-{t} & \Q = \cdots,
}
\end{align}
where $t$ denotes the trivial map. Thus, we have
\[
H(\Theta(p)) \cong [0, \infty)_0\oplus [0, 2n-1)_{2n} \oplus [2n-1, 4n-2)_{4n-1}
\]
as a persistence graded {\it module}.

(ii) Let $q : PS^{2n-1} \to S^{2n-1}$ be the path fibration, where $n\geq1$.  It follows that
$\theta(\iota_q)$ and $H(\theta(\iota_q))$ have the form
\begin{align}
\xymatrix@C15pt@R0pt{
\hspace{-2.5cm} 0 &\hspace{-1.5cm}  2n-2 &   \hspace{-2.5cm} 0 &\hspace{-0.8cm}  2n-2 \\
\wedge(u) =\cdots = \wedge(u) \  \ar@{>->}[r]& \wedge(u, v) = \cdots \  \text{and} & \wedge(u) =\cdots = \wedge(u) \ar[r]^-t &\Q = \cdots,
}
\end{align}
respectively. We see that $H(\Theta(q)) \cong [0, \infty)_0\oplus [0, 2n-2)_{2n-1}$ as a persistence graded module. Observe that the numbers of bars which appear in $H(\Theta(p))$ and $H(\Theta(q))$ are different from each other.
\end{ex}

We conclude this section by giving an upper bound of $d_{\rm IHC}$ for path fibrations.

\begin{prop} %\todo{Kuri (3/27): Check the assertion and its proof, Yama.}
  Let $\wedge V_X$ and $\wedge V_Y$ be the minimal models for simply-connected spaces $X$ and $Y$, respectively.
  Suppose that $\dim V_X <\infty$ and $\dim V_Y<\infty$.
    For the path fibrations $f:PX\to X$ and $g:PY\to Y$, one has
  \(d_{\rm IHC}(\Theta(f), \Theta(g)) \leq  (\max\{ \max\{ i \mid V_X^i\neq 0\}-1, \max\{ i \mid V_Y^i\neq 0\} -1\})/2\)=:N.
\end{prop}

\begin{proof}
 %Let $N=\max\{ N_X-1,N_Y-1\}/2$.
 It follows from \cite[Theorem 15.3]{FHT}, which gives a model for a fibration, that
 $\minmodel(PX)=\wedge V_X\otimes \wedge\overline{V}_X$ and $\minmodel(PY)=\wedge V_Y\otimes \wedge \overline{V}_Y$.
%For $\e \geq N$,
Define morphisms
$\varphi(i):\Theta(f)(i)\to \Theta (g)(i+N)$ and $\psi (i) :\Theta(g)(i)\to \Theta (f)(i+N)$ by the trivial maps, respectively.
Then, we see that $\Theta(f)$ and $\Theta(g)$ are $N$-homotopy interleaved in the homotopy category with
the interleaving $(\varphi, \psi)$.

In fact, since the path space $PY$ is contractible, it follows that the unit $\eta : \Q \to \minmodel(PY)$ is a quasi-isomorphism.
Then, Whitehead theorem (\cite[Chapter II, Theorem 1.10]{GJ}) enables us to conclude that $\eta$ is a homotopy equivalence
with the homotopy inverse $\gamma$.
We observe that the domain and codomain of $\eta$ are cofibrant-fibrant objects in $\mathsf{CDGA}$.
Let $H$ be a homotopy from $\eta\circ \gamma$ to $id_{\minmodel(PY)}$.
It follows from the choice of the number $N$ that  $\overline{V}_Y^{\leq 2N}= \overline{V}_Y$. %$\bar{V}_X^{\leq 2N}=\overline{V}_X$.
Then, Remark \ref{rem:Homotopies} allows us to obtain a homotopy $\varphi^{N}\circ \psi \simeq \Theta(g)(*<*+{2N})$ in $\et{\mathsf{CDGA}}$
with the restriction of  the homotopy $H$.
By the same argument as above, we see that
$\psi^{N}\circ\varphi\simeq \Theta(f)(*<*+{2N})$ in $\et{\mathsf{CDGA}}$.
\end{proof}

\section{Perspective}\label{sect:Perspective}
As seen in the Introduction, the result \cite[Theorem 3.3]{KNWY}  asserts that the interleaving distances and the cohomology interleaving distance coincide for persistence cochain complexes.
The starting point of the manuscript has been to find algebraic or geometric examples which show the difference between the cohomology interleaving distance and
the interleaving distance in the homotopy category. To this end, we have prepared the categorical framework in Sections \ref{sect:pCDGAs} and \ref{sect:MPs-P_CDGAs}.
This section describes future perspective based on the framework. %\todo{Kuri (3/27) (3/28): Section `Perspective' is added.}

\subsection{Extended tame persistence differential graded Lie algebras}\label{sect:dgl} As seen in \cite[Theorem 6.9]{M-S}, we may replace the category $\mathsf{CDGA}$ in Theorem \ref{thm:main_II} with the category
$\mathsf{cdgl}$ of complete differential graded Lie algebras (cdgls), which is deeply considered in \cite{BFMT}.  In fact, the category
$\mathsf{cdgl}$ is endowed with a cofibrantly generated model structure; see \cite[Chapter 8]{BFMT}. Thus, by combining
Proposition \ref{prop:pQe}, \cite[Corollary 8.2]{BFMT} and \cite[Theorem 2.10]{FFM},
we have a partial Quillen equivalence between $\mathsf{et(sSet^{\R_+})}$ and $\mathsf{et(cdgl^{\R_+})}$ on reduced nilpotent rational simplicial sets and homologically nilpotent connected cdgls. We may investigate extended tame persistence simplicial sets
as well as towers, which are mentioned in Definition \ref{defn_Postnikov} and Remark \ref{rem:CDGA_sSet},
with their Lie models from the homotopical point of view.

It is worthwhile mentioning that in the consideration of Lie models, the finiteness condition for persistence objects assumed in Theorem \ref{thm:main_II} is not needed.

\subsection{(Co)fibration categories}\label{sect:cof} Let $\mathcal{C}$ be a cofibration category in the sense of Baues \cite{Baues}.
Then, following the proof of \cite[Theorem 2.2]{CGL},
we may give a (non trivial) cofibration model category structure to $\mathsf{et}(\mathcal{C^{\R_+}})$. In fact, for example, by considering the terminal object virtually, the existence of a fibrant model follows from the same argument as that in the proofs of the factorization axiom and the left lifting property in the model category. Such a cofibration category structure on $\mathsf{et}(\mathcal{C^{\R_+}})$ will be considered in \cite{Sekizuka}. Then, in view of the investigation in \cite{M-M}, we may develop the homotopy theory on persistence spaces in {\it coarse geometry}.

There are some versions of a cofibration (fibration) category. Therefore, it will be important work to consider the expansion
of the (co)fibration category structure to the category of extended tame functors for each version.
The (co)fibration category structures in \cite{CDKOSW} and \cite{A-G} are applicable
to investigating {\it persistence digraphs} and {\it persistence} $C^*$-{\it algebras}.

\subsection{Multiparameter persistence objects}\label{sect:multi-p}
%\begin{rem}\label{rem:n_to_n+1}
One might expect developments of persistence CDGAs and related objects in multiparameter persistence theory.
As a consequence of Proposition \ref{prop:0_to_1}, one has
a commutative diagram
\begin{eqnarray}\label{eq:1_to_2}
\xymatrix@C30pt@R20pt{
(\text{Func}(I, \mathsf{Top}_0))^{\rm op})^{({\mathbb R}^n, \leq)} \ar[r]^-{\Theta_*}  &  (\text{Ho}(\et{\mathsf{CDGA}}))^{({\mathbb R}^n, \leq)}
\ar[d]^{({\bf L}(\text{colim}))_*} \\
( \mathsf{Top}_0^{\text{op}})^{({\mathbb R}^n, \leq)} \ar[r]_-{(q\circ A_{PL}(\ ))_*} \ar[u]^{i_*} & (\text{Ho}(\mathsf{CDGA}))^{({\mathbb R}^n, \leq)}.
}
\end{eqnarray}
Thus, we obtain naturally $(n+1)$-parameter persistence objects from $n$-parameter persistence objects via $\Theta_*$.
In \cite{Zhou}, Zhou considers the interleaving distance and related distances between persistence spaces with the functor $(q\circ A_{PL}(\ ))_*$ in the case where $n=1$.
%\end{rem}

The E.M. model mentioned in Section \ref{section:formalities} has the lower degrees adding to the dimensions of the generators. It seems that we have two parameter persistence object modifying $\Theta(f)$
with the second degree.

Moreover, we might be able to deal with the objects described in Sections \ref{sect:dgl} and \ref{sect:cof} as multiparameter persistence ones improving the diagram
(\ref{eq:equivalence}).

\medskip
\noindent
{\it Acknowledgements.} The authors thank Jos\'e Manuel Moreno Fern\'andez and Bruno Stonek for comments 
on \cite[Definition 2.4]{M-S} which make the argument in Section \ref{sect:pQe} more certain. 
The first author was partially supported by JSPS KAKENHI Grant Numbers JP23K20795, JP26K00604.
The second author was partially supported by JSPS KAKENHI Grant Numbers JP23K03097, JP26K06813.
The fourth author was partially supported by JSPS KAKENHI Grant Number JP25K17250.

\end{document}